\documentclass[12pt]{amsart}

\usepackage{amssymb}
\usepackage{amsmath}
\usepackage{amscd}
\input amssym.def
\usepackage{pstricks}
\usepackage[all]{xy}
\usepackage{hyperref}
\newtheorem{theorem}{Theorem}[section]
\newtheorem{lemma}[theorem]{Lemma}
\newtheorem{prop}[theorem]{Proposition}
\newtheorem{cor}[theorem]{Corollary}
\theoremstyle{definition}
\newtheorem{rem}[theorem]{Remark}
\newcommand{\M}{\mathcal{M}}
\newcommand{\PM}{\mathcal{PM}}
\newcommand{\Z}{\mathbb{Z}}
\newcommand{\Dif}{\mathcal{H}}
\newcommand{\bdr}[1]{\partial\! #1}
\newcommand{\C}{\mathcal{C}}
\newcommand{\D}{\mathcal{D}}
\newcommand{\G}{\mathcal{G}}
\newcommand{\cS}{\mathcal{S}}
\newcommand{\lr}[1]{\left<#1\right>}
\newcommand{\Stab}{\mathrm{Stab}}
\newcommand{\Sym}{\mathrm{Sym}}
\numberwithin{equation}{section}
\title[Presentation for the mapping class group]
{A presentation for the mapping class group of a
nonorientable surface}
\author{Luis Paris}
\address[Luis Paris]{Universit\'e de Bourgogne, Institut de Math\'ematiques de Bourgogne, UMR 5584 du CNRS, B.P. 47870, 21078 Dijon cedex, France.} \email{lparis@u-bourgogne.fr}
\author{B{\l}a\.zej Szepietowski}
\address[B{\l}a\.zej Szepietowski]{Institute of Mathematics, Gda\'nsk University, Wita Stwosza 57,
80-952 Gda\'nsk, Poland} \email{blaszep@mat.ug.edu.pl}
\begin{document}
\begin{abstract}
Let $N_{g,n}$ denote the nonorientable surface of genus $g$ with $n$ boundary components and $\M(N_{g,n})$ its mapping class group. We obtain an explicit finite presentation of  $\M(N_{g,n})$ for $n\in\{0,1\}$ and all $g$ such that $g+n>3$.
\end{abstract}
\maketitle
\section{Introduction}
Let $F$ be a compact surface with (possibly empty) boundary and let $\mathcal{P}=\mathcal{P}_m=\{P_1,\dots,P_m\}$ be a set of $m$ distinguished points in the interior of $F$, called {\it punctures}. We define $\Dif(F,\mathcal{P})$ to be the group of all, orientation preserving if $F$ is orientable, homeomorphisms
$h\colon F\to F$ such that $h(\mathcal{P})=\mathcal{P}$ and $h$ is equal to the identity on the boundary of $F$. The  {\it mapping class group} $\M(F,\mathcal{P})$ of $F$ relatively to $\mathcal{P}$ is the group of isotopy classes of elements of $\Dif(F,\mathcal{P})$. The {\it pure mapping class group} $\PM(F,\mathcal{P})$ is the subgroup of $\M(F,\mathcal{P})$ consisting of the isotopy classes of homeomorphisms fixing each puncture. If $\mathcal{P}=\emptyset$ then we drop it in the notation and write simply $\M(F)$. If $\mathcal{P}=\{P\}$ then we write $\M(F,P)$ instead of $\M(F,\{P\})$.
A compact surface of genus $g$ with $n$ boundary components will be denoted by $S_{g,n}$ if it is orientable, or by $N_{g,n}$ if it is nonorientable.

Historically, McCool \cite{McC} gave the first algorithm for finding a finite presentation for $\M(S_{g,1})$ for any $g$. His approach is purely algebraic and no explicit presentation has been derived from this algorithm. In their groundbreaking paper \cite{HT} Hatcher and Thurston gave an algorithm for computing a finite presentation for $\M(S_{g,1})$ from its action on a simply connected simplicial complex, the {\it cut system complex}. By this algorithm, Harer \cite{Harer} obtained a finite, but very unwieldy, presentation for  $\M(S_{g,1})$ for any $g$. This presentation was simplified by Wajnryb \cite{W,W1}, who also gave a presentation for $\M(S_{g,0})$. Using Wajnryb's result, Matsumoto \cite{Mat} found other presentations for $\M(S_{g,1})$ and $\M(S_{g,0})$, and Gervais \cite{Gerv} found a presentation for  $\M(S_{g,n})$ for arbitrary $g\ge 1$ and $n$. Starting from Matsumoto's presentations, Labru\`ere and Paris \cite{LabPar} computed a presentation for $\M(S_{g,n},\mathcal{P}_m)$ for arbitrary $g\ge 1$, $n$ and $m$. Benvenuti \cite{Benv} and Hirose \cite{Hir} independently recovered the Gervais presentation from the action of $\M(S_{g,n})$ on two different variations of the Harvey's  curve complex \cite{Harvey}, instead of the cut system complex. 

Until present, finite presentations of $\M(N_{g,n},\mathcal{P}_m)$ were know only for a few small values of $(g,n,m)$, with $g\le 4$. Using results of Lickorish \cite{Lick1,Lick2}, Chillingworth \cite{Chill} found a finite generating set for $\M(N_{g,0})$ for arbitrary $g$. This set was extended for $m>0$ by Korkmaz \cite{K}, and for $n+m>0$ and $g\ge 3$ by Stukow \cite{Stu_bdr}. For every nonorientable surface $N_{g,n}$ there is a covering
$p\colon S_{g-1,2n}\to N_{g,n}$ of degree two. By a result of Birman and Chillingworth \cite{BC}, generalised for $n>0$ in \cite{SzepB}, $\M(N_{g,n})$ is isomorphic to the subgroup of $\M(S_{g-1,2n})$ consisting of elements commuting with the covering involution. However, since the image of $\M(N_{g,n})$ has infinite index in $\M(S_{g-1,2n})$, it seems that it would be very hard to obtain a finite presentation for $\M(N_{g,n})$ from a presentation of $\M(S_{g-1,2n})$.
In \cite{Szep_Osaka} an algorithm for finding a finite presentation for $\M(N_{g,n})$ for any $g$ and $n$ is given, based on a result of Brown \cite{Br} and the action of $\M(N_{g,n})$ on the curve complex (following the idea of \cite{Benv}). By this algorithm, an explicit finite presentation for $\M(N_{4,0})$ was obtained in \cite{Szep1}.

In this paper we apply the algorithm given in \cite{Szep_Osaka} to find an explicit finite presentation for $\M(N_{g,n})$ for $n\in\{0,1\}$ and all $g$ such that $g+n>3$. We present $\M(N_{g,1})$ as a quotient of the free product
$\M(S_{\rho,r})\ast\M(S_{0,1},\mathcal{P}_g)$, where $g=2\rho+r$ and $r\in\{1,2\}$. The factor  $\M(S_{\rho,r})$ comes from an embedding of $S_{\rho,r}$ in $N_{g,1}$ and it is generated by Dehn twists. The factor $\M(S_{0,1},\mathcal{P}_g)$, which is isomorphic to the braid group, comes from the embedding $\M(S_{0,1},\mathcal{P}_g)\to\M(N_{g,1})$ defined in \cite{SzepB}, and it is generated by $g-1$ crosscap transpositions. There are three families of defining relations of $\M(N_{g,1})$: (A) relations from $\M(S_{\rho,r})$ between Dehn twists, (B) braid relations between crosscap transpositions, and (C) relations involving generators of both types. A presentation for $\M(N_{g,0})$ is obtained from that of $\M(N_{g,1})$ by adding three relations.

The presentations for $\M(N_{g,1})$ and $\M(N_{g,0})$ are given respectively in Theorems \ref{mainA} and \ref{mainB} in Section \ref{sec_pres}. They are proved simultaneously  by induction on $g$. The base cases $(g,n)\in\{(3,1), (4,0)\}$ are proved in Section \ref{sec_base}. Theorem \ref{mainA} is proved in Section \ref{sec_mainA} under the assumption that Theorem \ref{mainB} is true. The proof of Theorem \ref{mainB} uses the action of $\M(N_{g,0})$ on the ordered complex of curves defined in Section \ref{sec_curves}, and it
occupies Sections \ref{sec_v2}, \ref{sec_v13}, where presentations of stabilisers of vertices are calculated, and Sections \ref{sec_edges}, \ref{sec_triangles}, where we deal with relations corresponding to simplices of dimensions 1 and 2.
 
\subsection*{Acknowledgements}
First version of this paper was written during the second author's visit to Institut de Math\'ematiques de Bourgogne in Dijon in the period 01.10.2011 -- 30.09.2012  supported by 
the MNiSW ``Mobility Plus'' Program 639/MOB/2011/0. He wishes to thank the Institut for their hospitality. The second author was also partially supported by the MNiSW grant N~N201 366436. 
\section{Preliminaries}\label{sec_preli}
\subsection{Simple closed curves and Dehn twists.}
By a {\it simple closed curve} in $F$ we mean an embedding
$\gamma\colon S^1\to F\backslash\bdr{F}$. Note that $\gamma$ has an orientation; the
curve with the opposite orientation but same image will be denoted by
$\gamma^{-1}$. 
By abuse of notation, we will often identify a simple closed curve with its
oriented image and also with its isotopy class.
We say that $\gamma$ is {\it generic} if it 
does not bound a disc nor a M\"obius band and is not isotopic to a boundary component.
According to whether a regular neighbourhood of $\gamma$ is an annulus or a M\"obius strip, we call $\gamma$ respectively {\it two-} or {\it one-sided}. 
We say that $\gamma$ is {\it nonseparating} if $F\backslash\gamma$
is connected and {\it separating} otherwise. 

Given a two-sided simple closed curve $\gamma$, $T_\gamma$ denotes a Dehn
twist about $\gamma$. On a  nonorientable surface it is
impossible to distinguish between right and left twists, so the
direction of a twist $T_\gamma$ has to be specified for each curve
$\gamma$. In this paper it is usually indicated by arrows in a figure. Equivalently we may choose an orientation of a regular
neighbourhood of $\gamma$. Then $T_\gamma$ denotes the right Dehn twist with
respect to the chosen orientation. Recall that $T_\gamma$ does not depend on the orientation of $\gamma$. 

\begin{figure}
\input{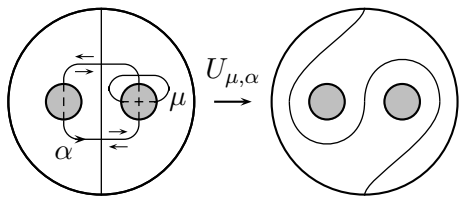}
\caption{\label{U} Crosscap transposition.}
\end{figure}

\begin{figure}
\input{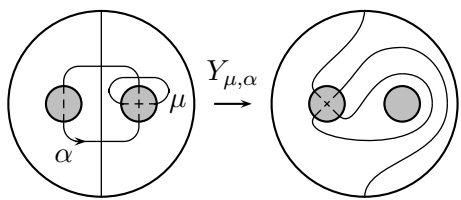}
\caption{\label{Y} Crosscap slide.}
\end{figure}

\subsection{Crosscap slides and transpositions.}
We begin this subsection by describing a convention used in all figures in this paper. We explain this on the example of Figure \ref{U}. The shaded discs
represent crosscaps; this means that their interiors should be
removed, and then antipodal points in each resulting boundary
component should be identified. The small arrows on two sides of the curve   $\alpha$ indicate the direction of the Dehn twist $T_\alpha$.

Let $N=N_{g,n}$ be a nonorientable surface of genus $g\ge 2$.
Suppose that $\mu$ and $\alpha$ are two simple closed curves in $N$,
such that $\mu$ is one-sided, $\alpha$ is two-sided and they
intersect in one point. Let $K\subset N$ be a regular neighbourhood of
$\mu\cup\alpha$, 
which is homeomorphic to the Klein bottle with a
hole. On Figure \ref{U} a homeomorphism of $K$ is shown, which interchanges the two crosscaps keeping the boundary of $K$ fixed. It may be extended 
by the identity outside $K$ to a homeomorphism of $N$, which
we call {\it crosscap
transposition} and denote as $U_{\mu,\alpha}$.
We define {\it crosscap slide} $Y_{\mu,\alpha}$ to be the composition
\[Y_{\mu,\alpha}=T_\alpha U_{\mu,\alpha},\]
where $T_\alpha$ is the Dehn twist about $\alpha$ in the direction indicated by the arrows in Figure \ref{U}. If $M\subset K$ is a regular neighbourhood of $\mu$, which is a M\"obius strip, then $Y_{\mu,\alpha}$ may be described as the effect of pushing $M$ once along $\alpha$ (Figure \ref{Y}).
Observe that $Y_{\mu,\alpha}$ reverses the orientation of $\mu$. 
Up to isotopy, $Y_{\mu,\alpha}$ does not depend on the
choice of the regular neighbourhood $K$. It also does not depend on the orientation of $\mu$
but does depend on the orientation of $\alpha$, as 
$Y_{\mu,\alpha^{-1}}=Y^{-1}_{\mu,\alpha}$. For any $h\in\M(N)$ we have the formula
\[hY_{\mu,\alpha}h^{-1}=Y_{h(\mu),h(\alpha)}.\]
The crosscap slide was introduced under the name Y-homeomorphism by Lickorish, who proved that $\M(N_{g,0})$ is generated by Dehn twists and one crosscap slide for $g\ge 2$ \cite{Lick1,Lick2}. 
\subsection{Exact sequences.}
Given an exact sequence of groups
\[1\to K\to G\stackrel{\rho}{\to}H\to 1\]
and presentations $K=\lr{S_K\,|\,R_K}$ and $H=\lr{S_H\,|\,R_H}$, a presentation for $G$ may be obtained as follows. For each $x\in S_H$ we choose $\widetilde{x}\in G$ such that $\rho(\widetilde{x})=x$ and let 
\[\widetilde{S_H}=\{\widetilde{x}\,|\,x\in S_H\}.\] For each $r=x_1^{\epsilon_1}\cdots x_k^{\epsilon_1}\in R_H$ let 
$\widetilde{r}=\widetilde{x_1}^{\epsilon_1}\cdots\widetilde{x_k}^{\epsilon_1}$. Since $\rho(\widetilde{r})=1$, there is a word $w_r$ over $S_K$ representing the same element of $G$ as $\widetilde{r}$. Let \[R_1=\{\widetilde{r}w_r^{-1}\,|\,r\in R_H\}.\] 
Since $K$ is a normal subgroup of $G$, for $x\in S_H$ and $y\in S_K$ we have $\widetilde{x}y\widetilde{x}^{-1}\in K$ and there is a word $w(x,y)$ over $S_K$ representing the same element of $G$ as  $\widetilde{x}y\widetilde{x}^{-1}$. Let
\[R_2=\{\widetilde{x}y\widetilde{x}^{-1}w(x,y)^{-1}\,|\,x\in S_H, y\in S_K\}.\]
Proof of the following lemma is left to the reader.
\begin{lemma}\label{ext_pres}
$G$ admits the presentation
\[G=\lr{S_K\cup \widetilde{S_H}\,|\, R_K\cup R_1\cup R_2}.\]
\end{lemma}
The generators $S_K$ and $\widetilde{S_H}$ will be called
{\it kernel} and {\it cokernel} generators respectively. The relators
$R_K$, $R_1$ and $R_2$ will be called {\it kernel, cokernel} and {\it conjugation} relators respectively. In this paper we work with relations rather then relators.

\medskip

The inclusion $\mathcal{P}_{m-1}\subset\mathcal{P}_m$ gives rise to a {\it forgetful homomorphism}
$\mathfrak{f}\colon\PM(F,\mathcal{P}_m)\to\PM(F,\mathcal{P}_{m-1})$. By \cite{Bir1}, if the Euler characteristic of $F\backslash\mathcal{P}_{m-1}$ is negative, then we have the following {\it Birman exact sequence}.
\begin{equation}\label{Bir_es}
1\to\pi_1(F\backslash\mathcal{P}_{m-1},P_m)\stackrel{\mathfrak{p}}{\to}\PM(F,\mathcal{P}_m)
\stackrel{\mathfrak{f}}{\to}\PM(F,\mathcal{P}_{m-1})\to 1.
\end{equation}
Although the above result is proved in \cite{Bir1} for orientable $F$, the same proof works for nonorientable $F$ as well.
The homomorphism $\mathfrak{p}\colon\pi_1(F\backslash\mathcal{P}_{m-1},P_m)\to\PM(F,\mathcal{P}_m)$ is called the {\it point pushing map}. Suppose that $\gamma$ is a simple loop on $F\backslash\mathcal{P}_{m-1}$ based at $P_m$ and let $A$ be its regular neighbourhood. If $[\gamma]$ denotes the homotopy class of $\gamma$, then
$\mathfrak{p}[\gamma]$ is isotopic to a homeomorphism equal to the identity outside $A$, and obtained by pushing $P$ once along $\gamma$ keeping the boundary of $A$ fixed, see Figure \ref{sl}. Note that $A$ is a M\"obius band if $\gamma$ is one-sided, or an annulus if $\gamma$ is two-sided. In the latter case $\mathfrak{p}[\gamma]$ may be expressed in terms of Dehn twists about the boundary components of $A$. 
\begin{figure}
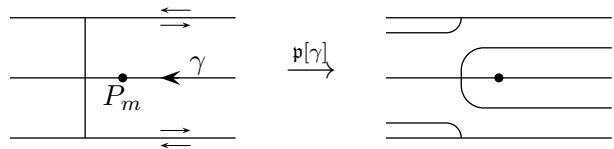

\pspicture*(8.5,2)
%
\psline[linewidth=.5pt](.25,.2)(3.25,.2)
\rput[b](2.75,1.05){\small$\gamma$}
\psline[linewidth=.5pt](.25,1)(3.25,1)
\psline[linewidth=.5pt](.25,1.8)(3.25,1.8)
\pscircle*(1.75,1){.06}\rput[t](1.75,.9){\small$P_m$}
\psline[linewidth=.5pt, arrowsize=4pt 3]{<-}(2.25,1)(2.5,1)
\psline[linewidth=.5pt](1.25,.2)(1.25,1.8)
\psline[linewidth=.4pt]{->}(2.25,.3)(2.65,.3)
\psline[linewidth=.4pt]{<-}(2.25,.1)(2.65,.1)
\psline[linewidth=.4pt]{<-}(2.25,1.9)(2.65,1.9)
\psline[linewidth=.4pt]{->}(2.25,1.7)(2.65,1.7)
\rput[b](4.25,1){$\stackrel{\mathfrak{p}[\gamma]}{\longrightarrow}$}
\psline[linewidth=.5pt](5.25,.2)(8.25,.2)
\psline[linewidth=.5pt](5.25,1)(8.25,1)
\psline[linewidth=.5pt](5.25,1.8)(8.25,1.8)
\pscircle*(6.75,1){.06}
%
\psline[linewidth=.5pt, linearc=.2](6.25,1.8)(6.25,1.6)(5.25,1.6)
\psline[linewidth=.5pt, linearc=.2](6.25,.2)(6.25,.4)(5.25,.4)
\psline[linewidth=.5pt, linearc=.3](8.25,1.4)(6.25,1.4)(6.25,.6)(8.25,.6)
%
\endpspicture
\caption{\label{sl}Pushing a puncture along a simple loop.}
\end{figure}
\begin{lemma}\label{push1} Suppose that $\gamma$ is a two-sided simple loop based at $P_m$ and $\delta_1$, $\delta_2$ are the boundary components of a regular neighbourhood of $\gamma$. Then $\mathfrak{p}[\gamma]=T_{\delta_1}T_{\delta_2}$,
where the directions of the twists
are determined by the orientation of $\gamma$ as indicated by arrows on the left hand side of Figure \ref{sl}.\hfill{$\Box$}
\end{lemma}
The group $\PM(F,\mathcal{P}_m)$ acts on $\pi_1(F\backslash\mathcal{P}_{m-1},P_m)$ in the obvious
way. The next lemma says that $\mathfrak{p}$ is $\PM(F,\mathcal{P}_m)$-equivariant.
\begin{lemma}\label{push2}
For $h\in\PM(F,\mathcal{P}_m)$ and $[\gamma]\in\pi_1(F\backslash\mathcal{P}_{m-1},P_m)$ we have
$\mathfrak{p}(h[\gamma])=h\mathfrak{p}[\gamma]h^{-1}$.\hfill{$\Box$}
\end{lemma}
Suppose that $N$ is a nonorientable surface. We define $\PM^+(N,\mathcal{P}_m)$ to be the subgroup of $\PM(N,\mathcal{P}_m)$ consisting of the isotopy classes of homeomorphisms preserving local orientation at each puncture. Observe that 
it is a normal subgroup of index $2^m$.
For $1\le m\le n$ choose $m$ boundary components $\gamma_1,\dots,\gamma_m$ of $N_{g,n}$. Consider the surface $N_{g,n-m}$ as being obtained from $N_{g,n}$ by gluing a disc with a puncture $P_i$ in its interior along $\gamma_i$ for $i=1,\dots,m$. Let $\mathcal{P}_m=\{P_1,\dots,P_m\}$. Since every homeomorphism in $\Dif(N_{g,n})$ may be extended by the identity on the discs to an element of $\Dif(N_{g,n-m},\mathcal{P}_m)$, thus the inclusion $\imath\colon N_{g,n}\to N_{g,n-m}$ induces a homomorphism \[\imath_\ast\colon\M(N_{g,n})\to\PM^+(N_{g,n-m},\mathcal{P}_m).\] It is clearly surjective, and if $(g,n)\neq(1,1)$ then its kernel is the free abelian group of rank $m$ generated by the Dehn twists $T_{\gamma_i}$ for $i=1,\dots,m$ (see \cite[Theorem 3.6]{Stu_geom}).
Summarising, we have the following exact sequence.
\begin{equation}\label{Cup_es}
1\to\Z^m\to\M(N_{g,n})\stackrel{\imath_\ast}{\to}\PM^+(N_{g,n-m},\mathcal{P}_m)\to 1.
\end{equation}
\subsection{Blow up homomorphism and crosscap pushing map.}\label{blowup_cpush} 
In this subsection we recall from \cite{Szep2} the definitions of blowup homomorphism and  crosscap pushing map which will be important tools in what follows.

Let $F$ be a surface with $m\ge 1$ punctures $\mathcal{P}_m=\{P_1,\dots,P_m\}$.
Let $U=\{z\in\mathbb{C}\,|\,|z|\le 1\}$ and for $i=1,\dots,m$ fix an embedding
$e_i\colon U\to F\backslash\bdr{F}$ such that $e_i(0)=P_i$.
Let $\widetilde{F}$ be the nonorientable surface obtained by removing from $F$ the interiors of $e_i(U)$ and then identifying $e_i(z)$ with $e_i(-z)$ for $z\in S^1=\bdr{U}$ and $i=1,\dots,m$. Thus $\widetilde{F}=N_{g+m,n}$ if $F=N_{g,n}$ or $\widetilde{F}=N_{2g+m,n}$ if
$F=S_{g,n}$.

We define a {\it blowup homomorphism} 
\[\mathfrak{b}\colon\M(F,\mathcal{P}_m)\to\M(\widetilde{F})\] as follows. 
Represent $h\in\M(F,\mathcal{P}_m)$ by a homeomorphism $h\colon F\to F$ such that for some permutation $\sigma\in\Sym_m$ we have $h(e_i(z))=e_{\sigma(i)}(z)$ or $h(e_i(z))=e_{\sigma(i)}(\overline{z})$ for $z\in U$ and $i=1,\dots,m$. Such $h$ commutes with the identification leading to $\widetilde{F}$ and thus induces an element $\mathfrak{b}(h)\in\M(\widetilde{F})$. We refer the reader to \cite{Szep2} for a proof that $\mathfrak{b}$ is well defined
(the proof in \cite{Szep2} is only for $m=1$ but it can be easily modified to work for $m>1$). The next proposition is proved in \cite{SzepB} for $F=S_{0,1}$ but the same proof works for any $F$.
\begin{prop}\label{blowup_inj}
The blowup homomorphism $\mathfrak{b}\colon\M(F,\mathcal{P}_m)\to\M(\widetilde{F})$ is injective for any surface $F$.\hfill{$\Box$}
\end{prop}
We define the {\it crosscap pushing map} 
\[\mathfrak{c}\colon\pi_1(F\backslash\mathcal{P}_{m-1},P_m)\to\M(\widetilde{F})\]
as the composition $\mathfrak{c}=\mathfrak{b}\circ\mathfrak{p}$, where $\mathfrak{p}$ is the point pushing map from the Birman exact sequence (\ref{Bir_es}). If $\gamma$ is a simple loop on $F\backslash\mathcal{P}_{m-1}$ based at $P_m$, then it follows immediately from the description of $\mathfrak{p}[\gamma]$, that $\mathfrak{c}[\gamma]$ is either a crosscap slide if $\gamma$ is one-sided, or a product of two Dehn twists about the boundary components of a M\"obius band with a hole if $\gamma$ is two-sided (just replace the puncture with a crosscap on Figure \ref{sl}).
\subsection{Notation.}\label{notation}
Let us  represent $N_{g,0}$ and $N_{g,1}$ as respectively a sphere or a disc  with $g$ crosscaps. This means that interiors of $g$ small pairwise disjoint discs should be removed from the sphere/disc, and then antipodal points in each of the resulting boundary components should be identified. Let us arrange the crosscaps as shown on Figure \ref{aI} and number them from $1$ to $g$. 
For each nonempty subset $I\subseteq\{1,\dots,g\}$ let $\gamma_I$ be the simple closed curve shown on Figure \ref{aI}.  Note that $\gamma_I$ is two-sided if and only if $I$ has even number of elements. In such case $T_{\gamma_I}$ will be the Dehn twist about $\gamma_I$ in the direction indicated by arrows on Figure \ref{aI}. 
\begin{figure}
\input{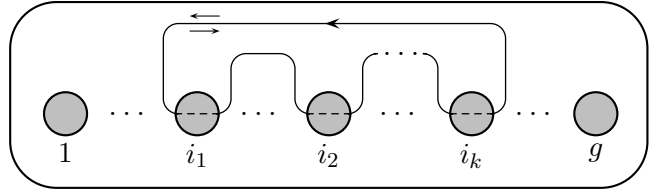}
\caption{\label{aI} The curve $\gamma_I$ for $I=\{i_1,i_2,\dots,i_k\}$.}
\end{figure}
The following curves will play a special role and so we give them different names.
\begin{itemize}
\item $\mu_i=\gamma_{\{i\}}$ for $i=1,\dots,g$
\item $\alpha_i=\gamma_{\{i,i+1\}}$ for $i=1,\dots,g-1$
\item $\beta=\gamma_{\{1,2,3,4\}}$
\item $\beta_j=\gamma_{\{1,\dots,2j+2\}}$ for $2\le 2j\le g-2$
\item $\xi=\gamma_{\{1,\dots,g\}}$
\end{itemize}
Note that $\beta=\beta_1$ and if $g=2\rho+2$ then $\xi=\beta_\rho$. We also give names to elements of $\M(N_{g,n})$ associated with these curves.
\begin{itemize}
\item $a_i=T_{\alpha_i}$, 
\item $y_i=Y_{\mu_{i+1},\alpha_i}$, 
\item $u_i=U_{\mu_{i+1},\alpha_i}$ for $i=1,\dots,g-1$,
\item $b=T_\beta$,
\item $b_j=T_{\beta_j}$ for $2\le 2j\le g-2$
\item $v=Y_{\mu_4,\beta}$, 
\item $c=T_{\gamma_{\{3,4,5,6\}}}$
\item $r_g=a_1\cdots a_{g-1}u_{g-1}\cdots u_1$
\end{itemize} 
If the surface is closed ($n=0$) then $r_g$ is isotopic to the homeomorphism induced by the reflection of Figure \ref{aI} across the line containing centers of the shaded discs (see \cite[Remark 2.4]{Szep1}).
%
\section{Presentations}\label{sec_pres}
The groups $\M(N_{1,0})$ and $\M(N_{1,1})$ are trivial by \cite[Theorem 3.4]{E}. The following presentations 
were obtained  in \cite{Lick1,Stu_Fund,BC} respectively.
\begin{align*}
\M(N_{2,0})=&\langle a_1,y_1\,|\, a_1^2=y_1^2=(a_1y_1)^2=1\rangle\\
\M(N_{2,1})=&\langle a_1,y_1\,|\, a_1y_1a_1=y_1\rangle\\
\M(N_{3,0})=&\langle a_1,a_2,y_2\,|\, a_1a_2a_1=a_2a_1a_2,  y_2^2=(a_1y_2)^2=(a_2y_2)^2=(a_1a_2)^6=1\rangle
\end{align*}
In this section we describe some other known presentations of various mapping class groups and also state our main theorems which provide presentations for $\M(N_{g,n})$ for $n\in\{0,1\}$ and $g+n\ge 4$. 
\subsection{Orientable subsurface.}
\begin{figure}
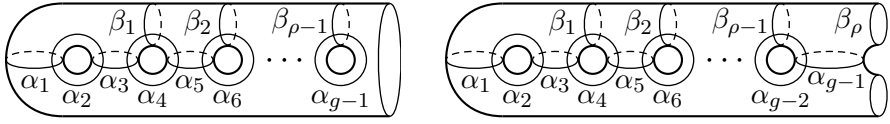

\begin{tabular}{cc}
\input{fig5}&\input{fig6}
\end{tabular}
\caption{\label{fig_S}Regular neighbourhood of the union of the curves $\alpha_i$ for $g=2\rho +1$ (left) and $g=2\rho+2$ (right).}
\end{figure}
Consider a regular neighbourhood $\Sigma$ of the union of the curves $\alpha_i$ for $i=1,\cdots,g-1$. This is an orientable subsurface of $N_{g,n}$ homeomorphic to
$S_{\rho,r}$, where $r\in\{1,2\}$ and $g=2\rho+r$ (Figure \ref{fig_S}).
The following theorem, whose proof is given in the Appendix, provides a presentation for $\M(S_{\rho,r})$, which will be a part of the presentation of $\M(N_{g,n})$. 

\begin{theorem}\label{presS}
For $r\in\{1,2\}$, $\rho\ge 1$ and $g=2\rho+r$, $\M(S_{\rho,r})$ admits a presentation with generators
$a_i$, $b_j$ for $1\le i\le g-1$, $0\le 2j\le g-2$ and relations:
\begin{itemize}
\item[(A1)] $a_ia_j=a_ja_i\quad$ for $|i-j|>1$,
\item[(A2)] $a_ia_{i+1}a_i=a_{i+1}a_ia_{i+1}\quad$ for $1\le i\le g-2$,
\item[(A3)] $a_ib_1=b_1a_i\quad$ for $i\ne 4$ if $g\ge 4$,
\item[(A4)] $b_1a_4b_1=a_4b_1a_4\quad$ if $g\ge 5$,
\item[(A5)] $(a_2a_3a_4b_1)^{10}=(a_1a_2a_3a_4b_1)^6\quad$ if $g\ge 5$,
\item[(A6)] $(a_2a_3a_4a_5a_6b_1)^{12}=(a_1a_2a_3a_4a_5a_6b_1)^{9}\quad$ if $g\ge 7$,
\item[(A7)] $b_0=a_1$,
\item [(A8)] $b_{i+1}=(b_{i-1}a_{2i}a_{2i+1}a_{2i+2}a_{2i+3}b_{i})^5(b_{i-1}a_{2i}a_{2i+1}a_{2i+2}a_{2i+3})^{-6}\quad$for $2\le 2i\le g-4$,
\item [(A9a)] $b_\rho a_{2\rho-3}=a_{2\rho-3}b_\rho\quad$ if $g=2\rho+2>6$,
\item [(A9b)] $b_2b_1=b_1b_2\quad$ if $g=6$.
\end{itemize}
\end{theorem}
It follows immediately from above presentation that $\M(S_{\rho,r})$ is generated by $b=b_1$ and $a_i$ for $i=1,\dots,g-1$. Moreover, if $g$ is odd, then we can drop the generators $b_j$ for $j\ne 1$ and the relations (A7, A8). The resulting presentation is the same as the one given in \cite{Mat}. However, if  we wanted to do the same for even $g$, then in the relations (A9a, A9b) the generator  $b_\rho$ would have to be replaced by its expression in terms of  $b$ and the $a_i$'s.
\begin{prop}\label{Sigma_inj}
The map $\jmath_\ast\colon\M(S_{\rho,r})\to\M(N_{g,n})$ induced by the inclusion of $\Sigma$ in $N_{g,n}$ is injective  for $n=1$, whereas for $n=0$ its kernel is an  infinite cyclic group generated by
$(a_1\cdots a_{g-1})^k$, where $k=g$ if $g$ is even, or $k=2g$ if $g$ is odd. The composition of $\jmath_\ast\colon\M(S_{\rho,r})\to\M(N_{g,1})$ with
$\imath_\ast\colon\M(N_{g,1})\to\M^+(N_{g,0},P)$ is also injective.
\end{prop}
\begin{proof}
Set $x=(a_1\cdots a_{g-1})^k$. If $g$ is odd then $x$ is equal to a Dehn twist about the boundary of $\Sigma$, while if $g$ is even then $x$ is the product of twists about the two boundary components (see \cite{LabPar}).
The complement in $N_{g,n}$ of the interior of $\Sigma$ is either a M\"obius band with $n$ holes if $g$ is odd, or an annulus with $n$ holes if $g$ is even. By
\cite[Theorem 3.6]{Stu_geom}, the maps $\jmath_\ast$ and $\imath_\ast\circ\jmath_\ast$ 
 are injective for $n=1$, whereas for $n=0$ the kernel of $\jmath_\ast$ is an infinite cyclic group generated by $x$.
\end{proof}

\subsection{Punctured disc and sphere.}\label{disc_sphere}
Mapping class groups of a punctured disc or sphere are very closely related to braid groups. In fact $\M(S_{0,1},\mathcal{P}_g)$ is isomorphic to the Artin braid group on $g$ strands, while $\M(S_{0,0},\mathcal{P}_g)$ is isomorphic to the quotient of the group of spherical braids on  $g$ strands by its center. Both groups are generated by $g-1$ elements called elementary braids or half twists. For $n\in\{0,1\}$ we have the blowup homomorphism
\[\mathfrak{b}\colon\M(S_{0,n},\mathcal{P}_g)\to\M(N_{g,n})\]
defined in Subsection \ref{blowup_cpush} which is injective and maps the elementary braids on the crosscap transpositions $u_i$ for $i=1,\dots,g-1$ (see \cite{SzepB}). 
From now on we will identify $\M(S_{0,n},\mathcal{P}_g)$ with its image in $\M(N_{g,n})$. The generators $u_i$ satisfy the well known defining relations listed in the following theorem (see \cite{Bir-book}).
\begin{theorem}\label{pres_braids}
The group $\M(S_{0,1},\mathcal{P}_g)$ admits a presentation with
generators $u_i$ for $i=1,\dots,g-1$ and relations
\begin{itemize}
\item[(B1)] $u_iu_j=u_ju_i\quad$ for $|i-j|>1$,
\item[(B2)] $u_iu_{i+1}u_i=u_{i+1}u_iu_{i+1}\quad$ for $i=1,\dots,g-2$.
\end{itemize}
The group $\M(S_{0,0},\mathcal{P}_g)$ is isomorphic to the quotient of $\M(S_{0,1},\mathcal{P}_g)$ by the relations 
\begin{itemize}
\item[(B3)] $(u_1\cdots u_{g-1})^g=1$,
\item[(B4)] $(u_1\cdots u_{g-2})^{g-1}=1$.\hfill{$\Box$}
\end{itemize}
\end{theorem}
For $k=1,\dots,g$ we define $\Delta_k\in\M(S_{0,1},\mathcal{P}_g)$ as
\[\Delta_1=1,\qquad
\Delta_k=(u_1\cdots u_{k-1})\Delta_{k-1}\]
The following relations hold in $\M(S_{0,1},\mathcal{P}_g)$.
\begin{itemize}
\item[(B5)] $\Delta_k u_i=u_{k-i}\Delta_k\quad$ for $i=1,\dots,k-1$
\item[(B6)] $\Delta_k=\Delta_{k-1}(u_{k-1}\cdots u_1)$
\item[(B7)] $\Delta_k^2=(u_1\cdots u_{k-1})^k$
\item[(B8)] $\Delta_k^2=\Delta_{k-1}^2(u_{k-1}\cdots u_1)(u_1\cdots u_{k-1})$
\end{itemize}
Equalities (B6) and (B8) are straightforward consequences of (B5) and (B7), which can be found in \cite{BS}.
By (B7) the relations (B3, B4) are respectively $\Delta_g^2=1$ and 
$\Delta_{g-1}^2=1$.
It follows immediately from (B8) that one of them may be replaced in the presentation of 
$\M(S_{0,0},\mathcal{P}_g)$  by
\[(\textrm{B4a})\quad (u_{g-1}\cdots u_1)(u_1\cdots u_{g-1})=1.\]
By (B5) $\Delta_g^2$ is central in $\M(S_{0,1},\mathcal{P}_g)$. Geometrically it is the right Dehn twist about the boundary of $N_{g,1}$.
\begin{lemma}\label{Delta_in_stab}
In $\M(S_{0,1},\mathcal{P}_g)$ we have
\[\Delta_g^2=u_{g-1}^2(u_{g-2}u_{g-1}^2u_{g-2})\cdots(u_1\cdots u_{g-1}^2\cdots u_1)\]
\end{lemma}
\begin{proof}
By expanding (B8) inductively we have
\[\Delta_g^2=u_{1}^2(u_{2}u_{1}^2u_{2})\cdots(u_{g-1}\cdots u_{1}^2\cdots u_{g-1}).\]
Conjugating both sides by $\Delta_g$ we obtain the desired equality. 
\end{proof}
\subsection{Main theorems.}
\begin{theorem}\label{mainA}
Let $g=2\rho +r$ for $r\in\{1,2\}$ and $\rho\ge 1$. The group
$\M(N_{g,1})$ is isomorphic to the quotient of the free product
$\M(S_{\rho,r})\ast\M(S_{0,1},\mathcal{P}_g)$ by the following relations.
\begin{itemize}
\item[(C1)] $a_1u_i=u_ia_1\quad$ for $i=3,\dots,g-1$
\item[(C2)] $a_iu_{i+1}u_i=u_{i+1}u_ia_{i+1}\quad$ for $i=1,\dots,g-2$
\item[(C3)] $a_{i+1}u_iu_{i+1}=u_iu_{i+1}a_i\quad$ for $i=1,\dots,g-2$
\item[(C4)] $a_1u_1a_1=u_1$ 
\item[(C5)] $u_2a_1a_2u_1=a_1a_2$ 
\item[(C6)] $(u_3b)^2=(a_1a_2a_3)^2(u_1u_2u_3)^2\quad$ if $g\ge 4$ 
\item[(C7)] $u_5b=bu_5\quad$ if $g\ge 6$ 
\item[(C8)]  $a_4u_4(a_4a_3a_2a_1u_1u_2u_3u_4)b=ba_4u_4\quad$ if $g\ge 5$ 
\end{itemize}
\end{theorem}

\begin{theorem}\label{mainB}
Let $g=2\rho +r$ for $r\in\{1,2\}$. For $g\ge 4$  the group $\M(N_{g,0})$ 
is isomorphic to the quotient of the free product
$\M(S_{\rho,r})\ast\M(S_{0,0},\mathcal{P}_g)$ by the relations (C1--C8) from Theorem \ref{mainA} and
\begin{itemize}
\item[(D)] $a_1(a_2\cdots a_{g-1}u_{g-1}\cdots u_2)a_1=a_2\cdots a_{g-1}u_{g-1}\cdots u_2$
\end{itemize}
\end{theorem}
Let $g=2\rho +r$ for $r\in\{1,2\}$. We define
\begin{itemize}
\item[$\G_{g,1}$] as the quotient of $\M(S_{\rho,r})\ast\M(S_{0,1},\mathcal{P}_g)$ by the relations (C1--C8),
\item[$\G_{g,0}$] as the quotient of $\M(S_{\rho,r})\ast\M(S_{0,0},\mathcal{P}_g)$ by the relations (C1--C8, D). 
\end{itemize}
The very essential idea of the proof of above theorems is the following.
In the first step we are going to show that there is a homomorphism 
$\varphi_{g,n}\colon\mathcal{G}_{g,n}\to \M(N_{g,n})$ and then the rest of the paper will be devoted to proving that it has an inverse.
\begin{prop}
Let $g=2\rho +r$ for $r\in\{1,2\}$. For $n\in\{0,1\}$ the map 
\[(\jmath_\ast\ast\mathfrak{b})\colon \M(S_{\rho,r})\ast\M(S_{0,n},\mathcal{P}_g)\to\M(N_{g,n})\] induces a homomorphism $\varphi_{g,n}\colon\mathcal{G}_{g,n}\to \M(N_{g,n})$.
\end{prop}
\begin{proof} We have to show that the relations (C1--C8, D) are satisfied in $\M(N_{g,n})$.
For $|i-j|>1$ the crosscap transposition $u_i$ is equal to the identity in a neighbourhood of the curve $\alpha_j$ and thus it commutes with the twist $a_j$. Thus (C1) is satisfied and analogously (C7). Observe that $u_{i+1}u_i(\alpha_{i+1})=\alpha_i$ and the local orientation used to define $a_i$ agrees with that induced by $u_{i+1}u_i$ from the local orientation used to define $a_{i+1}$. Thus $(u_{i+1}u_i)a_{i+1}(u_{i+1}u_i)^{-1}=a_i$ which is equivalent to (C2) and (C3) is proved analogously.
Since $u_i$ preserves $\alpha_i$ but reverses orientation of its neighbourhood thus
\[(\ast)\qquad u_ia_iu_i^{-1}=a_i^{-1}\qquad\mathrm{for\ }1\le i\le g-1.\] In particular (C4) is satisfied. Let $x=a_ia_{i+1}$. It can be easily checked that 
$x(\mu_{i+1})=\mu_{i+2}^{-1}$ and $x(\alpha_i)=\alpha_{i+1}^{-1}$. It follows that
$xy_ix^{-1}=y^{-1}_{i+1}$, hence $xa_iu_ix^{-1}=u_{i+1}^{-1}a_{i+1}^{-1}=a_{i+1}u_{i+1}^{-1}$, where the last equality follows from $(\ast)$.  By the braid relation (A2) we have $xa_ix^{-1}=a_{i+1}$ and thus
\[(\ast\ast)\qquad u_{i+1}a_ia_{i+1}u_i=a_ia_{i+1}\qquad\mathrm{for\ }1\le i\le g-2.\] 
In particular (C5) is satisfied. Let $K$ be a regular neighbourhood of 
$\beta\cup\alpha_3\cup\mu_4$. It is homeomorphic to Klein bottle with two holes and one of its boundary components is isotopic to $\alpha_1$ while the other one is isotopic to $y_2^{-1}u_3^{-1}y_2^{-1}(\alpha_1)$. By \cite[Lemma 7.8]{Szep_Osaka} we have
\begin{align*}
&(u_3b)^2=a_1y_2^{-1}u_3^{-1}y_2^{-1}a_1y_2u_3y_2=
a_1\underline{u_2^{-1}a_2^{-1}}u_3^{-1}\underline{u_2^{-1}a_2^{-1}}a_1a_2\underline{u_2u_3a_2}u_2\stackrel{(\ast,C3)}{=}\\
&a_1a_2\underline{u_2^{-1}u_3^{-1}a_2}u_2^{-1}a_1a_2a_3u_2u_3u_2\stackrel{(C2)}{=}
a_1a_2a_3\underline{u_2^{-1}u_3^{-1}u_2^{-1}a_1a_2a_3}u_2u_3u_2
\stackrel{(\ast\ast)}{=}\\
&(a_1a_2a_3)^2u_1u_2u_1u_2u_3u_2\stackrel{(B1,B2)}{=}(a_1a_2a_3)^2(u_1u_2u_3)^2
\end{align*}
which proves (C6).
Let $z=a_4a_3a_2a_1u_1u_2u_3u_4$. We have
\begin{align*}
z&=(a_4u_4)(u_4^{-1}a_3u_3u_4)(u_4^{-1}u_3^{-1}a_2u_2u_3u_4)(u_4^{-1}u_3^{-1}u_2^{-1}a_1u_1u_2u_3u_4)\\
&=y_4(u_4^{-1}y_3u_4)(u_4^{-1}u_3^{-1}y_2u_3u_4)(u_4^{-1}u_3^{-1}u_2^{-1}y_1u_2u_3u_4)\\
&=Y_{\mu_5,\gamma_{\{4,5\}}}Y_{\mu_5,\gamma_{\{3,5\}}}Y_{\mu_5,\gamma_{\{2,5\}}}Y_{\mu_5,\gamma_{\{1,5\}}}
\end{align*}
Consider the surface $N'$ obtained by cutting $N_{g,n}$ along $\mu_5$ and then gluing a disc with 
a puncture $P$ along the resulting boundary component. Then $N_{g,n}$ may be seen as being obtained from $N'$ by blowing up the puncture and we have the crosscap pushing map $\mathfrak{c}\colon\pi_1(N',P)\to\M(N_{g,n})$ whose image contains the crosscap slides $Y_{\mu_5,\gamma_{\{i,5\}}}$ for $i=1,2,3,4$. Since this is a homomorphism, thus $z$ is isotopic to the effect of pushing $\mu_5$ once along $\gamma_{\{1,2,3,4,5\}}$. One of the boundary components of the regular neighbourhood of $\mu_5\cup \gamma_{\{1,2,3,4,5\}}$ is isotopic to $\beta$, while the other one is isotopic to $y_4^{-1}(\beta)$. From Lemma \ref{push1} we have $z=y_4^{-1}by_4b^{-1}$ which is equivalent to (C8). 
Finally it is easy to check that if the surface is closed (i.e. $n=0$) then $a_2\cdots a_{g-1}u_{g-1}\cdots u_2$ preserves the curve $\alpha_1$ (up to isotopy) and reverses orientation of its neighbourhood, which proves the relation (D). 
\end{proof}

Consider the exact sequence (\ref{Cup_es}) in the case $m=1$.
\[1\to\Z\to\M(N_{g,1})\stackrel{\imath_\ast}{\to}\M^+(N_{g,0},P)\to 1.\]
The kernel of $\imath_\ast$ is generated by the Dehn twist 
$T_{\bdr{N_{g,1}}}=\Delta_g^2$. Let $\G_{g,0}^1$ denote the quotient of $\G_{g,1}$ by the normal closure of $\Delta_{g}^2$. Since 
$\imath_\ast(\varphi_{g,1}(\Delta_g^2))=1$, there is a homomorphism 
$\varphi_{g,0}^1\colon\G_{g,0}^1\to\M^+(N_{g,0},P)$ such that
$\varphi_{g,0}^1\circ p=\imath_\ast\circ\varphi_{g,1}$, where
$p\colon\G_{g,1}\to\G_{g,0}^1$ is the canonical projection.
By abuse of notation, we will denote by the same symbols the generators of $\G_{g,1}$ and their images in $\G_{g,0}^1$. We will prove in Section \ref{sec_mainA} that $\varphi_{g,0}^1$ is an isomorphism.
\subsection{Some consequences of the defining relations.}
Throughout this paper we will often have to verify that some relations are satisfied in $\G_{g,n}$ or $\G_{g,0}^1$ . In this subsection we prove the most useful ones.
 
\begin{lemma}\label{useful_C}
The following relations hold in $\mathcal{G}_{g,n}$ for $n=0,1$.
\begin{itemize}
\item[(C1a)] $a_iu_j=u_ja_i\quad$ for $|j-i|>1$ 
\item[(C4a)] $a_iu_ia_i=u_i\quad$ for $i=1,\cdots,g-1$ 
\item[(C5a)] $u_{i+1}a_ia_{i+1}u_i=a_ia_{i+1}\quad$ for $i=1,\cdots,g-2$ 
\item[(C6a)] $(bu_3)^2=(u_3b)^2=(a_1a_2a_3)^2(u_1u_2u_3)^2$
\item[(C7a)] $u_ib=bu_i\quad$ for $i=5,\cdots,g-1$. 
\end{itemize}
\end{lemma}
\begin{proof}
Fix $i>1$ and let $x=(u_{i-1}u_i)\cdots(u_2u_3)(u_1u_2)$. By (C3) we have
$xa_1x^{-1}=a_i$ and by the relations (B1,B2) we have $xu_1x^{-1}=u_i$ and
\[xu_jx^{-1}=\begin{cases}
u_j\textrm{\ for\ }j>i+1\\
u_{j-2}\textrm{\ for\ }3\le j<i+1
\end{cases}\]
Thus (C1a) may be obtained by conjugating by $x$ both sides of (C1). 
If we set $y=u_1u_2\cdots u_{g-1}$, then, for $i\in\{1,\dots,g-2\}$, we have
$yu_iy^{-1}=u_{i+1}$ by (B1,B2) and $ya_iy^{-1}=a_{i+1}$ by (C1a,C3). Hence (C4a,C5a) follow from (C4,C5) by applying conjugation by $y$ as many times as needed.
We have \[u_3(a_1a_2a_3)^2=(a_1a_2a_3)u_2^{-1}(a_1a_2a_3)=(a_1a_2a_3)^2u_1\] by (C1a,C5a) and  $u_1(u_1u_2u_3)^2=(u_1u_2u_3)^2u_3$ by (B1,B2). Thus $u_3$ commutes with
$(a_1a_2a_3)^2(u_1u_2u_3)^2$, which together with (C6) proves (C6a).
If $i>5$ then for $z=(a_{i-1}a_i)\cdots(a_5a_6)$ we have $zu_5z^{-1}=u_i^{\pm 1}$ by (C5a) and $zbz^{-1}=b$ by (A3). Thus (C7a) is obtained by conjugating both sides of (C7) by $z$.
\end{proof}

\begin{lemma}\label{useful_T}
The following relations hold in $\mathcal{G}_{g,n}$ for $n=0,1$
\begin{itemize}
\item[(E1)] $\Delta_k a_i=a^{-1}_{k-i}\Delta_k\quad$ for $2\le k\le g$ and $i=1,\dots,k-1$
\item[(E2)] $r_g^2=\Delta_g^2$
\item[(E3)] $r_ga_i=a_ir_g$ for $i=2,\dots,g-1$
\item[(E4)] $u_ir_gu_i=r_g$ for $i=2,\dots,g-1$
\end{itemize}
In $\mathcal{G}^1_{g,0}$ and $\mathcal{G}_{g,0}$ we have
\begin{itemize}
\item[(E2a)] $r_g^2=1$
\end{itemize}
\end{lemma}
\begin{proof}
We prove (E1) by induction on $k$. For $k=2$ it is equivalent to (C4). Suppose (E1) is true for some $k\ge 2$. For $i\le k-1$ we have
\begin{align*}
\Delta_{k+1}a_i&=(u_1\cdots u_k)\Delta_k a_i=(u_1\cdots u_k)a^{-1}_{k-i}\Delta_k
\stackrel{(C1a,C3)}{=}a^{-1}_{k+1-i}(u_1\cdots u_k)\Delta_k=a^{-1}_{k+1-i}\Delta_{k+1}
\end{align*}
and for $i=k$
\begin{align*}
\Delta_{k+1}a_k&=\Delta_k(u_k\cdots u_1)a_k\stackrel{(C1a,C2)}{=}
\Delta_k a_{k-1}(u_k\cdots u_1)
=a_1^{-1}\Delta_k(u_k\cdots u_1)=a_1^{-1}\Delta_{k+1}
\end{align*}
which finishes the proof of (E1).

Let $x=a_1\cdots a_{g-1}$ and $z=u_{g-1}\cdots u_1$, so that $r_g=xz$.
We are going to prove by induction on $i$ that 
\[(\ast)\quad xz(u_2\cdots u_{g-1})(u_2\cdots u_{g-2})\cdots(u_2\cdots u_{i+1})x=\Delta_g\Delta_{i}\]
for $i=1,\dots,g-1$. If $i=1$ then $(\ast)$ becomes 
$x\Delta_gx=\Delta_g$
and it follows from (E1). Suppose that $(\ast)$ holds for some $i<g-1$. By (C1a,C5a) we have
$(u_2\cdots u_{i+1})x=x(u_i\cdots u_1)^{-1}$
 and 
\[xz(u_2\cdots u_{g-1})\cdots(u_2\cdots u_{i+2})x=\Delta_g\Delta_{i}(u_{i}\cdots u_1)=\Delta_g\Delta_{i+1}.\]
For $i=g-1$ we obtain 
\[xzx=\Delta_g\Delta_{g-1}\iff (xz)^2=\Delta_g\Delta_{g-1}z=\Delta_g^2\]
which is equivalent to (E2). For $i=2,\dots,g-1$ we have
$a_ix=xa_{i-1}$ by (A1, A2), $a_{i-1}z=za_i$ by (C1a, C2),
$u_ix=xu_{i-1}^{-1}$ by (C1a, C5a), $u_{i-1}z=zu_i$ by (B1, B2). The relations (E3, E4) follow. In $\G^1_{g,0}$ and $\G_{g,0}$ we have $\Delta^2_g=1$ (B3) and thus (E2a) is a consequence of (E2).
\end{proof}

\begin{lemma}\label{useful_T_closed}
In $\mathcal{G}_{g,0}$ we have
\begin{itemize}
\item[(E3a)] $r_ga_i=a_ir_g$ for $i=1,\dots,g-1$
\item[(E4a)] $u_ir_gu_i=r_g$ for $i=1,\dots,g-1$
\item[(E5)] $(a_1\cdots a_{g-1})^2=(u_{g-1}\cdots u_1)^{-2}=(u_1\cdots u_{g-1})^2$ 
\item[(E6)] $(a_1\cdots a_{g-1})^k=1$, where $k=g$ if $g$ is even or $k=2g$ if $g$ is odd.
\end{itemize}
\end{lemma}
\begin{proof}
For $i>1$ (E3a, E4a) are the same as (E3, E4), while for $i=1$ (E3a)  follows from (D, C4), and 
\[r_gu_1r_g\stackrel{C5}{=}r_ga_2^{-1}a_1^{-1}u_2^{-1}a_1a_2r_g\stackrel{E3a}{=}a_2^{-1}a_1^{-1}r_gu_2^{-1}r_ga_1a_2\stackrel{E4}{=}u_1^{-1}\]
By (E2a) we have
\begin{align*}
&1=r_g^2=r_g(a_1\cdots a_{g-1}u_{g-1}\cdots u_1)\stackrel{E3a}{=}(a_1\cdots a_{g-1})r_g(u_{g-1}\cdots u_1)\\
&=(a_1\cdots a_{g-1})^2(u_{g-1}\cdots u_1)^2\stackrel{B4a}{=}(a_1\cdots a_{g-1})^2(u_1\cdots u_{g-1})^{-2}
\end{align*}
This proves (E5), which together with (B3) implies (E6).
\end{proof}

\begin{lemma}\label{Da_replace}
In the presentation of $\M(N_{g,0})$ given in Theorem \ref{mainB} the relation (D) may be replaced by 
\[\mathrm{(Da)}\quad a_{g-1}(u_{g-2}\cdots u_1a_1\cdots a_{g-2})a_{g-1}=u_{g-2}\cdots u_1a_1\cdots a_{g-2}\]
\end{lemma}
\begin{proof} If we conjugate both sides of (Da) by $\Delta_g$, then by (B5) and (E1) we obtain 
\begin{align*}
&a_1^{-1}(u_2\cdots u_{g-1}a^{-1}_{g-1}\cdots a^{-1}_2)a_1^{-1}=u_2\cdots u_{g-1}a^{-1}_{g-1}\cdots a^{-1}_2\stackrel{C4}{\iff}\\
&a_1(u_1u_2\cdots u_{g-1}a^{-1}_{g-1}\cdots a^{-1}_2)a_1^{-1}=u_1u_2\cdots u_{g-1}a^{-1}_{g-1}\cdots a^{-1}_2\stackrel{B4a}{\iff}\\
&a_1(u_1^{-1}u_2^{-1}\cdots u_{g-1}^{-1}a^{-1}_{g-1}\cdots a^{-1}_2)a_1^{-1}=u_1^{-1}u_2^{-1}\cdots u_{g-1}^{-1}a^{-1}_{g-1}\cdots a^{-1}_2\stackrel{C4}{\iff}\\
&a_1^{-1}(u_2^{-1}\cdots u_{g-1}^{-1}a^{-1}_{g-1}\cdots a^{-1}_2)a_1^{-1}=u_2^{-1}\cdots u_{g-1}^{-1}a^{-1}_{g-1}\cdots a_2^{-1} 
\end{align*}
which is equivalent to (D).
\end{proof}

\begin{lemma}\label{shortcut_AB}
Suppose that $w$ is either (1) a word in the generators $u_i$ and their inverses  or (2) a word in the generators
$a_i$, $b_j$ and their inverses. Then $w$ represents the trivial element of $\M(N_{g,n})$ or $\M^+(N_{g,0},P)$ if and only if it represents the trivial element of $\G_{g,n}$ or $\G_{g,0}^1$ respectively.
\end{lemma}
\begin{proof}
Case (1) follows from the injectivity of $\mathfrak{b}\colon\M(S_{0,n},\mathcal{P}_g)\to\M(N_{g,n})$  (Proposition \ref{blowup_inj}) and the fact that the kernel of
$\imath_\ast\circ\mathfrak{b}\colon\M(S_{0,1},\mathcal{P}_g)\to\M^+(N_{g,0},P)$
is generated by $\Delta_g^2$, which is trivial in $\G_{g,0}^1$.
Analogously for (2), we have that
$\jmath_\ast\colon\M(S_{\rho,r})\to\M(N_{g,1})$ and $\imath_\ast\circ\jmath_\ast\colon\M(S_{\rho,r})\to\M^+(N_{g,0},P)$ are injective by Proposition \ref{Sigma_inj}.
For $n=0$ we have to check that the image under the map $\M(S_{\rho,r})\to\G_{g,0}$ of the kernel of $\jmath_\ast\colon\M(S_{\rho,r})\to\M(N_{g,0})$ is trivial. This is the case by Proposition \ref{Sigma_inj} and the relation (E6). 
\end{proof}

\begin{lemma}\label{useful_g4}
In $\mathcal{G}^1_{4,0}$ we have:
\begin{itemize}
\item[(G1)] $bu_3u_2b^{-1}=(a_1a_2a_3)^3u_2u_3(a_1a_2a_3)^{-1}$
\item[(G2)] $bu_3u_2u_1b=(a_1a_2a_3)^3(u_3u_2u_1)^{-1}(a_1a_2a_3)^3$
\item[(G3)] $((a_1a_2a_3)^{-4}br_4)^2=1$
\end{itemize}
In $\mathcal{G}_{4,0}$ we have 
\begin{itemize}
\item[(G3a)] $(br_4)^2=1$
\end{itemize}
\end{lemma}
\begin{proof}
Let $x=a_1a_2a_3$ and $z=u_1u_2u_3$. We have
\[(C6a)\ (bu_3)^2=(u_3b)^2=x^2z^2,\quad (i)\ u_2=x^{-1}u_3^{-1}x,\quad (ii)\ u_1=x^{-1}u_2^{-1}x,\]
the last two relations following from (C5a).
Since $b$ commutes with $x$, from (C6a, i, ii) we obtain
\[(iii)\ (b^{-1}u_2)^2=x^{-1}z^{-2}x^{-1},\qquad (iv)\ (bu_1)^2=z^2x^2.\]
\begin{align*}
&\underline{bu_3}u_2b^{-1}\stackrel{C6a}{=}x^2z^2u_3^{-1}\underline{b^{-1}u_2b^{-1}}\stackrel{(iii)}{=}
x^2z^2u_3^{-1}x^{-1}z^{-2}\underline{x^{-1}u_2^{-1}}
\stackrel{(ii)}{=}x^2\underline{z^2u_3^{-1}}x^{-1}\underline{z^{-2}u_1}x^{-1}\\
&=
x^2u_1u_2u_3u_1u_2x^{-1}\underline{u_3^{-1}u_2^{-1}u_1^{-1}u_3^{-1}}u_2^{-1}x^{-1}
\stackrel{B1,B2}{=}x^2u_1u_2u_3u_1u_2\underline{x^{-1}u_2^{-1}}u_3^{-1}u_2^{-1}u_1^{-1}u_2^{-1}x^{-1}\\
&\stackrel{(ii)}{=}x^2\underline{u_1u_2u_3u_1u_2u_1}x^{-1}u_3^{-1}u_2^{-1}u_1^{-1}u_2^{-1}x^{-1}=
x^2\underline{\Delta_4x^{-1}}u_3^{-1}u_2^{-1}u_1^{-1}u_2^{-1}x^{-1}\\
&\stackrel{E1}{=}x^3\Delta_4u_3^{-1}u_2^{-1}u_1^{-1}u_2^{-1}x^{-1}\stackrel{(B1,B2)}{=}x^3u_2u_3x^{-1}.
\end{align*}
\begin{align*}
&bu_3u_2u_1b=(bu_3u_2b^{-1})(bu_1b)=(x^3\Delta_4u_3^{-1}u_2^{-1}u_1^{-1}\underline{u_2^{-1}x^{-1}})(z^2\underline{x^2u_1^{-1}})\\
&\stackrel{(i,ii)}{=}x^3\Delta_4u_3^{-1}u_2^{-1}u_1^{-1}x^{-1}\underline{u_3u_1u_2u_3u_1u_2}x^2=
x^3\Delta_4u_3^{-1}u_2^{-1}u_1^{-1}x^{-1}\Delta_4x^2\\
&\stackrel{E1,B5}{=}x^3\Delta^2_4(u_3u_2u_1)^{-1}x^3\stackrel{B3}{=}x^3(u_3u_2u_1)^{-1}x^3
\end{align*}
\[(x^{-4}br_4)^2=x^{-3}\underline{bu_3u_2u_1b}x^{-3}u_3u_2u_1\stackrel{(G2)}{=}(u_3u_2u_1)^{-1}u_3u_2u_1=
1\]
The relation (G3a) follows from (G3) and (E6).
\end{proof}
\section{The base cases.}\label{sec_base}
In this section we deduce the main theorems for
$(g,n)\in\{(3,1),(4,0)\}$ from the presentations of $\M(N_{g,n})$ obtained in \cite{Szep_Osaka,Szep1}. 
Theorem \ref{mainA} for $g=3$ follows from the following.
\begin{theorem}\label{g3n1}
The map $\varphi_{3,1}\colon\G_{3,1}\to\M(N_{3,1})$ is an isomorphism.
\end{theorem}
\begin{proof}
By \cite[Theorem 7.16]{Szep_Osaka} $\M(N_{3,1})$ admits a presentation  with generators
$a_1$, $a_2$, $u_2$, $d$ (called respectively $B$, $A_1$, $U$, $A_2$ in \cite{Szep_Osaka}),
where $d$ is a Dehn twist about the curve $a_1^{-1}u_2(\alpha_1)$.
The defining relations are 
\begin{align*}
&(i)\ a_2d=da_2\qquad (ii)\ a_2a_1a_2=a_1a_2a_1\qquad (iii)\ da_1d=a_1da_1\qquad
(iv)\ u_2a_2u_2^{-1}=a_2^{-1}\\
&(v)\  u_2a_1u_2^{-1}=a_1d^{-1}a_1^{-1}\qquad 
(vi)\ (du_2)^2=(u_2d)^2\qquad (vii)\  (du_2)^2=(a_2d^2a_1)^3.
\end{align*}
We define $\psi\colon\M(N_{3,1})\to\G_{3,1}$ on the generators as
$\psi(a_1)=a_1$, $\psi(a_2)=a_2$, $\psi(u_2)=u_2$, $\psi(d)=a_1^{-1}u_2a_1^{-1}u_2^{-1}a_1$. To prove that $\psi$ is a homomorphism we have to show that it respects the relations (i--vii). This is obvious for (v) and (ii, iv) are (A2,C4a). 
\begin{align*}
\psi(d)&=a_1^{-1}\underline{u_2a_1^{-1}u_2^{-1}}a_1\stackrel{C3}{=}a_1^{-1}u_1^{-1}a_2^{-1}\underline{u_1a_1}
\stackrel{C4}{=}a_1^{-1}\underline{u_1^{-1}a_2^{-1}a_1^{-1}}u_1\\
&\stackrel{C5}{=}a_1^{-1}a_2^{-1}a_1^{-1}u_2u_1=(a_1a_2a_1)^{-1}u_2u_1\\
\psi(a_2)\psi(d)&=a_2(a_1a_2a_1)^{-1}u_2u_1\stackrel{A2}{=}(a_1a_2a_1)^{-1}a_1u_2u_1
\stackrel{C2}{=}(a_1a_2a_1)^{-1}u_2u_1a_2=\psi(d)\psi(a_2)
\end{align*}
The relation $\psi(d)\psi(a_1)\psi(d)=\psi(a_1)\psi(d)\psi(a_1)$ is equivalent to
\begin{align*}
&(a_1a_2a_1)^{-1}u_2u_1a_1(a_1a_2a_1)^{-1}u_2u_1=a_1(a_1a_2a_1)^{-1}u_2u_1a_1\iff\\
&u_2\underline{u_1a_2^{-1}a_1^{-1}u_2}u_1=a_2u_2u_1a_1\stackrel{C5}{\iff}
u_2a_2^{-1}a_1^{-1}u_1=a_2u_2u_1a_1
\end{align*}
The last relation follows easily from (C4a).
\begin{align*}
(\psi(d)\psi(u_2))^2&=(a_1a_2a_1)^{-1}u_2u_1u_2(a_1a_2a_1)^{-1}u_2u_1u_2\\
&=(a_1a_2a_1)^{-1}\Delta_3(a_1a_2a_1)^{-1}\Delta_3\stackrel{E1}{=}\Delta_3^2\\
(\psi(u_2)\psi(d))^2&=u_2(\psi(d)\psi(u_2))^2u_2^{-1}=u_2\Delta^2_3u_2^{-1}=\Delta^2_3=(\psi(d)\psi(u_2))^2
\end{align*}
\begin{align*}
&(\psi(a_2)\psi(d)^2\psi(a_1))^3=(a_2(a_1a_2a_1)^{-1}u_2u_1(a_1a_2a_1)^{-1}u_2u_1a_1)^3\\
&\stackrel{A2}{=}(a_2(a_2a_1a_2)^{-1}u_2u_1(a_2a_1a_2)^{-1}u_2u_1a_1)^3
=a_1^{-1}(a_2^{-1}u_2\underline{u_1a_2^{-1}a_1^{-1}}a_2^{-1}u_2u_1)^3a_1\\
&\stackrel{C5}{=}a_1^{-1}(\underline{a_2^{-1}u_2a_2^{-1}}a_1^{-1}\underline{u_2^{-1}a_2^{-1}u_2}u_1)^3a_1
\stackrel{C4a}{=}a_1^{-1}(u_2a_1^{-1}a_2u_1)^3a_1\\
&=a_1^{-1}(u_2a_1^{-1}\underline{a_2u_1u_2}a_1^{-1}\underline{a_2u_1u_2}a_1^{-1}a_2u_1)a_1
=a_1^{-1}(u_2a_1^{-1}u_1u_2u_1u_2a_2u_1)a_1\\
&=a_1^{-1}(u_2a_1^{-1}\Delta_3u_2a_2u_1)a_1\stackrel{B5,E1}{=}
a_1^{-1}(\Delta_3u_1\underline{a_2u_2a_2}u_1)a_1\\
&\stackrel{C4a}{=}a_1^{-1}(\Delta_3u_1u_2u_1)a_1=a_1^{-1}\Delta_3^2a_1\stackrel{E1}{=}\Delta_3^2=(\psi(d)\psi(u_2))^2
\end{align*}
Thus $\psi$ is a homomorphism. Observe that $\G_{3,1}$ and $\M(N_{3,1})$ are generated by $a_1$, $a_2$ and $u_2$. Indeed, the generator $u_1$ of $\G_{3,1}$ is redundant because of (C5) and the generator $d$ of $\M(N_{3,1})$ is redundant because of (v). Since for $x\in\{a_1, a_2, u_2\}$ we have
$\psi(\varphi_{3,1}(x))=x=\varphi_{3,1}(\psi(x))$ thus $\psi$ is the inverse of $\varphi_{3,1}$.
\end{proof}
\begin{cor}\label{genus3_shortcut}
Suppose that $g\ge 4$, $i\in\{1,\dots,g-2\}$ is fixed  and $w\in\G_{g,n}$ is represented by a word in the generators $a_i$, $u_i$, 
$a_{i+1}$, $u_{i+1}$ and their inverses. Then $\varphi_{g,n}(w)=1$ if and only if $w=1$.
\end{cor}
\begin{proof} Consider a subsurface $K\subset N_{g,n}$ which is a disc with crosscaps $i, i+1, i+2$. Thus $K$ is a copy of $N_{3,1}$. By \cite[Theorem 3.6]{Stu_geom} the map $\imath_\ast\colon\M(K)\to\M(N_{g,n})$ induced by the inclusion of $K$ in $N_{g,n}$ is injective, except for the case $(g,n)=(4,0)$ where its kernel is generated by a Dehn twist about the boundary of $K$. As in the proof of Theorem \ref{g3n1}, there is  a homomorphism $\psi\colon\imath_\ast(\M(K))\to\G_{g,n}$ such that $\psi(x)=x$ for $x\in\{a_i,a_{i+1},u_i,u_{i+1}\}$. For $(g,n)=(4,0)$ we additionally have to verify that $\psi(T_{\bdr{K}})=1$, which is true by (B4), because either
$T_{\bdr{K}}=\Delta_2^2$ if $i=1$, or $T_{\bdr{K}}=\Delta_3\Delta_2^2\Delta_3$ if $i=2$.
Since $\psi(\varphi_{g,n}(w))=w$ the corollary is proved.
\end{proof}
Theorem \ref{mainB} for $g=4$ follows from the following.
\begin{theorem}\label{g4n0}
The map $\varphi_{4,0}\colon\G_{4,0}\to\M(N_{4,0})$ is an isomorphism.\end{theorem}
\begin{proof}
By \cite[Theorem 2.1]{Szep1} $\M(N_{4,0})$ admits the presentation with generators 
$a_i,u_i,$ for $i=1,2,3$, $b$, $r_4$, $d$ (the last two generators are denoted respectively as $t$ and  $a_4$ in \cite{Szep1}), where $d$ is a Dehn twist about the curve $a_2^{-1}u_3(\alpha_2)$.  
The defining relations are (A1--A3, B1, C1, C4, C5a, E2a, E3a, E4, E6, G3a) 
and
\begin{align*}
&(i)\ r_4=a_1a_2a_3u_3u_2u_1\qquad (ii)\  u_3a_2u_3^{-1}=a_2d^{-1}a_2^{-1}\qquad
(iii)\ u_1^2=u_3^2\\
&(iv)\ (u_3b)^2=1\qquad
(v)\ (u_3d)^2=1\qquad
(vi)\ da_3=a_3d\\
&(vii)\ da_2d=a_2da_2\qquad
(viii)\ (da_2a_3)^4=1\qquad
(ix)\ u_3du_3^{-1}=u_1du_1^{-1}.
\end{align*}
We define $\psi\colon\M(N_{4,0})\to\G_{4,0}$ on the generators as
$\psi(a_i)=a_i$, $\psi(u_i)=u_i$ for $i=1,2,3$, $\psi(b)=b$,
$\psi(r_4)=a_1a_2a_3u_3u_2u_1$ and $\psi(d)=a_2^{-1}u_3a_2^{-1}u_3^{-1}a_2$.
To show that $\psi$ is a homomorphism we have to show that the relations (iii--ix) are satisfied in $\G_{4,0}$.
By Lemma \ref{shortcut_AB} the relation (iii) is satisfied in $\G_{4,0}$. 
The relations (v,vi,vii,viii) can be rewritten using (ii) in the generators $a_2, u_2, a_3, u_3$ and so they hold in $\mathcal{G}_{4,0}$ by Corollary \ref{genus3_shortcut}. We have
\[(u_3b)^2\stackrel{(C6)}{=}(a_1a_2a_3)^2(u_1u_2u_3)^2\stackrel{(E5)}{=}(a_1a_2a_3)^4\stackrel{(E6)}{=}1\]
As in the proof of Theorem \ref{g3n1} we have 
$\psi(d)=(a_2a_3a_2)^{-1}u_3u_2$ and 
$\psi(u_3)\psi(d)\psi(u_3)^{-1}=\psi(u_1)\psi(d)\psi(u_1)^{-1}$ is equivalent to
\begin{align*}
&u_3(a_2a_3a_2)^{-1}u_3u_2u_1=u_1\underline{(a_2a_3a_2)^{-1}u_3u_2u_3}\stackrel{(C2,C3,C4a)}{\iff}\\
&u_3(a_2a_3a_2)^{-1}u_3u_2u_1=u_1u_3u_2u_3a_2a_3a_2\stackrel{(B1)}{\iff}\\
&(a_2a_3a_2)^{-1}\underline{u_3u_2u_1(a_2a_3a_2)^{-1}}=u_1u_2u_3\stackrel{(C1a,C2)}{\iff}\\
&\underline{(a_2a_3a_2)^{-1}(a_1a_2a_1)^{-1}}u_3u_2u_1=u_1u_2u_3\stackrel{(A1,A2)}{\iff}\\
&(a_1a_2a_3)^{-2}=u_1u_2u_3(u_3u_2u_1)^{-1}\stackrel{(E5)}{\iff}\\
&(u_3u_2u_1)^2=u_1u_2u_3(u_3u_2u_1)^{-1}
\end{align*}
The last relation is satisfied in $\G_{4,0}$ by Lemma \ref{shortcut_AB}. Since $\varphi_{4,0}\circ\psi=id$, hence $\psi$ is injective, and since $\G_{4,0}$ is generated by $a_i, u_i$ and $b$, it is also surjective. It follows that $\varphi_{4,0}$ is an isomorphism.
\end{proof}
\section{Curve complexes}\label{sec_curves}
\subsection{Definitions and simple connectedness.}
Let $N=N_{g,n}$.
Suppose that $C=(\gamma_1,\dots,\gamma_m)$ is an $m$-tuple of generic curves on $N$. We say that $C$ is a {\it generic $m$-tuple of disjoint curves} if for
$i\ne j$
\begin{itemize}
\item $\gamma_i$ is disjoint from $\gamma_j$, and
\item $\gamma_i$ is neither isotopic to $\gamma_j$ nor to $\gamma^{-1}_j$. 
\end{itemize}
We denote by $N_C$ the compact surface obtained by cutting $N$ along $C$.
If $C'=(\gamma'_1,\dots,\gamma'_m)$ then we say that $C$ and $C'$ are {\it equivalent} if $\gamma_i$ is isotopic to ${\gamma'}^{\pm 1}_i$ for
$i=1,\dots,m$, and {\it equivalent up to permutation} if
$\gamma_i$ is isotopic to ${\gamma'}_{\tau(i)}^{\pm 1}$ for $i=1,\dots,m$ and for some permutation $\tau\in\Sym_m$. We denote by
$[C]=[\gamma_1,\dots,\gamma_m]$ the equivalence class of $C$, and by
$\lr{C}=\lr{\gamma_1,\dots,\gamma_m}$ its equivalence class up permutation.

The {\it complex of curves} $\C(N)$ is a simplicial complex whose $m$-simplices are the equivalence classes up to permutation of generic $(m+1)$-tuples of disjoint curves on $N$. We are going to use its two full subcomplexes: $\C_0(N)$ is the subcomplex of $\C(N)$ consisting of simplices $\lr{C}$ such that $N_C$ is connected; $\D(N)$ is the subcomplex of $\C_0(N)$ consisting of simplices $\lr{C}$ such that $N_C$ is nonorientable.

The {\it ordered complex of curves} $\C^{ord}(N)$ is a $\Delta$-complex
(in the sense of \cite{Hat}, Chapter 2) whose $m$-simplices are the
equivalence classes of generic $(m+1)$-tuples of disjoint curves
on $N$. If $[\gamma_1,\dots,\gamma_{m+1}]$ is an $m$-simplex then its
faces are the $(m-1)$-simplices
$[\gamma_1,\dots,\widehat{\gamma_i},\dots,\gamma_{m+1}]$ for
$i=1,\dots,m+1$, where $\widehat{\gamma_i}$ means that $\gamma_i$ is
deleted. We define the subcomplexes $\C^{ord}_0(N)$ and $\D^{ord}(N)$ as the ordered versions of $\C_0(N)$ and $\D(N)$. 

The complex of curves was introduced by Harvey \cite{Harvey} and the ordered complex of curves by Benvenuti \cite{Benv}.
The following theorem was proved for the unordered complexes in \cite[Theorems 5.4, 5.5]{Wahl}. To obtain the result for their ordered versions, the proof of \cite[Proposition 8]{Benv} can be applied.
\begin{theorem}\label{simplyc_wahl}
The complexes $\C_0(N_{g,n})$ and $\C_0^{ord}(N_{g,n})$ are simply connected for $g\ge 5$; $\D(N_{g,n})$ and $\D^{ord}(N_{g,n})$ are simply connected for $g\ge 7$.\hfill{$\Box$} 
\end{theorem}
The mapping class group $\M(N)$ acts on the set of  isotopy classes of generic curves on $N$, and thus it also acts on the complexes $\C(N)$, $\C_0(N)$, $\D(N)$ and their ordered versions by permuting their simplices. We say that two simplices $\sigma_1$, $\sigma_2$ are {\it $\M(N)$-equivalent} if $\sigma_2=h\sigma_1$ for some $h\in\M(N)$. Observe that $\C^{ord}(N)$ has a natural orientation (the vertices of every simplex are ordered) preserved by $\M(N)$. In particular, $\M(N)$ acts on the $1$-simplices of $\C^{ord}(N)$ without inversion, which simplifies the statement of Brown's theorem below. This is the only, purely technical, reason for considering $\C^{ord}(N)$ instead of $\C(N)$. 
\subsection{The structure of a stabiliser.}\label{struct_stab}
Let $N=N_{g,n}$. 
Suppose that $C=(\gamma_1,\dots,\gamma_m)$ is a  generic $m$-tuple of disjoint curves on $N$ such that $N_C$ is connected. Suppose that $\gamma_i$ are two-sided for $i\le r$ and one-sided for $i>r$.

Let $\Stab[C]=\Stab_{\M(N)}[C]$ denote the stabiliser of the simplex $[C]$  with respect to the action of $\M(N)$ on $C_0^{ord}(N)$. This is the subgroup of $\M(N)$ consisting of the  isotopy classes of homeomorphisms fixing each curve $\gamma_i$. We define
$\Stab^+[C]=\Stab^+_{\M(N)}[C]$ to be the subgroup of $\Stab[C]$ consisting of the isotopy classes of homeomorphisms fixing each curve $\gamma_i$, preserving its orientation, and preserving its sides if $\gamma_i$ is two-sided. We have an exact sequence
\begin{equation}\label{stab_es}
1\to\Stab^+[C]\to\Stab[C]\stackrel{\eta}{\to}\mathbb{Z}_2^{m+r},
\end{equation}
where $\eta(h)$ is the vector $(e_i)_{i=1}^{m+r}$ defined as follows
\begin{itemize}
\item for $i=1,\dots,m$, $e_i=0$ if $h$  preserves orientation of $\gamma_i$ and $e_i=1$ otherwise,
\item for $j=1,\dots,r$, $e_{m+j}=0$ if $h$  preserves  sides of $\gamma_j$ and $e_{r+j}=1$ otherwise.
\end{itemize}
\begin{rem}\label{eta_onto}
The map $\eta$ is not surjective in general and its image depends on $C$. 
For example, if $N_C$ is orientable and $m>1$ then $\eta$ is not onto.
Indeed, suppose that $h\in\Stab[C]$ preserves sides of $\gamma_j$ for 
$j=1,\dots, r$. Then since $N_C$ is orientable, $h$ either preserves orientation of each $\gamma_i$ or reverses orientation of each $\gamma_i$.
On the other hand, we leave it as an exercise for the reader to check that if $N_C$ is nonorientable then $\eta$ is surjective.
\end{rem}

The gluing map $N_C\to N$ induces a surjective homomorphism 
\[\rho_C\colon\M(N_C)\to\Stab^+[C].\] For $i=1,\dots,r$ let $\delta_i$, $\delta'_i$ be the boundary components of  a regular neighbourhood $A_i$ of $\gamma_i$. Note that if $T_{\delta_i}$, $T_{\delta'_i}$ are right Dehn twists with respect to some orientation of $A_i$, then $T_{\delta_i}^{-1}T_{\delta'_i}\in\ker\rho_C$. For $i=r+1,\dots,m$ let $\epsilon_i$ be the boundary curve of a regular neighbourhood $M_i$ of $\gamma_i$ and note that $T_{\epsilon_i}\in\ker\rho_C$. By \cite[Lemma 4.1]{Szep_Osaka} $\ker\rho_C$ is the free abelian group of rank $m$ generated by
$T_{\delta_i}^{-1}T_{\delta'_i}$ for $i=1,\dots,r$ and $T_{\epsilon_j}$ for $i=r+1,\dots,m$. Summarising, we have the following exact sequence
\begin{equation}\label{cut_es}
1\to\Z^m\to\M(N_C)\stackrel{\rho_C}{\to}\Stab^+[C]\to 1.
\end{equation}
Suppose that $N_C$ is nonorientable and $C$ consists entirely of one-sided curves ($r=0$).  Let $N'$ be the surface obtained by cutting $N$ along  $\gamma_i$ and gluing a disc with a puncture $P_i$ along the resulting boundary component for $i=1,\dots,m$. Note that $N$  may be seen as being obtained from $N'$ by blowing up the punctures $\mathcal{P}_m=\{P_1,\dots,P_m\}$, and we have the blowup homomorphism $\mathfrak{b}\colon\PM(N',\mathcal{P}_m)\to\M(N)$, whose image is contained in $\Stab[C]$.
\begin{lemma}\label{blow_stab}
$\mathfrak{b}\colon\PM(N',\mathcal{P}_m)\to\Stab_{\M(N)}[C]$ is an isomorphism.
\end{lemma}
\begin{proof}
Since $\mathfrak{b}$ is injective by Proposition \ref{blowup_inj}, it suffices to show  that its image is equal to $\Stab_{\M(N)}[C]$.
It follows immediately from the definitions  that 
$\rho_{C}=\mathfrak{b}\circ\imath_\ast$, where $\imath_\ast\colon\M(N_{C})\to\M^+(N',\mathcal{P}_m)$ is the map induced by the inclusion of $N_{C}$ in $N'$. Thus the image of $\mathfrak{b}$  contains $\Stab^+[C]$. As $N'$ is nonorientable by assumption, $\pi_1(N'\backslash(\mathcal{P}_m\backslash\{P_i\}),P_i)$ contains a homotopy class of one-sided loops, whose image under the crosscap pushing map $\mathfrak{c}\colon\pi_1(N'\backslash(\mathcal{P}_m\backslash\{P_i\}),P_i)\to\M(N)$ is a crosscap slide reversing the orientation of $\gamma_i$ and equal to the identity on $\gamma_j$ for $j\ne i$ . It follows that the image of $\mathfrak{b}$ is equal to $\Stab{[C]}$.
\end{proof}
Let us identify $N_{g-1,n}$ with the surface obtained from $N'$ by blowing up $\mathcal{P}_{m-1}$, and
by abuse of notation, treat $C'=(\gamma_1,\dots,\gamma_{m-1})$ as a generic
$(m-1)$-tuple of disjoint curves on $N_{g-1,n}$. 
Consider the following commutative diagram
\[
\begin{CD}
\pi_1(N'\backslash\mathcal{P}_{m-1},P_m) @>\mathfrak{p}>> \PM(N',\mathcal{P}_{m}) @>\mathfrak{f}>> \PM(N',\mathcal{P}_{m-1})\\
 @| @VV\mathfrak{b}V @VV\mathfrak{b}V \\
\pi_1(N'\backslash\mathcal{P}_{m-1},P_m) @>\mathfrak{c}>> \Stab_{\M(N_{g,n})}[C] @>\zeta>> \Stab_{\M(N_{g-1,n})}[C'] 
\end{CD}\]
whose top row is a part of the Birman exact sequence (\ref{Bir_es}), $\mathfrak{c}$ is the crosscap pushing map and $\zeta=\mathfrak{b}\circ\mathfrak{f}\circ\mathfrak{b}^{-1}$. As the vertical maps are isomorphisms, exactness of (\ref{Bir_es}) implies exactness of the sequence
\begin{equation}\label{stab_bir_es}
1\to\pi_1(N'\backslash\mathcal{P}_{m-1},P_m)\stackrel{\mathfrak{c}}{\to}\Stab_{\M(N_{g,n})}[C]\stackrel{\zeta}{\to}\Stab_{\M(N_{g-1,n})}[C']\to 1. 
\end{equation}
\subsection{Orbits and a presentation of $\M(N)$.}\label{orbits}
For the rest of this section we fix $g\ge 5$ and $N=N_{g,0}$. Let
$\widetilde{X}$ denote $\D^{ord}(N)$ if $g\ge 7$ or $\C^{ord}_0(N)$ if $g\in\{5,6\}$. Let $X=\widetilde{X}/\M(N)$ and $p\colon\widetilde{X}\to X$ be the canonical projection. Observe that $X$ inherits from $\widetilde{X}$ the structure of a $\Delta$-complex. Let $\cS_m(X)$ (resp. $\cS_m(\widetilde{X})$) be the set of $m$-simplices of $X$ (resp. $\widetilde{X}$). 
The simplices of dimension $0$, $1$ and $2$ will be called {\it vertices}, {\it edges} and {\it triangles} respectively. Observe that the canonical projection $p\colon\widetilde{X}\to X$ induces a surjection $p\colon\cS_m(\widetilde{X})\to\cS_m(X)$.  In the present subsection we will determine a section to $p$ for $m=0,1,2$, that is a map
$s\colon \cS_m(X)\to\cS_m(\widetilde{X})$
such that $p\circ s=identity$. We also describe a presentation of $\M(N)$ obtained by applying Brown's theorem to the action of $\M(N)$ on $\widetilde{X}$.

For $C=(\gamma_1,\dots,\gamma_m)$ and $I\subseteq\{1,\dots,m\}$ let $C_I=(\gamma_i)_{i\in I}$.
The following proposition is a special case of \cite[Proposition 5.2]{Szep_Osaka}.
\begin{prop}\label{orbits_descr}
Two simplices  $[C]=[\gamma_1,\cdots,\gamma_m]$
and $[C']=[\gamma'_1,\cdots,\gamma'_m]$ of $\widetilde{X}$ are $\M(N)$-equivalent if and only if the following two conditions are satisfied.
\begin{itemize}
\item For every $i\in\{1,\dots,m\}$, $\gamma_i$ is one-sided if and only if $\gamma'_i$ is one-sided.
\item For every $I\subseteq\{1,\dots,m\}$, the surface $N_{C_I}$ is orientable if and only if $N_{C'_I}$ is orientable.\hfill{$\Box$}
\end{itemize}
\end{prop}
Note that the second condition is vacuous for $\widetilde{X}=\D^{ord}(N)$, that is if $g\ge 7$. As an immediate corollary we see that every vertex of $\widetilde{X}$ is $\M(N)$-equivalent to one of the following  (see Subsection \ref{notation} for definitions).
\begin{itemize}
\item $[\alpha_1]$ -- two-sided curve with a non-orientable complement, 
\item $[\mu_g]$ -- one-sided curve with a non-orientable complement,
\item $[\xi]$ -- curve with an orientable complement, one-sided for odd $g$ or two-sided for even $g$. 
\end{itemize}
We define
\begin{align*}
&v_1=p[\alpha_1],\ v_2=p[\mu_g],\ v_3=p[\xi],\qquad
s(v_1)=[\alpha_1],\ s(v_2)=[\mu_g],\ s(v_3)=[\xi].
\end{align*}
Note that $\cS_0(X)=\{v_1, v_2\}$ if $g\ge 7$ or $\cS_0(X)=\{v_1, v_2, v_3\}$ if $g\in\{5,6\}$. 
\begin{table}\caption{\label{tabE} The edges of $X$.}
\begin{tabular}{|c|c|c|c|c|c|}
\hline
$e$&$s(e)$&$s(t(e))$&$h_e$&$N_{s(e)}$&$g$\\
\hline
$e_1$&$[\alpha_1,\mu_g]$&$[\mu_g]$&$1$&$N_{g-3,3}$&$\ge 5$\\
\hline
$e_2$&$[\alpha_1,\alpha_3]$&$[\alpha_1]$& $a_2a_3a_1a_2$ &$N_{g-4,4}$&$\ge 5$\\
\hline
$e_3$&$[\mu_g,\mu_{g-1}]$&$[\mu_g]$&$a_{g-1}^{-1}$&$N_{g-2,2}$&$\ge 5$\\
\hline
$e_4$&$[\alpha_1,\xi]$&$[\xi]$&$1$&$S_{1,g-2}$& $5,6$\\
\hline
$e_5$&$[\mu_5,\beta_1]$&$[\alpha_1]$&$a_4ba_3a_4a_2a_3a_1a_2$&$S_{1,3}$& $5$\\
\hline
$e_6$&$[\alpha_1,\gamma_{\{3,4,5,6\}}]$&$[\alpha_1]$&$a_2ca_1a_2$&$S_{1,4}$& $6$\\
\hline
$e_7$&$[\mu_6,\gamma_{\{1,2,3,4,5\}}]$&$[\mu_6]$&$b_2^{-1}$&$S_{2,2}$& $6$\\
\hline
\end{tabular}
\end{table}
If $e$ is an edge of $X$ or $\widetilde{X}$, then we denote by 
$i(e)$ and $t(e)$ its initial and terminal vertices respectively, and by
$\overline{e}$ the edge with the same vertices as $e$ but with the opposite orientation. 
We define edges $e_i\in \cS_1(X)$ for $i\in\{1,\dots,7\}$ as $e_i=p(s(e_i))$, where $s(e_i)$ are defined in the second column of Table \ref{tabE}.
\begin{prop}\label{edges_descr}
If $g\ge 7$ then $\cS_1(X)=\{e_1,\overline{e_1},e_2,e_3\}$;\\
If $g=5$ then $\cS_1(X)=\{e_1,\overline{e_1},e_2,e_3,e_4,\overline{e_4},
e_5,\overline{e_5}\}$;\\
If $g=6$ then $\cS_1(X)=\{e_1,\overline{e_1},e_2,e_3,e_4,\overline{e_4},
e_6,e_7\}$.
\end{prop}
\begin{proof} Let $[C]=[\gamma_1,\gamma_2]$ be any edge of $\widetilde{X}$. If $N_C$ is nonorientable, then it follows easily from Proposition \ref{orbits_descr} that $[C]$ is $\M(
N)$ equivalent to one of the edges
$s(e_1)$, $\overline{s(e_1)}$, $s(e_2)$, or $s(e_3)$. This finishes the proof for $g\ge 7$.
Suppose that $N_C$ is orientable. There are two cases: (1) $N_{(\gamma_1)}$ and $N_{(\gamma_2)}$ are nonorientable;
(2) $N_{(\gamma_1)}$ or $N_{(\gamma_2)}$ is orientable. In the case (1) $C$ is $\M(N)$-eqivalent to one of the edges 
$s(e_5)$, $\overline{s(e_5)}$, $s(e_6)$, or $s(e_7)$. Suppose that we are in case (2) and  $N_{(\gamma_2)}$ is orientable. Since $N_C$ is connected, there is a curve on $N$ disjoint from $\gamma_1$ and intersecting $\gamma_2$ in one point. As such curve must be one-sided, $N_{(\gamma_1)}$ is nonorientable and $[C]$ is $\M(N)$-equivalent to $s(e_4)$.
\end{proof}

The representatives $s(e_i)$ of the edges $e_i$ for $i\in\{1,\dots,7\}$ have been chosen in such a way that $i(s(e_i))=s(i(e_i))$. The elements $h_{e_i}$ defined in the fourth column of Table \ref{tabE} satisfy $h_{e_i}(s(t(e_i))=t(s(e_i))$.
For $i\in\{1,4,5\}$ we define $s(\overline{e_i})=h_{e_i}^{-1}(\overline{s(e_i)})$ and $h_{\overline{e_i}}=h_{e_i}^{-1}$. In this way, for every $e\in \cS_1(X)$ we have
$i(s(e))=s(i(e))$ and $h_{e}(s(t(e))=t(s(e))$. The conjugation map $c_e$ defined as $c_e(x)=h_{e}^{-1}xh_e$ maps $\Stab\,t(s(e))$ onto $\Stab\,s(t(e))$; in particular $c_e(\Stab\,s(e))\subset\Stab\, s(t(e))$.

\begin{table}\caption{\label{tabT} The triangles of $X$.}
\begin{tabular}{|c|c|c|c|c|} 
\hline
$f$&$s(f)$&$\nu_1, \nu_2, \nu_3$&$\varepsilon_1, \varepsilon_2, \varepsilon_3$ &$g$\\
\hline
$f_1$&$[\alpha_1,\alpha_3,\alpha_5]$&$v_1, v_1, v_1$&$e_2, e_2, e_2$& $\ge 6$\\
\hline
$f_2$&$[\alpha_1,\alpha_3,\mu_g]$&$v_1, v_1, v_2$&$e_2, e_1, e_1$& $\ge 5$ \\
\hline
$f_3$&$[\alpha_1,\mu_g,\mu_{g-1}]$&$v_1, v_2, v_2$&$e_1, e_3, e_1$& $\ge 5$ \\
\hline
$f_4$&$[\mu_g,\mu_{g-1},\mu_{g-2}]$&$v_2, v_2, v_2$&$e_3, e_3, e_3$& $\ge 5$ \\
\hline
$f_5$&$[\alpha_1,\alpha_3,\xi]$&$v_1, v_1, v_3$&$e_2, e_4, e_4$&
$5,6$\\
\hline
$f_6$&$[\mu_6,\mu_5,\beta]$&$v_2, v_2, v_1$&$e_3, \overline{e_1}, \overline{e_1}$& $6$\\
\hline
$f_7$&$[\mu_5,\mu_4,\gamma_{\{1,2,3\}}]$&$v_2, v_2, v_2$&$e_3, e_3, e_3$& $5$\\
\hline
$f_8$&$[\alpha_1,\mu_6,\gamma_{\{1,\dots,5\}}]$&$v_1, v_2, v_2$&$e_1, e_7, e_1$&
$6$\\
\hline
$f_9$&$[\alpha_1,\mu_5,\beta]$&$v_1, v_2, v_1$&$e_1, e_5, e_2$& 
$5$\\
\hline
$f_{10}$&$[\alpha_1,\alpha_3,\gamma_{\{3,4,5,6\}}]$&$v_1, v_1, v_1$&$e_2, e_2, e_6$& 
$6$\\
\hline
\end{tabular}
\end{table}

\medskip

Suppose that $\widetilde{f}=[\gamma_1,\gamma_2,\gamma_3]\in\cS_2(\widetilde{X})$ and $f=p(\widetilde{f})\in \cS_2(X)$. For a permutation $\sigma\in\Sym_3$ we define
$\widetilde{f}^\sigma=[\gamma_{\sigma(1)},\gamma_{\sigma(2)},\gamma_{\sigma(3)}]$
and $f^\sigma=p[\widetilde{f}^\sigma]$. We say that $\widetilde{f}^\sigma$ (resp. $f^\sigma$ ) is a permutation of $\widetilde{f}$ (resp. $f$). We also define
$\varepsilon_1(f)=p[\gamma_1,\gamma_2]$, 
$\varepsilon_2(f)=p[\gamma_2,\gamma_3]$, 
$\varepsilon_3(f)=p[\gamma_1,\gamma_3]$, and
$\nu_i(f)=p[\gamma_i]$ for $i=1,2,3$. 

We define triangles $f_i\in \cS_2(X)$ for $i\in\{1,\dots,10\}$ as $f_i=p(s(f_i))$, where $s(f_i)$ are defined in the second column of Table \ref{tabT}.
\begin{prop}\label{triangles_descr}
Every triangle of $X$ is a permutation of  $f_i$ for some
$i\in\{1,\dots,10\}$.
\end{prop} 
\begin{proof}
Suppose that $f=p[C]$ for $C=(\gamma_1,\gamma_2,\gamma_3)$. 
If $N_C$ is nonorientable, then by Proposition \ref{orbits_descr}, $f$ is determined up to permutation by the number of one-sided  vertices. It follows that $f$ is a permutation of one of the triangles: $f_1$ (if $g\ge 7$), $f_2$ (if $g\ge 6$), $f_3$ or $f_4$. We assume that $N_C$ is orientable, $g\in\{5,6\}$. There are three cases.

Case 1: $N_{(\gamma_i,\gamma_j)}$ are nonorientable for all $1\le i<j\le 3$. Then again $f$ is determined up to permutation by the number of one-sided  vertices, and it is a permutation of one of the triangles: $f_1$, $f_2$, $f_6$ or $f_7$.

Case 2: $N_{(\gamma_i)}$ is orientable for some $i\in\{1,2,3\}$. Assume $i=3$. By  the same argument as in the proof of Proposition \ref{edges_descr} (case (2)), 
$N_{(\gamma_1,\gamma_2)}$ is nonorientable  and $f=f_5$ by Proposition \ref{orbits_descr}.

Case 3: $f$ contains one of the edges: $e_5$, $e_6$ or $e_7$. Assume $p[\gamma_1,\gamma_2]=e_i$ for $i\in\{5,6,7\}$. Since $N_{(\gamma_1,\gamma_2)}$ is orientable, thus $\gamma_3$ is two-sided and it is easy to see, 
by similar argument as in the proof of Proposition \ref{edges_descr} (case (2)), 
that $N_{(\gamma_i,\gamma_3)}$ are nonorientable  for $i=1,2$. It follows that $f$ is a permutation of one of the triangles $f_8$, $f_9$ or $f_{10}$.
\end{proof}
\begin{figure}
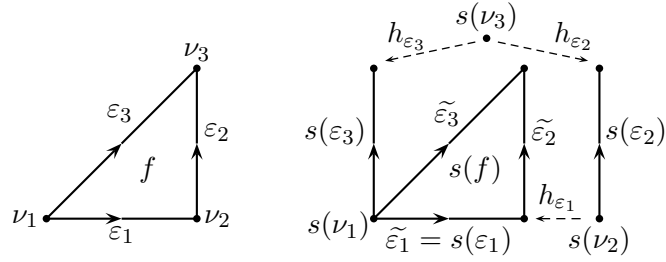

\begin{tabular}{cc}
\input{fig7}&\input{fig8}
\end{tabular}
\caption{\label{tr}A triangle in $X$ and its representative in $\widetilde{X}$.}
\end{figure}

Let $f=f_i$ for some $i\in\{1,\dots, 10\}$. For $j=1,2,3$ we let
$\nu_j=\nu_j(f)$, $\varepsilon_j=\varepsilon_j(f)$ and define $\widetilde{\varepsilon_j}$ to be the edge of $s(f)$ such that
$p(\widetilde{\varepsilon_j})=\varepsilon_j$ (Figure \ref{tr}).
The representatives $s(f)$ have been chosen in such a way that
$\widetilde{\varepsilon_1}=s(\varepsilon_1)$. For $j=1,2,3$ we choose 
$x_j=x_j(f)\in\Stab\,s(\nu_j)$  such that
\begin{equation}\label{the_x_i}
x_1(s(\varepsilon_3))=\widetilde{\varepsilon_3},\quad 
h_{\varepsilon_1}x_2(s(\varepsilon_2))=\widetilde{\varepsilon_2},\quad   
h_{\varepsilon_1}x_2h_{\varepsilon_2}x_3h^{-1}_{\varepsilon_3}=x_1. 
\end{equation}
For $\sigma\in\Sym_3$ we define $s(f^\sigma)=z_\sigma(s(f))^\sigma$ and $x_j(f^{\sigma})$ according to the following table:
\begin{center}
\begin{tabular}{|c|c|c|c|c|c|} 
\hline
$\sigma$&$(1,2)$&$(1,3)$&$(2,3)$&$(1,2,3)$&$(1,3,2)$\\
\hline
$z_\sigma$&$h_{\varepsilon_1}^{-1}$&$h_{\varepsilon_2}^{-1}x_2^{-1}h_{\varepsilon_1}^{-1}$&$x_1^{-1}$&$x_2^{-1}h^{-1}_{\varepsilon_1}$&$h^{-1}_{\varepsilon_3}x_1^{-1}$\\
\hline
$x_1(f^\sigma)$&$x_2$&$x_3$&$x_1^{-1}$&$x_2^{-1}$&$x_3^{-1}$\\
\hline
$x_2(f^\sigma)$&$x_1$&$x_2^{-1}$&$x_3^{-1}$&$x_3$&$x_1^{-1}$\\
\hline
$x_3(f^\sigma)$&$x_3^{-1}$&$x_1$&$x_2^{-1}$&$x_1^{-1}$&$x_2$\\
\hline
\end{tabular}\end{center}
In this way the equations  $\widetilde{\varepsilon_1}=s(\varepsilon_1)$ and (\ref{the_x_i}) are satisfied for every $f\in \cS_2(X)$. We check this for $\sigma=(1,2)$, the other cases can be checked similarly. Let $f'=f^{(1,2)}$,
and for $j=1,2,3$, $x_j'=x_j(f')$, $\varepsilon_j'=\varepsilon_j(f')$.
We have
$\varepsilon_1'=\overline{\varepsilon_1}$, $\varepsilon_2'=\varepsilon_3$,
$\varepsilon_3'=\varepsilon_2$, the edges of $s(f')$ are $\widetilde{\varepsilon_1'}=h^{-1}_{\varepsilon_1}(\overline{\widetilde{\varepsilon_1}})$,
$\widetilde{\varepsilon_2'}=h^{-1}_{\varepsilon_1}(\widetilde{\varepsilon_3})$,
$\widetilde{\varepsilon_3'}=h^{-1}_{\varepsilon_1}(\widetilde{\varepsilon_2})$,
and $s(\varepsilon_1')=s(\overline{\varepsilon_1})=h^{-1}_{\varepsilon_1}(\overline{s(\varepsilon_1}))=
\widetilde{\varepsilon_1'}$,  
\begin{align*}
&x'_1(s(\varepsilon_3'))=x_2(s(\varepsilon_2))=h_{\varepsilon_1}^{-1}(\widetilde{\varepsilon_2})=
\widetilde{\varepsilon_3'},\\
&h_{\varepsilon_1'}x_2'(s(\varepsilon_2'))=
h_{\varepsilon_1}^{-1}x_1(s(\varepsilon_3))=h_{\varepsilon_1}^{-1}(\widetilde{\varepsilon_3})=
\widetilde{\varepsilon_2'},\\  
&h_{\varepsilon_1'}x_2'h_{\varepsilon_2'}x_3'h^{-1}_{\varepsilon_3'}=
h_{\varepsilon_1}^{-1}x_1h_{\varepsilon_3}x_3^{-1}h^{-1}_{\varepsilon_2}=x_2=
x'_1.
\end{align*}
The following theorem is a special case of a general result of Brown \cite{Br} (cf. \cite[Theorem 6.3]{Szep_Osaka}).
\begin{theorem}\label{Brown}
Suppose that:
\begin{itemize}
\item[(1)] for each $v\in \cS_0(X)$ the stabiliser $\Stab\,s(v)$ admits a presentation $\lr{S_v\,|\,R_v}$,
\item[(2)] for each $e\in \cS_1(X)$ the stabiliser $\Stab\,s(e)$ is generated by $G_e$.
\end{itemize} 
Then $\M(N)$ admits a presentation with generators 
\[\bigcup_{v\in \cS_0(X)}S_v \cup\{h_e\, |\,e\in \cS_1(X)\}\]
and relations:
\begin{itemize}
\item $\bigcup_{v\in \cS_0(X)}R_v$,
\item $h_{e_1}=1$ and (if $g\in\{5,6\}$) $h_{e_4}=1$,
\item $h_e^{-1}\imath_e(x)h_e=c_e(x)$ for $e\in \cS_1(X)$ and $x\in G_e$, where
$\imath_e\colon\Stab\,s(e)\to \Stab\,s(i(e))$ is the inclusion and
$c_e\colon\Stab\,s(e)\to \Stab\,s(t(e))$ is the conjugation map defined above,
\item $h_{\varepsilon_1(f)}x_2(f)h_{\varepsilon_2(f)}x_3(f)h^{-1}_{\varepsilon_3(f)}=x_1(f)$ for $f\in \cS_2(X)$.\hfill{$\Box$}
\end{itemize}
\end{theorem}
Recall that in order to prove Theorem \ref{mainB} we want to define a homomorphism
$\psi\colon\M(N_{g,0})\to\G_{g,0}$, which will be the inverse of $\varphi_{g,0}$. 
We will use the presentation  from  Theorem \ref{Brown} and 
define $\psi$ on the generators and prove that it respects the defining relations.
To this end we need presentations of $\Stab\,s(v)$ and generators of $\Stab\,s(e)$  which will be obtained by using the exact sequences defined in the previous subsection and the induction hypothesis. 
\section{The stabiliser of $[\mu_g]$.}\label{sec_v2}
In this section we are assuming  $n\in\{0,1\}$,  $g+n\ge 4$ and that Theorems \ref{mainA} and \ref{mainB} are true for $g-1$. 
\begin{figure}
\input{fig9}
\caption{\label{loop} The loop $\eta_i$.}
\end{figure}
\begin{theorem}\label{pres_Stab_mu_g}
The stabiliser $\Stab[\mu_g]=\Stab_{\M(N_{g,n})}[\mu_{g}]$ is generated by $u_i$, $a_i$, $b_j$ for $1\le i\le g-2$, 
$2\le 2j\le g-3$ and $a_{g-1}u_{g-1}$.
There is a homomorphism $\psi_{v_2}\colon\Stab[\mu_g]\to\G_{g,n}$ such that
$\varphi_{g,n}\circ\psi_{v_2}=\mathrm{id}_{\Stab[\mu_g]}$ and $\psi_{v_2}(x)=x$ for each generator $x$ of $\Stab[\mu_g]$. 
\end{theorem}
\begin{proof}
We will obtain a presentation of $\Stab[\mu_g]$ by applying Lemma \ref{ext_pres} to the exact sequence (\ref{stab_bir_es}), which in this case is
\[
1\to\pi_1(N_{g-1,n},P)\stackrel{\mathfrak{c}}{\to}\Stab[\mu_g]\stackrel{\zeta}{\to}\M(N_{g-1,n})\to 1, 
\]
where we assume that $N_{g-1,n}$ was obtained by cutting $N_{g,n}$ along $\mu_{g}$ and then gluing a disc with puncture $P$ along the resulting boundary component. 

The kernel $\pi_1(N_{g-1,n},P)$ is generated by the homotopy classes of the loops $\eta_i$ in Figure \ref{loop} for $i=1,\dots,g-1$.
Let $\sigma_i=\mathfrak{c}[\eta_i]=Y_{\mu_g,\gamma_{\{i,g\}}}$. 
If $n=1$ then $\pi_1(N_{g-1,n},P)$ is free, while if $n=0$, then there is a single kernel relation: \[(\mathrm{K})\qquad \sigma_{g-1}^2\cdots \sigma_1^2=1.\]
By the induction hypothesis $\M(N_{g-1,n})$ admits the presentation given in Theorem \ref{mainA} if $n=1$ or \ref{mainB} if $n=0$. In the latter case we replace the relation (D) by (Da) (see Lemma \ref{Da_replace}). 
For the cokernel generators we take $u_i, a_i$ and $b_j$ for $1\le i\le g-2$, 
$0\le 2j\le g-3$. Observe that all defining relations of $\M(N_{g-1,n})$ are satisfied also in $\M(N_{g,n})$, except  (B4) and (Da) if $n=0$, in which case we have instead
\begin{align*}
&(\widetilde{\textrm{B4}})\ (u_1\cdots u_{g-3})^{g-2}=\sigma_{g-1}^2\\
&(\widetilde{\textrm{Da}})\ a_{g-2}(u_{g-3}\cdots u_1a_1\cdots a_{g-3})a_{g-2}(u_{g-3}\cdots u_1a_1\cdots a_{g-3})^{-1}=\sigma_{g-1}^2\sigma_{g-2}\sigma_{g-1}^{-1}
\end{align*}
These relations hold in $\M(N_{g,0})$ because the corresponding relations, with each $\sigma_i$ replaced by $\mathfrak{p}[\eta_i]$, hold in $\M(N_{g-1,0},P)$.  Indeed, by Lemma \ref{push1}, $\mathfrak{p}[\eta_{g-1}^2]$ is equal to a product of two Dehn twists about the boundary curves of a regular neighborhood of a simple loop homotopic to $\eta_{g-1}^2$. One of these curves bounds a 
M\"obius strip, and hence the twist is trivial, while the other curve bounds $N_{g-2,1}$ and the twist is equal to $\Delta^2_{g-2}=(u_1\cdots u_{g-3})^{g-2}$.
Analogously, $\mathfrak{p}[\eta_{g-1}^2\eta_{g-2}\eta_{g-1}^{-1}]$ is equal to the product of Dehn twists about the curves $\alpha_{g-2}$ and $u_{g-3}\cdots u_1a_1\cdots a_{g-3}(\alpha_{g-2})$.

To determine the conjugation relations we have to express
$x\sigma_ix^{-1}$ in terms of the kernel generators for $i=1,\dots,g-1$ and 
each cokernel generator $x$. This can be done by first expressing
$x[\eta_i]$ in the generators of $\pi_1(N_{g-1,n},P)$, and then applying $\mathfrak{c}$ together with Lemma \ref{push2}.
Since every cokernel generator can be expressed in terms of  $u_i$ for $i=1,\dots,g-2$, $a_{g-2}$ and $b$, we only have to use these cokernel generators to produce the conjugation relations. As a result we obtain:  
\begin{align*}
&(1)\ u_i\sigma_{i+1}u_i^{-1}=\sigma_i,\quad
(2)\ u_i\sigma_iu_i^{-1}=\sigma_i^{-2}\sigma_{i+1}\sigma_i^2,\quad
(3)\ u_i\sigma_ju_i^{-1}=\sigma_j \textrm{\ for\ } j\ne i,i+1,\\
&(4)\ a_{g-2}\sigma_{g-1}a_{g-2}^{-1}=\sigma_{g-1}^2\sigma_{g-2},\quad
(5)\ a_{g-2}\sigma_{g-2}a_{g-2}^{-1}=\sigma_{g-2}^{-1}\sigma_{g-1}^{-1}\sigma_{g-2},\\
&(6)\ a_{g-2}\sigma_ja_{g-2}^{-1}=\sigma_j \textrm{\ for\ }j<g-2,\quad
(7)\ b\sigma_jb^{-1}=\sigma_j \textrm{\ for\ }j>4,\\ 
&(8)\ b\sigma_4b^{-1}=\sigma_4\delta,\quad
(9)\ b\sigma_4^2\sigma_3b^{-1}=\sigma_4^2\sigma_3\delta,\quad 
(10)\ b\sigma_4^2\sigma_3^2\sigma_2b^{-1}=\sigma_4^2\sigma_3^2\sigma_2\delta,\\
&(11)\  b\sigma_4^2\sigma_3^2\sigma_2^2\sigma_1b^{-1}=\sigma_4^2\sigma_3^2\sigma_2^2\sigma_1\delta,
\quad\textrm{where\ }\delta=\sigma_4\sigma_3\sigma_2\sigma_1.
\end{align*}
We also have the following relations:
\begin{align*}
&a_3\sigma_{4}a_{3}^{-1}=\sigma^2_4\sigma_3,\quad a_{2}a_{3}\sigma_{4}a_{3}^{-1}a_{2}^{-1}=\sigma_4^2\sigma^2_{3}\sigma_{2}\\
&a_{1}a_{2}a_{3}\sigma_{4}a_{3}^{-1}a_{2}^{-1}a_{1}^{-1}=\sigma_4^2\sigma^2_{3}\sigma_{2}^2\sigma_{1},\quad
a_j\delta a_j^{-1}=\delta,\textrm{\ for\ }j=1,2,3, 
\end{align*}
but since every $a_i$ can be expressed in terms of $a_{g-2}$ and $u_j$'s by (C2), these relations are consequences of (C2), (1--6) and  the kernel relation (K).
Since $b$ commutes with $a_j$ for $j=1,2,3$ (A3), the relations above together with (8) imply (9, 10, 11). This shows that (9, 10, 11) are redundant, they follow from other relations. 
For $i>1$, if we conjugate both sides of (2) by $u_iu_{i-1}$, then by using (B2), (1), (3) we obtain the relation $u_{i-1}\sigma_{i-1}u_{i-1}^{-1}=\sigma_{i-1}^{-2}\sigma_i\sigma_{i-1}^2$. If follows that we only need (2) with $i=g-2$. Therefore we replace (2) by
\[\textrm{(2')\ }u_{g-2}\sigma_{g-2}u_{g-2}^{-1}=\sigma_{g-2}^{-2}\sigma_{g-1}\sigma_{g-2}^2.\]
We claim that we can also replace (3) by
\[\textrm{(3')\ }u_i\sigma_{g-1}u_i^{-1}=\sigma_{g-1}\quad\mathrm{for\ }i\le g-3.\]
Indeed, if we set $x=u_ju_{j+1}\cdots u_{g-2}$, then  $x\sigma_{g-1}x^{-1}=\sigma_j$ by (1), for $i<j-1$ we have $xu_ix^{-1}=u_i$ by (B1), and for $i>j$ we have $xu_{i-1}x^{-1}=u_i$ by (B1,B2). Thus (3) follows from (3') by applying conjugation by $x$.
Similarly, it can  be easily proved, using (1) and (C1a, C7a), that (6, 7) can be replaced by
\[\textrm{(6')\ }a_{g-2}\sigma_{g-3}a_{g-2}^{-1}=\sigma_{g-3}\qquad
\textrm{(7')\ }b\sigma_{g-1}b^{-1}=\sigma_{g-1}\textrm{\ if\ }g>5.\]
We have $\sigma_{g-1}=y_{g-1}=a_{g-1}u_{g-1}$, and $\sigma_i=(u_i\cdots u_{g-2})a_{g-1}u_{g-1}(u_i\cdots u_{g-2})^{-1}$ for $i<{g-1}$ by (1). 
It follows that $\Stab[\mu_g]$ is generated by the elements listed in the theorem. To prove that the mapping $\psi_{v_2}(x)=x$ for $x$ a generator of $\Stab[\mu_g]$ extends to  a homomorphism, we have to check that (K, $\widetilde{\textrm{B4}}$, 
$\widetilde{\textrm{Da}}$, 2', 3', 4, 5, 6', 7', 8) are satisfied in $\G_{g,n}$.

By (1), the kernel relation (K) can be rewritten using only the generators $u_i$ and $\sigma_{g-1}^2$. Since $\sigma_{g-1}^2=(a_{g-1}u_{g-1})^2=u_{g-1}^2$ by (C4a), (K) is satisfied in $\G_{g,n}$ by Lemma \ref{shortcut_AB}, and so is ($\widetilde{\textrm{B4}}$).
By (Da) in $\G_{g,n}$ we have
\begin{align*}
&a_{g-1}^{-1}=(u_{g-2}\cdots u_1a_1\cdots \underline{a_{g-2})a_{g-1}(a_{g-2}^{-1}}\cdots a_1^{-1}u_1^{-1}\cdots u_{g-2}^{-1})\stackrel{(B1,B2,C1)}=\\
&u_{g-2}a_{g-1}^{-1}(u_{g-3}\cdots u_1a_1\cdots a_{g-3})a_{g-2}(a_{g-3}^{-1}\cdots a_1^{-1}u_1^{-1}\cdots u_{g-3}^{-1})a_{g-1}u_{g-2}^{-1}
\end{align*}
and after substitution ($\widetilde{\textrm{Da}}$) becomes
\[a_{g-2}(u_{g-2}a_{g-1}^{-1})^{-1}a_{g-1}^{-1}(u_{g-2}a_{g-1}^{-1})=\sigma_{g-1}^2\sigma_{g-2}\sigma_{g-1}^{-1}.\] 
Since $\sigma_{g-1}$ and $\sigma_{g-2}$ can be expressed in the generators
$a_{g-1}$, $a_{g-2}$, $u_{g-1}$, $u_{g-2}$, the last relation holds in $\G_{g,n}$ by Corollary \ref{genus3_shortcut}  and so do (2', 4, 5). (3') follows from (B1, C1a), (6') follows from (C3, A1, C1a), (7') follows from 
(A3, C7a).
By (B1, B2, C1a, C2) we have
\begin{align*}
&\sigma_4=(u_4\cdots u_{g-2})a_{g-1}u_{g-1}(u_4\cdots u_{g-2})^{-1}=
(u_{5}\cdots u_{g-1})^{-1}a_4u_4(u_{5}\cdots u_{g-1})\\
&\delta=(u_5\cdots u_{g-1})^{-1}a_4a_3a_2a_1u_1u_2u_3u_4(u_5\cdots u_{g-1})
\end{align*}
and we see that (8) follows from (C8) and the fact that $(u_{5}\cdots u_{g-1})$ commutes with $b$ (C7a). 
Thus $\psi_{v_2}$ is a homomorphism and obviously  $\varphi_{g,n}\circ\psi_{v_2}=\mathrm{id}_{\Stab[\mu_g]}$.
\end{proof}
\begin{lemma}\label{relC9}
If $g\ge 5$ then the following relation holds in $\G_{g,1}$.
\[(\mathrm{C9})\quad b(a_4a_3a_2a_1u_1u_2u_3u_4)=(a_4a_3a_2a_1u_1u_2u_3u_4)b.\]
\end{lemma}
\begin{proof}
By (C8) we have 
$a_4a_3a_2a_1u_1u_2u_3u_4=(a_4u_4)^{-1}b(a_4u_4)b^{-1}$. (C9) is satisfied in $M(N_{g,1})$ because $(a_4u_4)^{-1}b(a_4u_4)$ is a Dehn twist about the curve $(a_4u_4)^{-1}(\beta)$, which is disjoint from $\beta$ up to isotopy. Since $b$, $a_4u_4\in\Stab[\mu_g]$, (C9) also holds in $\G_{g,1}$ by Theorem \ref{pres_Stab_mu_g}.
\end{proof}
We define $\mathcal{S}_{g,n}(v_2)$ to be the image of $\psi_{v_2}$ in $\G_{g,n}$.  By Theorem \ref{pres_Stab_mu_g} $\psi_{v_2}$ is an isomorphism onto $\mathcal{S}_{g,n}(v_2)$, whose inverse is the restriction of $\varphi_{g,n}$.

\begin{lemma}\label{Delta2_central}
$\Delta_g^2$ is central in $\G_{g,1}$.
\end{lemma}
\begin{proof}
Since $u_{g-1}^2=(a_{g-1}u_{g-1})^2$ by (C4a), thus
$u_{g-1}^2\in \mathcal{S}_{g,n}(v_2)$. 
By Lemma \ref{Delta_in_stab},  
$\Delta^2_g\in\mathcal{S}_{g,1}(v_2)$. Since $\varphi_{g,1}(\Delta_g^2)$ is a Dehn twist about the boundary of $N_{g,1}$, thus it is central in $\M(N_{g,1})$, and since the restriction of $\varphi_{g,1}$ to $\mathcal{S}_{g,1}(v_2)$ is an isomorphism, thus $\Delta_g^2$ is central in $\mathcal{S}_{g,1}(v_2)$. Note that
$\G_{g,1}$ is generated by $\mathcal{S}_{g,1}(v_2)$, $u_{g-1}$, and if $g=4$ then also $b$. By (B5) $\Delta_g^2$ commutes with $u_{g-1}$ and it remains to prove that it commutes with $b$ if $g=4$. By (C6a) $b$ commutes with $(a_1a_2a_3)^2(u_1u_2u_3)^2$, and by (A3) it commutes with $(a_1a_2a_3)^2$, hence it commutes with $(u_1u_2u_3)^2$ and with $\Delta_4^2=(u_1u_2u_3)^4$.
\end{proof}
\begin{cor}\label{five}
$\varphi_{g,1}$ is an isomorphism if and only if $\varphi_{g,0}^1$ is an isomorphism.\hfill{$\Box$} 
\end{cor}
\begin{proof}
By Lemma \ref{Delta2_central} the normal closure of $\Delta_g^2$ is a cyclic subgroup of $\G_{g,1}$. Moreover, as $\ker\imath_\ast$ is infinite cyclic, thus the restriction of $\varphi_{g,1}$ to the subgroup of $\G_{g,1}$ generated by $\Delta_g^2$ is an isomorphism onto $\ker\imath_\ast$. We have the following commutative diagram.
\[
\begin{CD}
1 @>>> \Z @>>>  \G_{g,1} @>>> \G_{g,0}^1 @>>> 1\\
@. @| @VV\varphi_{g,1}V @VV\varphi_{g,0}^1V \\
1 @>>> \Z @>>> \M(N_{g,1}) @>\imath_\ast>> \M^+(N_{g,0},P) @>>> 1 
\end{CD}\]
Since the rows are exact, the  corollary follows from the five lemma.
\end{proof}
 Let $\mathcal{S}_{g,0}^1(v_2)$ be the image of  
$\mathcal{S}_{g,1}(v_2)$ under the canonical projection $p\colon\G_{g,1}\to\G_{g,0}^1$.
Since $\Stab_{\M^+(N_{g,0},P)}[\mu_g]=\imath_\ast(\Stab_{\M(N_{g,1})}[\mu_g])$, there is an isomorphism \[\psi_{v_2}^1\colon \Stab_{\M^+(N_{g,0},P)}[\mu_g]\to\mathcal{S}_{g,0}^1,\] such that
$\imath_\ast\circ\psi_{v_2}^1=\psi_{v_2}\circ p$, and whose inverse is the restriction of $\varphi_{g,0}^1$.
\section{Proof of Theorem \ref{mainA}.}\label{sec_mainA}
In this section we assume that $g\ge 4$ is fixed, Theorem \ref{mainB} is true for $g$ and Theorem \ref{mainA} is true for $g-1$. The last assumption implies that Theorem \ref{pres_Stab_mu_g} is true for $g$, as well as Lemma \ref{relC9} and Corollary \ref{five}.
Theorem \ref{mainA} for $g$ will follow from Corollary \ref{five} and the following theorem.
\begin{theorem}\label{mainA_punctured}
$\varphi_{g,0}^1\colon\G_{g,0}^1\to\M^+(N_{g,0},P)$ is an isomorphism.
\end{theorem} 
\begin{proof}
First we will obtain a presentation for $\M(N_{g,0},P)$ by applying Lemma \ref{ext_pres} to the Birman exact sequence (\ref{Bir_es})
\[
1\to\pi_1(N_{g,0},P)\stackrel{\mathfrak{p}}{\to}\M(N_{g,0},P)\stackrel{\mathfrak{f}}{\to}\M(N_{g,0})\to 1 
\]
and the presentation of $\M(N_{g,0})$ given in Theorem \ref{mainB}.
Then we will apply the Reidemeister-Schreier method to find a presentation of $\M^+(N_{g,0},P)$, which is an index $2$ subgroup of $\M(N_{g,0},P)$.
 
To obtain a presentation for $\M(N_{g,0},P)$ we proceed in the same way as we did in the proof of Theorem \ref{pres_Stab_mu_g}, with the following differences: (1) we use the sequence (\ref{Bir_es}) instead of (\ref{stab_bir_es}); (2) in the presentation of  $\M(N_{g,0})$ we use (D) instead of (Da), and we replace the relation (B4) by the equivalent relation 
\[(\textrm{B4b})\, (u_2\cdots u_{g-1})^{g-1}=1,\]
obtained by conjugating (B4) by $\Delta_g$;  (3) to produce the conjugation relations we use the cokernel generators $u_i$ for $i=1,\dots,g-1$, $b$ and  $a_1^{-1}$ (instead of $a_{g-1}$). As a result we obtain a presentation with 
kernel generators $\sigma_i=\mathfrak{p}[\eta_i]$ for $i\in\{1,\dots, g\}$ and 
cokernel generators $u_i, a_i$ and $b_j$ for $1\le i\le g-1$, 
$0\le 2j\le g-2$. There is the single kernel relation
\[(\mathrm{K})\qquad \sigma_{g}^2\cdots \sigma_1^2=1,\]
the cokernel relations are the defining relations of $\M(N_{g,0})$ except for (B4b) and (D), instead of which we have  
\begin{itemize}
\item[($\widetilde{\textrm{B4b}}$)] $(u_2\cdots u_{g-1})^{g-1}=\sigma_{1}^2$ 
\item[($\widetilde{\textrm{D}}$)] $a_1(a_2\cdots a_{g-1}u_{g-1}\cdots u_2)a_1(a_2\cdots a_{g-1}u_{g-1}\cdots u_2)^{-1}=(\sigma_2\sigma_1)^{-1}$
\end{itemize}
By Lemma \ref{relC9}, if $g\ge 5$ then (C9) is a consequence of the cokernel relations.

The conjugation relations are 
\begin{align*}
&(1)\ u_i\sigma_{i+1}u_i^{-1}=\sigma_i,\quad
(2)\ u_i\sigma_iu_i^{-1}=\sigma_i^{-2}\sigma_{i+1}\sigma_i^2,\quad
(3)\ u_i\sigma_ju_i^{-1}=\sigma_j \textrm{\ for\ } j\ne i,i+1,\\
&(4)\ a_1^{-1}\sigma_1a_1=\sigma_2\sigma_1^2,\quad
(5)\ a_1^{-1}\sigma_2a_1=\sigma_2\sigma_1^{-1}\sigma_2^{-1},\quad
(6)\ a_1^{-1}\sigma_ia_1=\sigma_i\textrm{\ for\ }i>2,\\
&(7)\ b^{-1}\sigma_jb=\sigma_j \textrm{\ for\ }j>4,\quad 
(8)\ b^{-1}\sigma_1b=\delta\sigma_1,\quad
(9)\ b^{-1}\sigma_2\sigma_1^2b=\delta\sigma_2\sigma_1^2,\\ 
&(10)\ b^{-1}\sigma_3\sigma_2^2\sigma_1^2b=\delta\sigma_3\sigma_2^2\sigma_1^2,\quad
(11)\  b^{-1}\sigma_4\sigma_3^2\sigma_2^2\sigma_1^2b=\delta\sigma_4\sigma_3^2\sigma_2^2\sigma_1^2,
\quad\textrm{where\ }\delta=\sigma_4\sigma_3\sigma_2\sigma_1.
\end{align*}
For $i>1$ and $j\notin\{i,i+1\}$ we have
\[a_i^{-1}\sigma_ia_i=\sigma_{i+1}\sigma_i^2,\quad a_i^{-1}\sigma_{i+1}a_i=\sigma_{i+1}\sigma_i^{-1}\sigma_{i+1}^{-1},\quad a_i^{-1}\sigma_ja_i=\sigma_j.\]
Since $a_i$ can be expressed in terms of $a_1$ and the $u_k$'s by (C2) or (C3), these relations follow from (C2,C3), (1--6) and (K). As a consequence of the above relations we have
\begin{align*}
&a_1^{-1}\sigma_1a_1=\sigma_2\sigma_1^2,\quad a_2^{-1}a_1^{-1}\sigma_1a_1a_2=\sigma_3\sigma_2^2\sigma^2_1\\
&a_3^{-1}a_2^{-1}a_1^{-1}\sigma_1a_1a_2a_3=\sigma_4\sigma_3^2\sigma^2_2\sigma_1^2,\quad
a_j^{-1}\delta a_j=\delta,\textrm{\ for\ }j=1,2,3. 
\end{align*}
These relations together with (8) and (A3) imply (9,10,11). Hence the last relations are redundant. We also have 
\begin{align*}
&(a_2a_1a_3a_2)^{-1}\sigma_2\sigma_1(a_2a_1a_3a_2)=
(a_1a_3a_2)^{-1}(\sigma_3\sigma_2^2\sigma_1)(a_1a_3a_2)=\\
&(a_3a_2)^{-1}(\sigma_3\sigma_2)(a_3a_2)=
a_2^{-1}(\sigma_4\sigma_3^2\sigma_2)a_2=\sigma_4\sigma_3.
\end{align*}
It follows that (8) can be replaced by
\[\textrm{(8')\ }  b^{-1}\sigma_1b=(a_2a_1a_3a_2)^{-1}\sigma_2\sigma_1(a_2a_1a_3a_2)\sigma_2\sigma_1^2.\] 
Similarly as in the proof of Theorem \ref{pres_Stab_mu_g} it can be proved that  (2,3,6,7) can be replaced by
\begin{align*}
&\textrm{(2')\ } u_1\sigma_1 u_1^{-1}=\sigma_1^{-2}\sigma_2\sigma_1^2,\quad
\textrm{(3')\ } u_i\sigma_1 u_i^{-1}=\sigma_1\textrm{\ for\ }i\ge 2,\\
&\textrm{(6')\ } a_1^{-1}\sigma_3 a_1=\sigma_3,\quad
\textrm{(7')\ } b\sigma_5 b^{-1}=\sigma_5.
\end{align*}
We have
\[a_1^{-1}\sigma_3 a_1\stackrel{(1)}{=}a_1^{-1}u_2^{-1}u_1^{-1}\sigma_1u_1u_2a_1\stackrel{(C3)}{=}u_2^{-1}u_1^{-1}a_2^{-1}\sigma_1a_2u_1u_2.\] 
Therefore we can replace (6') by (6'') $a_2^{-1}\sigma_1 a_2=\sigma_1$.
By (1), the relation $(5)$ is equivalent to
\begin{align*}
&a_1^{-1}u_1^{-1}\sigma_1u_1a_1=u_1^{-1}\sigma_1u_1\sigma_1^{-1}u_1^{-1}\sigma_1^{-1}u_1\stackrel{(C4)}{\Longleftrightarrow}\\
&a_1\sigma_1a_1^{-1}=\sigma_1u_1\sigma_1^{-1}u_1^{-1}\sigma_1^{-1}\stackrel{(1)}{\Longleftrightarrow}
a_1\sigma_1a_1^{-1}=\sigma_1u_1\sigma_1^{-1}\sigma_2^{-1}u_1^{-1} \textrm{\ (5')}.
\end{align*}
Let $z=a_4a_3a_2a_1u_1u_2u_3u_4$. We have 
\[\sigma_5\stackrel{(1)}{=}(u_1u_2u_3u_4)^{-1}\sigma_1(u_1u_2u_3u_4)=
z^{-1}(a_4a_3a_2a_1)\sigma_1(a_4a_3a_2a_1)^{-1}z.\]  Since $z$ commutes with $b$ by (C9), we can replace (7') by
\[\textrm{(7'')\ }  b^{-1}(a_4a_3a_2a_1)\sigma_1(a_4a_3a_2a_1)^{-1}b=(a_4a_3a_2a_1)\sigma_1(a_4a_3a_2a_1)^{-1}.\]
Summarising, we have reduced the conjugation relations to the following ones, which are rewritten in a convenient way.
\begin{itemize}
\item[(R1)] $u_i\sigma_{i+1}u_i^{-1}=\sigma_i$ for $i=1,\dots,g-1$
\item[(R2)] $\sigma_1u_1\sigma_1^{-1}=(\sigma_2\sigma_1)^{-1}\sigma_1^2u_1$
\item[(R3)] $\sigma_1u_i\sigma_1^{-1}=u_i$ for $i\ge 2$
\item[(R4)] $\sigma_1a_1\sigma_1^{-1}=a_1(\sigma_2\sigma_1)$
\item[(R5)] $\sigma_1^{-1}a_1\sigma_1=u_1(\sigma_2\sigma_1)^{-1}u_1^{-1}a_1$
\item[(R6)] $\sigma_1a_2\sigma_1^{-1}=a_2$
\item[(R7)] $\sigma_1(a_4a_3a_2a_1)^{-1}b(a_4a_3a_2a_1)\sigma_1^{-1}=(a_4a_3a_2a_1)^{-1}b(a_4a_3a_2a_1)$
\item[(R8)] $\sigma_1b\sigma_1^{-1}=b(a_2a_1a_3a_2)^{-1}\sigma_2\sigma_1(a_2a_1a_3a_2)\sigma_2\sigma_1$
\end{itemize}
Since all defining relations of $\mathcal{G}_{g,0}^1$ appear as cokernel relations in our presentation, the relations (E2a, E3) from Lemma \ref{useful_T} are consequences of the cokernel relations. Let $r=r_g$ and note that
($\widetilde{\textrm{D}}$) can be rewritten as 
\[\sigma_2\sigma_1=a_1^{-1}ru_1^{-1}a_1^{-1}u_1r^{-1}\stackrel{C4, E2a}{=}a_1^{-1}ra_1r.\]

{\bf Claim 1:} (R7) is redundant.

{\it Proof of Claim 1.} 
By (A1--A4) we have 
\[(a_4a_3a_2a_1)^{-1}b(a_4a_3a_2a_1)=(ba_4a_3a_2)a_1(ba_4a_3a_2)^{-1}\]
thus (R7) is equivalent to 
\[\sigma_1(ba_4a_3a_2)a_1(ba_4a_3a_2)^{-1}\sigma_1^{-1}=(ba_4a_3a_2)a_1(ba_4a_3a_2)^{-1}\]
It is a consequence of (R3), (R6) and (C3) that $\sigma_1$ commutes with $a_2$, $a_3$ and $a_4$. It follows that the relation above is equivalent to
\[
(b^{-1}\sigma_1b\sigma_1^{-1})a_4a_3a_2(\sigma_1a_1\sigma_1^{-1})(a_4a_3a_2)^{-1}(\sigma_1b^{-1}\sigma_1^{-1}b)=(a_4a_3a_2)a_1(a_4a_3a_2)^{-1}\]
Let $L$ denote the left hand side of the last relation.
By ($\widetilde{D}$, R4, R8) we have
\begin{align*}
L&=((a_2a_1a_3a_2)^{-1}a_1^{-1}ra_1r(a_2a_1a_3a_2)a_1^{-1}ra_1r)a_4a_3a_2(ra_1r)\cdot\\
&a_2^{-1}a_3^{-1}a_4^{-1}(ra_1^{-1}ra_1(a_2a_1a_3a_2)^{-1}ra_1^{-1}ra_1(a_2a_1a_3a_2)),
\end{align*}
and since
\begin{align*}
&(a_2a_1a_3a_2)a_1^{-1}ra_1\underline{ra_4a_3a_2r}a_1\underline{ra_2^{-1}a_3^{-1}a_4^{-1}r}a_1^{-1}ra_1(a_2a_1a_3a_2)^{-1}\stackrel{(E2a, E3)}{=}\\
&(a_2a_1a_3a_2)a_1^{-1}r\underline{a_1a_4a_3a_2a_1a_2^{-1}a_3^{-1}a_4^{-1}a_1^{-1}}ra_1(a_2a_1a_3a_2)^{-1}\stackrel{(A1,A2)}{=}\\
&(a_2a_1a_3a_2)a_1^{-1}\underline{ra_2^{-1}a_3^{-1}a_4a_3a_2r}a_1(a_2a_1a_3a_2)^{-1}\stackrel{(E2a, E3)}{=}\\
&(a_2a_1a_3a_2)a_1^{-1}a_2^{-1}a_3^{-1}a_4a_3a_2a_1(a_2a_1a_3a_2)^{-1}\stackrel{(A1,A2)}{=}a_3^{-1}a_4a_3,
\end{align*}
thus
\begin{align*} 
&L=a_2^{-1}a_3^{-1}a_1^{-1}a_2^{-1}\underline{a_1^{-1}ra_1ra_3^{-1}a_4a_3ra_1^{-1}ra_1}a_2a_1a_3a_2\stackrel{(E2a, E3, A1)}{=}\\
&a_2^{-1}a_3^{-1}a_1^{-1}a_2^{-1}a_3^{-1}a_4a_3a_2a_1a_3a_2\stackrel{(A1,A2)}{=}
(a_4a_3a_2)a_1(a_4a_3a_2)^{-1}.
\end{align*}
This ends the proof of Claim 1, which allows us to rule out (R7) from the presentation. 
Then we make the following transformations.

{\bf (1)} By using the relations (R1) and ($\widetilde{\textrm{B4b}}$) we rewrite (K) in the generators $u_i$. By Lemma \ref{shortcut_AB}, (K) can be removed from the presentation.

{\bf (2)} In  (C8) we replace $a_4a_3a_2a_1u_1u_2u_3u_4$ by the right hand side of the following equation, which is a consequence of (B1, B2, C2).
\begin{align*}
&a_4u_4(u_4^{-1}a_3u_3u_4)(u_4^{-1}u_3^{-1}a_2u_2u_3u_4)(u_4^{-1}u_3^{-1}u_2^{-1}a_1u_1u_2u_3u_4)=\\
&a_4u_4(u_3a_4u_4u_3^{-1})(u_2u_3a_4u_4u_3^{-1}u_2^{-1})(u_1u_2u_3a_4u_4u_3^{-1}u_2^{-1}u_1^{-1})
\end{align*}
In this way (C8) is expressed in the generators of $\Stab[\mu_g]$.

{\bf (3)} In (C6) we replace $(a_1a_2a_3)^2(u_1u_2u_3)^2$ by its expression in  the generators of $\Stab[\mu_g]$ resulting from the following equations.
\begin{align*}
&(a_1a_2a_3)^2(u_1u_2u_3)^2\stackrel{(C1a,C3)}{=}a_1a_2a_3a_1a_2(u_1u_2u_3)^2a_1
\stackrel{(A1, B1, B2)}{=}\\
&a_1a_2a_1a_3a_2(u_2u_3u_1u_2u_3^2)a_1=a_1a_2a_1a_3u_3(u_3^{-1}a_2u_2u_3)u_1u_2u_3^2a_1\stackrel{(B2, C2)}{=}\\
&a_1a_2a_1a_3u_3(u_2a_3u_3u_2^{-1})u_1u_2u_3^2a_1\stackrel{(C4a)}{=}a_1a_2a_1(a_3u_3)u_2(a_3u_3)u_2^{-1}u_1u_2(a_3u_3)^2a_1
\end{align*}

{\bf (4)} Using Lemma \ref{Delta_in_stab}, we replace 
$(u_1\cdots u_{g-1})^g$ by \[u_{g-1}^2(u_{g-2}u_{g-1}^2u_{g-2})\cdots(u_1\cdots u_{g-1}^2\cdots u_1)\] in (B3), and 
$(u_2\cdots u_{g-1})^{g-1}$ by
\[u_{g-1}^2(u_{g-2}u_{g-1}^2u_{g-2})\cdots(u_2\cdots u_{g-1}^2\cdots u_2)\]
in ($\widetilde{\textrm{B4b}}$). Note that these are expressions in the generators of
$\Stab[\mu_g]$.

{\bf (5)}
We replace $(\sigma_2\sigma_1)$  by $a_1^{-1}ra_1r$ in (R2--R6, R8) and by
$u_1^{-1}\sigma_1u_1\sigma_1^{-1}\sigma_1^2$ in 
($\widetilde{\textrm{D}}$),
then we rule out the generators $\sigma_i$ for $i>1$ together with (R1).

{\bf (6)} If $g=4$, then we replace the right hand side of (R8) by
$b(a_1a_2a_3)^{-4}$. We have to prove that this yields an equivalent relation.
Let $w=a_1a_2a_3$, $z=u_3u_2u_1$. We have $r=wz$ and
\begin{align*}
ra_1^{-1}ra_1&=wza_2a_3za_1\stackrel{(C1a,C2)}{=}wa_1a_2z^2a_1\stackrel{(B1,B2)}{=}
wa_1a_2(u_3^2u_2u_3u_1u_2)a_1\\
&\stackrel{(C1a,C3)}{=}wa_1a_2a_3(u_3^2u_2u_3u_1u_2)=w^2z^2.
\end{align*}
Hence, the inverse of the right hand side of (R8) times $b$ equals
\begin{align*}
&w^2z^2\underline{(a_2a_1a_3a_2)^{-1}w^2}z^2(a_2a_1a_3a_2)
\stackrel{(A1,A2)}{=}w^2z^2a_3^2z^2(a_2a_1a_3a_2)\\
&\stackrel{(B1,B2)}{=}w^2\underline{(u_3^2u_2u_1u_3u_2)a_3^2}z^2(a_2a_1a_3a_2)
\stackrel{(C1,C2)}{=}w^2a_1^2z^4a_2a_1a_3a_2\\
&\stackrel{(B3)}{=}w^2a_1^2a_2a_1a_3a_2\stackrel{(A1,A2)}{=}(a_1a_2a_3)^4.
\end{align*}

{\bf (7)} In (A9) we replace $b_\rho$ by the expression in the generators $b$ and $a_i$ resulting from (A7,A8).
Then we rule out the generators $b_j$ for $j\ne 1$ together with  (A7, A8).

As a result of these transformations we obtain a presentation with generators
$a_i, u_i$ for $i=1,\dots,g-1$, $b$ and $\sigma_1$, which we will denote simply
as $\sigma$. Until the end of this proof, we will denote $\M(N_{g,0},P)$ as $\M$,   $\M^+(N_{g,0},P)$ as $\M^+$, $\mathcal{G}_{g,0}^1$ as $\G$, and
$\varphi_{g,0}^1$ as $\varphi$.

In the next step we obtain a presentation for $\M^+$ by  the Reidemeister-Shreier method, for which we take $\{1,\sigma\}$ as a transversal of $\M/\M^+$. The generators of $\M^+$ are:
\[b,\, a_i,\, u_i,\, b'=\sigma b\sigma^{-1},\, a_i'=\sigma a_i\sigma^{-1},\, u_i'=\sigma u_i\sigma^{-1},\, \sigma^2,\]
for $i=1,\dots,g-1$. 
Note that every defining relation (Rel) of $\M$ can be regarded as a relation in the generators of  $\M^+$. Conjugating (Rel) by $\sigma$ yields another relation (Rel'), obtained by replacing in (Rel) every 
$x\in\{a_i, u_i, b\}$ by $x'$, and every $x'$ by $\sigma^2x\sigma^{-2}$.
The relations (Rel) and (Rel'), for all defining relations (Rel) of $\M$, are the defining relations of $\M^+$.

Let $G_{\M^+}=\{a_i, u_i, a_i', u_i'\,|\,1\le i\le g-1\}\cup\{b, b', \sigma^2\}$ denote the set of generators of $\M^+$. We define $\psi\colon G_{\M^+}\to\G$  as
\begin{align*}
&\psi(x)=x\quad\textrm{for\ }x\in\{a_i, u_i\,|\,i=1,\dots,g-1\}\cup\{b\},\\ 
&\psi(a'_i)=a_i,\quad \psi(u'_i)=u_i\quad\textrm{for\ }i=2,\dots,g-1,\\ 
&\psi(\sigma^2)=u_{g-1}^2(u_{g-2}u_{g-1}^2u_{g-2})\cdots(u_2\cdots u_{g-1}^2\cdots u_2),\\
&\psi(a'_1)=ra_1r,\\
&\psi(u_1')=ra_1^{-1}ra_1\psi(\sigma^2)u_1,\\
&\psi(b')=b(a_2a_1a_3a_2)^{-1}a_1^{-1}ra_1r(a_2a_1a_3a_2)a_1^{-1}ra_1r
\end{align*}
Observe that for $x\in G_{\M^+}$ we have $\varphi(\psi(x))=x$. In order to show that $\psi$ can be extended to a homomorphism $\psi\colon\M^+\to\G$, we have to show that it respects the defining  relations of 
$\M^+$.

It is obvious that $\psi$ respects the relations
(A1--A6, A9, B1--B3, C1--C8, $\widetilde{\textrm{B4b}}$, R2, R3, R4, R6, R8).

Let $\mathcal{S}=\mathcal{S}_{g,0}^1(v_2)$ and recall that this is the subgroup of $\G$ generated by $a_i$, $u_i$ for $i\in\{1,\dots,g-2\}$, $a_{g-1}u_{g-1}$ and if $g>4$ then also $b$.
Suppose that $w=x_1^{\varepsilon_1}\cdots x_k^{\varepsilon_k}$, where $x_i\in G_{\M^+}$ and $\varepsilon_i\in\{1,-1\}$. We say that $w$ {\it is expressible in} $\mathcal{S}$ if 
$\psi(x_1)^{\varepsilon_1}\cdots \psi(x_k)^{\varepsilon_k}\in\mathcal{S}$, and a relation in $\M^+$ is expressible in $\mathcal{S}$ if its both sides are expressible in $\mathcal{S}$. If $x_1^{\varepsilon_1}\cdots x_k^{\varepsilon_k}=1$ is a relation expressible in $\mathcal{S}$, then since $\varphi(\psi(x_i))=x_i$ and the restriction of  $\varphi$ to $\mathcal{S}$ is an isomorphism, $\psi(x_1)^{\varepsilon_1}\cdots \psi(x_k)^{\varepsilon_k}=1$. Thus $\psi$ respects the relations expressible in $\mathcal{S}$.

 Set
$G_{\cS^+}=\{a_i, u_i, a_i', u_i'\,|\,1\le i\le g-2\}\cup\{\sigma^2, a_{g-1}u_{g-1}, a'_{g-1}u'_{g-1}, u^2_{g-1}, {u'}^2_{g-1}\}$ plus $\{b,b'\}$ if $g>4$. Observe that every element of $G_{\cS^+}$ is expressible in $\cS$, hence every word on $G_{\M+}\cup G_{\M+}^{-1}$ obtained as a concatenation of elements of $G_{\cS+}\cup G_{\cS+}^{-1}$ is expressible in $\cS$. It follows that the following relations are expressible in $\mathcal{S}$, hence are respected by $\psi$:
(R2', R4', R5, R5', R6', $\widetilde{\textrm{D}}$, $\widetilde{\textrm{D}}$', B3', $\widetilde{\textrm{B4b}}$', C4', C5', C8') and (R8') if $g>4$.
(R3') is expressible in $\mathcal{S}$ for $i<g-1$, and for $i=g-1$ it is
$\sigma^2u_{g-1}\sigma^{-2}=u'_{g-1}$. Since $\psi(u'_{g-1})=\psi(u_{g-1})=u_{g-1}$ and $\psi(\sigma^2)$ is a word in the $u_i$'s, thus $\psi$ respects (R3') for $i=g-1$ by Lemma \ref{shortcut_AB}.

\medskip
{\bf Claim 2.} (C6') is expressible in $\mathcal{S}$.

{\it Proof of Claim 2.}
The right hand side of (C6') (after transformation (3)) is a concatenation of elements of $G_{\cS+}\cup G_{\cS+}^{-1}$, thus it is expressible in $\cS$. It suffices to show that the same is true for the left hand side.
There is nothing to do if $g>4$, so we assume $g=4$. 
Let $w=(a_1a_2a_3)$.
\begin{align*}
&(\psi(u_3')\psi(b'))^2=(u_3bw^{-4})^2\stackrel{(A3)}{=}u_3w^{-4}(bu_3)^2u_3^{-1}w^{-4}
\stackrel{(C6a)}{=}
u_3w^{-2}(u_1u_2u_3u_1u_2)w^{-4}\\
&=u_3\underline{w^{-2}\Delta_4} u_1^{-1}w^{-4}\stackrel{(E1)}{=}u_3\Delta_4 \underline{w^{2}u_1^{-1}w^{-4}}
\stackrel{(C1a, C5a)}{=}
u_3\Delta_4 u_3^{-1}w^{-2}=u_3^2u_2u_1u_3u_2w^{-2}\\
&\stackrel{(A1,A2)}{=}
u_3^2u_2u_1(u_3u_2a_2^{-1}a_3^{-1})(a_1a_2)^{-2}=
u_3^2u_2u_1(u_3u_2a_2^{-1}u_3^{-1})u_3a_3^{-1}(a_1a_2)^{-2}\\
&\stackrel{(B2, C3, C4a)}{=}
u_3^2u_2u_1(u_2^{-1}a_3u_3u_2)a_3u_3(a_1a_2)^{-2}\in\mathcal{S}
\end{align*}

{\bf Claim 3.} We have $\psi(b')=rb^{-1}(a_1a_2a_3)^4r$.

{\it Proof of Claim 3.}
If $g=4$ then we have $\psi(b')=b(a_1a_2a_3)^{-4}$ and the claim follows from
(G3) in Lemma \ref{useful_g4}.
By looking at the effect of $r\sigma$ on the curve
$\beta$ it can be checked that
$r\sigma(\beta)$ and $\beta$ are isotopic to the boundary curves of a regular neighbourhood of the union of $\alpha_i$ for $i\in\{1,2,3\}$. Hence the chain relation 
$brb'r=(a_1a_2a_3)^4$ holds in $\M^+(N_{g,0},P)$. If $g>4$ then the last relation is expressible in $\mathcal{S}$ and the claim follows.

It follows from Claim 3 and  $\psi(a_i')=ra_ir$ (E3), that if $w'$ is a word in 
$a_i'$, $b'$ and their inverses, then $\psi(w')=rwr$, where $w$ is a word in  $a_i$, $b$ and their inverses. If $w'$ represents the trivial element of $\M^+$, then so does $w$, and by Lemma \ref{shortcut_AB} $\psi(w')=1$ in $\G$. Thus $\psi$ respects the relations (A1'--A6',A9').  

We check that $\psi$ respects the remaining defining relations of $\M^+$.

(B1'): $u_i'u_j'=u_j'u_i'$. If $i,j>1$ then (B1') is the same as (B1), and if $i,j<g-1$ then it is expressible in $\mathcal{S}$. 
It remains to show that $\psi(u_{g-1}')=u_{g-1}$ commutes in $\G$ with
$\psi(u_1')=ra_1^{-1}ra_1\psi(\sigma^2)u_1$, which is true, because
$ru_{g-1}=u_{g-1}^{-1}r$ (E4),  $u_{g-1}$ commutes with $a_1$, $u_1$ (B1, C1) and
$\psi(\sigma^2)=(u_2\cdots u_{g-1})^{g-1}$ is central in the subgroup of $\G$ generated by $u_i$ for $i>1$ (B5).  

(C1'): $a_1'u_i'=u_i'a_1'$ for $i>1$. This relation is respected, because $\psi(u_i')=u_i$ commutes with $\psi(a_1')=ra_1r$ by (C1, E4)

(B2', C2', C3') are either the same as (B2, C2, C3) if $i>1$, or they are expressible in $\mathcal{S}$ if $i=1$.

(C7'): $u'_5b'=b'u_5'$. This relation is respected, because $\psi(u_5')=u_5$ commutes with
$\psi(b')=rb^{-1}(a_1a_2a_3)^4r$ by (C1a, C7, E4).

(R8') for $g=4$ is 
$\sigma^2b\sigma^{-2}=b'(a_1'a_2'a_3')^{-4}$. Let $w=(a_1a_2a_3)$. Because $\psi$ respects (A3'), we have 
\[\psi(b')(\psi(a_1')\psi(a_2')\psi(a_3'))^{-4}=(\psi(a_1')\psi(a_2')\psi(a_3'))^{-4}\psi(b')=rw^{-4}rw^{-4}b,\] 
$\psi(\sigma^2)=(u_3u_2)^3$, and
\begin{align*}
&(u_3u_2)^3b(u_3u_2)^{-3}=rw^{-4}rw^{-4}b \iff
w^4rw^4r(u_3u_2)^3=b(u_3u_2)^3b^{-1}
\stackrel{(G1)}{\iff}\\
&w^4rw^4r(u_3u_2)^3=(w^3u_2u_3w^{-1})^3\iff
w^4rw^4r(u_3u_2)^3=w^3u_2u_3w^2u_2u_3w^2u_2u_3w^{-1}\\
&\iff wrw^5(u_3u_2u_1u_3u_2u_3)\underline{u_2u_3u_2w}=u_2u_3w^2u_2\underline{u_3w^2}u_2u_3\stackrel{(C1a, C5a)}{\iff}\\
&wrw^5\underline{\Delta_4 w}u_1^{-1}u_2^{-1}u_1^{-1}=u_2u_3w^2u_2w^2u_1u_2u_3\stackrel{(E1)}{\iff}
wrw^4\Delta_4=u_2u_3w^2u_2w^2\Delta_4\iff\\
&wrw^2=u_2u_3w^2u_2\stackrel{(C1a, C5a)}{\iff}
wrw^2=wu_1^{-1}u_2^{-1}u_3^{-1}w\iff r^2=1.
\end{align*}
Since $\psi$ respects the defining relations of $\M^+$, it extends to a homomorphism, which is the inverse of $\varphi$.\end{proof}
\section{The stabilisers of $[\alpha_1]$ and $[\xi]$.}\label{sec_v13}
In this section we assume that $g\ge 5$ is fixed, Theorems \ref{mainA} and
\ref{mainB} are true for genera less then $g$ and consequently  Theorem \ref{pres_Stab_mu_g} and Lemma \ref{relC9} are true for $g$.
\subsection{The stabiliser of $[\xi]$.}
\begin{lemma}\label{two-holed-torus}
Let $g>4$. In $\G_{g,0}$ we have
$b^{-1}(a_1a_2a_3)^4=\Delta_4b^{-1}\Delta_4^{-1}=r_gbr_g$.
\end{lemma}
\begin{proof}
First we show that these relations hold in $\M(N_{g,0})$. Let $\Sigma$ be a regular neighbourhood of the union of the curves $\alpha_i$ for $i\in\{1,2,3\}$ orientated so that $a_i$ are right Dehn twists.
$\Sigma$ is a two holed torus, whose one boundary component is isotopic to $\beta$. Let $\beta'$ be the other boundary component and denote as $b'$ the right Dehn twist about $\beta'$. We have the well known relation
$(a_1a_2a_3)^4=bb'$ (called two holed torus or $3$-chain relation).
It can be checked that $\Delta_4$ and $r_g$ preserve $\Sigma$ up to isotopy and map $\beta$ on $\beta'$. Moreover, $\Delta_4$ reverses and $r_g$ preserves the orientation of $\Sigma$. Thus $b'=\Delta_4b^{-1}\Delta_4^{-1}=r_gbr_g$ (recall that $r_g$ has order two in $\M(N_{g,0})$ by (E2a)) and the relations from the lemma are satisfied in $\M(N_{g,0})$. To see that they are also satisfied in $\G_{g,0}$ note that they are composed of elements of $\Stab_{\M(N_{g,0})}[\mu_g]$ (because $g>4$) and hence it suffices to apply the homomorphism $\psi_{v_2}$ from Theorem \ref{pres_Stab_mu_g}.
\end{proof}
\begin{theorem}\label{pres_Stab_xi}
Let $g\in\{5,6\}$. The stabiliser $\Stab[\xi]=\Stab_{\M(N_{g,0})}[\xi]$ is generated by $a_i$, $b_j$ for $1\le i\le g-1$, $2\le 2j\le g-2$ and $\Delta_4$ if $g=5$, or  $u_5^{-1}\Delta_4$ and $r_6$ if $g=6$. There is 
a homomorphism $\psi_{v_3}\colon\Stab[\xi]\to\G_{g,0}$ such that
$\varphi_{g,0}\circ\psi_{v_3}=\mathrm{id}_{\Stab[\xi]}$ and
$\psi_{v_3}(x)=x$ for each generator $x$ of $\Stab[\xi]$.
\end{theorem} 
\begin{proof}
Let $\Stab=\Stab[\xi]$, $\Stab^+=\Stab^+[\xi]$, 
$\G=\G_{g,0}$, $\varphi=\varphi_{g,0}$.
Recall from Subsection \ref{struct_stab} that
$\Stab^+=\rho_{\xi}(\M(N_{\xi}))$, where
$N_{\xi}$ is obtained by cutting $N_{g,0}$ along $\xi$ 
and $\rho_{\xi}$ is the homomorphism induced by the gluing map.
$N_{\xi}$ is homeomorphic to $S_{2,g-4}$ and it is easy to see that
$\Stab^+=\rho_{\xi}(\M(N_{\xi}))=\jmath_\ast(\M(S_{2,g-4}))$,
where $\jmath_\ast$ is the map induced by the inclusion in $N_{g,0}$ of a regular neighbourhood of the union of the curves $\alpha_i$ for $i\in\{1,\dots,g-1\}$. It follows that $\Stab^+$ is generated by  $a_i$, $b_j$ for $1\le i\le g-1$, $2\le 2j\le g-2$. Moreover, by Lemma \ref{shortcut_AB} every relation in $\Stab^+$ between these generators is also a relation in $\G$.  By applying Lemma \ref{ext_pres}  to the sequence (\ref{stab_es})
\[1\to\Stab^+\to\Stab\to\Z_2^{g-4}\to 1,\]
we see that $\Stab$ is generated by $\Stab^+$ and $g-4$ cokernel generators. 
If $g=5$, then for the cokernel generator we take $\Delta_4$, which preserves the curve $\xi$ and reverses its orientation. If $g=6$, then for the cokernel generators we take 
$u_5^{-1}\Delta_4$, which reverses orientation of $\xi$ and preserves its sides, and $r=r_6$
which preserves orientation of $\xi$ and swaps its sides. 
It remains to check that the mapping $\psi_{v_3}(x)=x$ for $x$ a generator of $\Stab[\xi]$ respects the cokernel and conjugation relations. 

{\bf Case $g=5$.} The cokernel relation is (B4) $\Delta_4^2=1$.
The conjugation relations are (E1) $\Delta_4 a_i\Delta_4=a^{-1}_{4-i}$ for $i\in\{1,2,3\}$ and
\[(1)\ \Delta_4 b\Delta_4=b(a_1a_2a_3)^{-4},\quad
(2)\ \Delta_4 a_4\Delta_4=(a_1a_2a_3)a_4^{-1}(a_1a_2a_3)^{-1}.\]
These relations are satisfied in $\G$, (1) by Lemma \ref{two-holed-torus} and (2)   because
\begin{align*} 
&(a_1a_2a_3)a_4^{-1}(a_1a_2a_3)^{-1}\stackrel{(Da)}{=}
(u_3u_2u_1)^{-1}a_4(u_3u_2u_1)\\
&\stackrel{(C1a)}{=}(u_3u_2u_1)^{-1}\Delta_3^{-1}a_4\Delta_3(u_3u_2u_1)
=\Delta_4 a_4\Delta_4.
\end{align*}
{\bf Case $g=6$.}  
The cokernel relations are (E2a) $r^2=1$ and 
\[(3)\quad(u_5^{-1}\Delta_4)^2=1,\qquad (4)\quad(ru_5^{-1}\Delta_4)^2=1.\]
They are satisfied in $\G$ because $u_5$ commutes with $\Delta_4$ by (B1) and
\begin{align*}
&u_5^{-2}\stackrel{(B4a)}{=}(u_4\cdots u_1)(u_1\cdots u_4)\stackrel{(B8)}{=}\Delta_5^2\Delta_4^{-2}\stackrel{(B3)}{=}\Delta_4^{-2}\\
&(ru_5^{-1}\Delta_4)^2\stackrel{(E4a)}{=}u_5\Delta_4^{-1}u_5^{-1}\Delta_4=1.
\end{align*}
The conjugation relations are (E3a) $r a_ir=a_i$ for $i\in\{1,\dots,5\}$ and
\begin{align*}
&(5)\ rbr=b^{-1}(a_1a_2a_3)^4,\quad (6)\ u_5^{-1}\Delta_4b\Delta_4^{-1}u_5=b(a_1a_2a_3)^{-4}\\
&(7)\ u_5^{-1}\Delta_4a_i\Delta_4^{-1}u_5=a^{-1}_{4-i}\textrm{\ for\ } i\in\{1,2,3\},\quad (8)\  u_5^{-1}\Delta_4a_5\Delta_4^{-1}u_5=a^{-1}_5\\
&(9)\ u_5^{-1}\Delta_4a_4\Delta_4^{-1}u_5=(a_1a_2a_3a_4)a_5^{-1}(a_1a_2a_3a_4)^{-1} 
\end{align*}
(5, 6) follow from Lemma \ref{two-holed-torus} and (C1a, C7),  
(7) from (C1a, E1) and (8) from (C1a, C4a). We have
\begin{align*}
&u_5^{-1}\Delta_4a_4\Delta_4^{-1}u_5\stackrel{(3)}{=}\Delta_4^{-1}u_5a_4u_5^{-1}\Delta_4
\stackrel{(C3)}{=}\Delta_4^{-1}u_4^{-1}a_5u_4\Delta_4\stackrel{(B1, C1a)}{=}\\
&(u_4u_3u_2u_1)^{-1}a_5(u_4u_3u_2u_1)\stackrel{(Da)}{=}(a_1a_2a_3a_4)a_5^{-1}(a_1a_2a_3a_4)^{-1}
\end{align*}
which proves that (9) holds in $\G$. We do not need to check that $\psi_{v_3}$ respects the conjugation relations expressing $rb_2r$ and $u_5^{-1}\Delta_4b_2\Delta_4^{-1}u_5$ in the generators of $\Stab^+$, because they follow from (A8) and the remaining conjugation relations.
\end{proof}
\subsection{The stabiliser of $[\alpha_1]$.}
\begin{theorem}\label{pres_stab_alpha1}
Let $g\ge 5$. The stabiliser $\Stab[\alpha_1]=\Stab_{\M(N_{g,0})}[\alpha_1]$ is generated by $u_i$, $a_i$ for $i\in\{1,3,\dots,g-1\}$, $b_j$ for $2\le 2j\le g$, $c=T_{\gamma_{\{3,4,5,6\}}}$ (if $g\ge 6$), $v=Y_{\mu_4,\beta}$ and $r_g$. 
There is a homomorphism $\psi_{v_1}\colon\Stab_{[\alpha_1]}\to\G_{g,0}$ such that $\varphi_{g,0}\circ\psi_{v_1}=\mathrm{id}_{\Stab[\alpha_1]}$ and
$\psi_{v_1}(c)=(a_1\cdots a_5)^2b(a_1\cdots a_5)^{-2}$, $\psi_{v_1}(v)=a_3a_2a_1u_1u_2u_3$ and $\psi_{v_1}(x)=x$ for the remaining generators of $\Stab[\alpha_1]$.
\end{theorem} 
\begin{proof}
Let $\Stab=\Stab[\alpha_1]$, $\Stab^+=\Stab^+[\alpha_1]$, $\G=\G_{g,0}$ and
$\varphi=\varphi_{g,0}$. First we are going to define a homomorphism 
$\psi^+\colon\Stab^+\to\G$ such that $\varphi\circ\psi^+=\mathrm{id}_{\Stab^+}$. Recall the exact sequence (\ref{cut_es})
\[1\to\Z\to\M(N_{\alpha_1})\stackrel{\rho}{\to}\Stab^+\to 1,\]
where $N_{\alpha_1}$ is obtained by cutting $N_{g,0}$ along $\alpha_1$, and
$\rho=\rho_{\alpha_1}$ is induced by the gluing. In order to define $\psi^+$
it suffices to construct a homomorphism $\psi'\colon\M(N_{\alpha_1})\to\G$ satisfying $\psi'(\ker\rho)=1$ and $\varphi\circ\psi'=\rho$. 

Since $N_{\alpha_1}$ is homeomorphic to $N_{g-2,2}$, we need a presentation for $\M(N_{g-2,2})$, which can be obtained from the sequence (\ref{Cup_es})
\[1\to\Z\to\M(N_{g-2,2})\to\M^+(N_{g-2,1},P)\to 1,\]
provided that we have a presentation for $\M^+(N_{g-2,1},P)$. 

\medskip

\noindent{\bf Step 1: a presentation of $\M^+(N_{g-2,1},P)$.} We proceed in the same way as in the proof of Theorem \ref{mainA_punctured}. First we apply Lemma \ref{ext_pres} to the Birman exact sequence
\[1\to\pi_1(N_{g-2,1},P)\to\M(N_{g-2,1},P)\to\M(N_{g-2,1})\to 1\]
and the presentation of $\M(N_{g-2,1})$ given in Theorem \ref{mainA}. 
As a result we obtain a presentation of $\M(N_{g-2,1},P)$ on generators
$a_i$, $u_i$ for $1\le i\le g-3$, $b_j$ for $2\le 2j\le g-2$ and $\sigma_l$ for
$1\le l\le g-2$. Since $\pi_1(N_{g-2,1},P)$ is free, there are no kernel relations. The cokernel relations are
(A1--A9, B1, B2, C1--C8) and the conjugation relations are (1--11), the same as in the proof of Theorem \ref{mainA_punctured}, and by repeating the arguments from that proof we can reduce the conjugation relations to
(R1--R8). Then we make the transformations (2, 3, 7) from the proof of Theorem \ref{mainA_punctured} and instead of (5) we  replace $(\sigma_2\sigma_1)$  by 
$a_1^{-1}\sigma_1 a_1\sigma_1^{-1}$ in (R2, R5, R8)
and by $u_1^{-1}\sigma_1u_1\sigma_1^{-1}\sigma_1^2$ in (R4). 
Then we rule out the generators $\sigma_i$ for $i>1$ together with (R1).
By the Reidemeister-Schreier method we obtain a presentation for
$\M^+(N_{g-2,1},P)$ on the generators  $b$, $a_i$, $u_i$, $b'=\sigma b\sigma^{-1}$, $a_i'=\sigma a_i\sigma^{-1}$, $u_i'=\sigma u_i\sigma^{-1}$ and $\sigma^2$, where $\sigma=\sigma_1$. The relations are the defining relations 
of $\M(N_{g-2,1},P)$ (Rel) and their conjugates by $\sigma$ (Rel').
Note that  $a_i'=a_i$ and $u_i'=u_i$ for $i>1$ while $b'$ and $\sigma^2$ may be expressed in the remaining generators by  (R8) and (R2) respectively. We record 
for future reference the following remark.
\begin{rem}\label{gensNg11}
$\M^+(N_{g-2,1},P)$ is generated by $a_i$ for $i=1,\dots,g-3$,
$u_j$ for $j=1,\dots,g-4$, $a_{g-3}u_{g-3}$, $a_1'$, $a_1'u_1'$, and
$b$ if $g-2\ge 4$.
\end{rem}

\noindent{\bf Step 2: a presentation for $\M(N_{g-2,2})$} We apply Lemma \ref{ext_pres} to the sequence (\ref{Cup_es}). 
Let $d_1$ and $d_2$ be Dehn twists about the boundary components of $\M(N_{g-2,2})$, such that $d_1$ generates the kernel of the map
$\M(N_{g-2,2})\to\M^+(N_{g-2,1},P)$. We are assuming that 
$N_{g-2,2}=N_{g-2,1}\backslash U$, where $U$ is a small open neighbourhood of $P$, and we treat $b$, $b'$, $a_i$, $a_i'$ (resp. $u_i$,  $u_i'$ ) for $i\in\{1,\dots,g-3\}$ as Dehn twists (resp.  crosscap transpositions) on $N_{g-2,2}$. These will be our cokernel generators of $\M(N_{g-2,2})$, together with $\sigma^2$ defined according to (R2) as $\sigma^2=a_1^{-1}a_1'u_1'u_1^{-1}$. There are two types of defining relations of $\M(N_{g-2,2})$:
(I) The conjugation relations:  
$d_1x=xd_1$, for which it suffices to take the cokernel generators $x$ from Remark \ref{gensNg11}.
(II) The cokernel relations: for every defining relation $w=1$ of
$\M^+(N_{g-2,1},P)$ we have a relation  
$w=d_1^k$ in $\M(N_{g-2,2})$, where $k$ is some integer depending on $w$. We denote as ($\widetilde{\textrm{Rel}}$) and ($\widetilde{\textrm{Rel'}}$) the cokernel relations corresponding to (Rel) and (Rel') respectively.

\medskip

\noindent{\bf Step 3: definition of $\psi'$.}
There is a homeomorphism $f\colon N_{g-2,2}\to N_{\alpha_1}$ inducing an isomorphism $f_\ast\colon\M(N_{g-2,2})\to\M(N_{\alpha_1})$ such that
$f_\ast(a_i)=a_{i+2}$, $f_\ast(u_i)=u_{i+2}$, for
$i\in\{1,\dots,g-3\}$, $f_\ast(b)=T_{\gamma_{\{3,4,5,6\}}}=c$, $f_\ast(a_1')=b$, $f_\ast(b')=b_2$ and
$f_\ast(u_1')=U_{\mu_4,\beta}=T_{\beta}^{-1}Y_{\mu_4,\beta}=b^{-1}v$.
We assume that the Dehn twists $d_1$, $d_2$ are such that
$\rho(f_\ast(d_1))=a_1=\rho(f_\ast(d_2^{-1}))$.

We define $\psi'\colon\M(N_{\alpha_1})\to\G$ as $\psi'=\theta\circ f_\ast^{-1}$, where $\theta\colon\M(N_{g-2,2})\to\G$ is defined on the generators as:
\begin{align*}
&\theta(d_1)=a_1,\quad\theta(a_i)=a_{i+2},\quad\theta(u_i)=u_{i+2}\quad\textrm{for\ }i\in\{1,\dots,g-3\},\\
&\theta(a_1')=b,\quad\theta(u_1')=b^{-1}a_3a_2a_1u_1u_2u_1,\\
&\theta(b)=(a_1\cdots a_5)^2b(a_1\cdots a_5)^{-2},\quad\theta(b')=b_2,\\
&\theta(\sigma^2)=\theta(a_1)^{-1}\theta(a_1')\theta(u_1')\theta(u_1)^{-1}=a_2a_1u_1u_2.
\end{align*}
For every generator $x$ of $\M(N_{g-2,2})$ we have
$\varphi(\theta(x))=\rho(f_\ast(x))$. This is obvious for all generators except $b$ and $u'_1$. We have $\varphi(\theta(b))=(a_1\cdots a_5)^2b(a_1\cdots a_5)^{-2}$,  $\rho(f_\ast(b))=c$, and since $(a_1\cdots a_5)^2$ maps $\beta$ on $\gamma_{\{3,4,5,6\}}$, the equality holds. We have
$\rho(f_\ast(u_1'))=b^{-1}v$, $\varphi(\theta(u_1'))=b^{-1}a_3a_2a_1u_1u_2u_1$ and
since the crosscap pushing map is a homomorphism, thus
\begin{align*}
v&=Y_{\mu_4,\beta}=Y_{\mu_4,\gamma_{\{3,4\}}}Y_{\mu_4,\gamma_{\{2,4\}}}Y_{\mu_4,\gamma_{\{1,4\}}}=
y_3(u_3^{-1}y_2u_3)(u_3^{-1}u_2^{-1}y_1u_2u_3)\\
&=a_3u_3(u_3^{-1}a_2u_2u_3)(u_3^{-1}u_2^{-1}a_1u_1u_2u_3)=a_3a_2a_1u_1u_2u_3.
\end{align*}
By abuse of notation we are going to denote $\theta(b)$ by $c$ and $b\theta(u_1')$ by $v$. We also set $e=\theta(\sigma^2)=a_2a_1u_1u_2$.
In order to prove that $\theta$ is a homomorphism, we have to check that it respects the defining relations of $\M(N_{g-2,2})$.

\medskip

\noindent{\bf Step 4: proof that $\theta$ is a homomorphism.}
Let $\mathcal{S}=\mathcal{S}_{g,0}(v_2)$ and let $\mathcal{A}$ be the subgroup of 
$\G$ generated by $b$ and the $a_i$'s.
Suppose that $w$ is a word in the generators of $\M(N_{g-2,2})$ and their inverses.
We say that $w$ is expressible in $\mathcal{S}$ or in $\mathcal{A}$ if it is mapped by $\theta$ on an element of  $\mathcal{S}$ or $\mathcal{A}$ respectively.
By similar argument as in the proof of Theorem \ref{mainA_punctured}, $\theta$ respects the relations expressible in $\mathcal{S}$, and by Lemma \ref{shortcut_AB} it also respects the relations expressible in $\mathcal{A}$.
Since $\theta(d_1)=a_1\in \mathcal{S}\cap\mathcal{A}$, a cokernel relation
$w=d_1^k$ is expressible in $\mathcal{S}$ or $\mathcal{A}$ if and only if
$w$ is. By this observation we will be able to deduce that $\theta$ respects some cokernel relations without having to determine the exponent $k$.

The conjugation relations are mapped by $\theta$ on  
$a_1\theta(x)=\theta(x)a_1$ for each cokernel generator $x$ from Remark \ref{gensNg11}. Note that 
$\theta(x)\in\mathcal{S}\cup\mathcal{A}$ and  since 
$a_1\in \mathcal{S}\cap\mathcal{A}$, thus $\theta$ respects the conjugation relations.

Let $s=(u_1\cdots u_{g-1})$. By (C1a, C3) we have
$s^2a_is^{-2}=a_{i+2}=\theta(a_i)$, and by (B1, B2)
$s^2u_is^{-2}=u_{i+2}=\theta(u_i)$ for $i\in\{1,\dots,g-3\}$. Also
\[s^2bs^{-2}\stackrel{(E5)}{=}(a_1\cdots a_{g-1})^2b(a_1\cdots a_{g-1})^{-2}
\stackrel{(A1, A3)}{=}c=\theta(b).
\]
 If (Rel) is one of the relations (A1--A6, A9, B1, B2, C1--C8) then ($\widetilde{\textrm{Rel}}$) is the same as (Rel) and it is 
mapped by $\theta$ on its conjugate by $s^2$. 

If (Rel) is one of the relations (R2, R3, R6) then
($\widetilde{\textrm{Rel}}$) is the same as (Rel) and it is trivially preserved by $\theta$. ($\widetilde{\textrm{R2'}}$, $\widetilde{\textrm{R4}}$, $\widetilde{\textrm{R4'}}$, $\widetilde{\textrm{R5}}$, $\widetilde{\textrm{R5'}}$,
$\widetilde{\textrm{C4'}}$, $\widetilde{\textrm{C8'}}$) are expressible in $\mathcal{S}$. 
($\widetilde{\textrm{A1'}}-\widetilde{\textrm{A6'}}$, $\widetilde{\textrm{A9'}}$,  $\widetilde{\textrm{R7}}$, $\widetilde{\textrm{R8}})$ are expressible in $\mathcal{A}$.

Note that $(\widetilde{\textrm{R7}})$ is 
\[(a_4a_3a_2a_1')^{-1}b'(a_4a_3a_2a_1')=(a_4a_3a_2a_1)^{-1}b(a_4a_3a_2a_1),\]
 and $(\widetilde{\textrm{R7'}})$ is 
\[\sigma^2(a_4a_3a_2a_1)^{-1}b(a_4a_3a_2a_1)\sigma^{-2}=(a_4a_3a_2a_1')^{-1}b'(a_4a_3a_2a_1').\]
We already know that $\theta$ respects $(\widetilde{\textrm{R7}})$ and to prove the same for $(\widetilde{\textrm{R7'}})$ it suffices to show that $e=\theta(\sigma^2)$ commutes in $\G$ with $(a_6a_5a_4a_3)^{-1}c(a_6a_5a_4a_3)$. It follows from earlier part of the proof that in $\G$ we have  \[(a_6a_5a_4a_3)^{-1}c(a_6a_5a_4a_3)=s^2(a_4a_3a_2a_1)^{-1}b(a_4a_3a_2a_1)s^{-2},\] where $s=(u_1\dots u_{g-1})$. Setting $w=(a_4a_3a_2a_1)^{-1}b(a_4a_3a_2a_1)$,   it suffices to show that it  commutes  with $s^{-2}es^2$.
We have
\[e=a_2a_1u_1u_2=a_2u_2(u_2^{-1}a_1u_1u_2)\stackrel{(B1,C2)}{=}(a_2u_2)u_1(a_2u_2)u_1^{-1}.\]
By (B1, B2, B3) $s^{-2}u_1s^2=s^{g-2}u_1s^{2-g}=u_{g-1}$, and as (R7) is valid only for $g-2\ge 5$, $u_{g-1}$ commutes
with $w$ by (C1a, C7a). By (B1, B2, C1a, C3) we  have
\[s^{-2}a_2u_2s^2=s^{-1}a_1u_1s=u_{g-1}^{-1}(u_1\cdots u_{g-2})^{-1}a_1u_1(u_1\cdots u_{g-2})u_{g-1}.\] 
Since $w$ and $(u_1\cdots u_{g-2})^{-1}a_1u_1(u_1\cdots u_{g-2})$ are in $\mathcal{S}$ and they commute in $\M(N_{g,0})$, they also commute in $\G$.
We already  proved that $w$ commutes with $u_{g-1}$, hence it commutes with $s^{-2}es^2$ and $\theta$ respects $(\widetilde{\textrm{R7'}})$.

($\widetilde{\textrm{R3'}}$) and ($\widetilde{\textrm{R6'}}$) are $\sigma^2u_i=u_i\sigma^2$ for $i\ge 2$ and $\sigma^2a_2=a_2\sigma^2$ respectively, and they are mapped on $eu_{i+2}=u_{i+2}e$, $ea_4=a_4e$,  which follow from (A1, B1, C1a).

($\widetilde{\textrm{B1'}}$) is either the same as (B1) or $u_1'u_j=u_ju_1'$ for $j>2$. The last relation is 
mapped by $\theta$ on $b^{-1}vu_i=u_ib^{-1}v$ for $i>4$, which follows from (B1, C1a, C7a). 
($\widetilde{\textrm{B2'}}$) is either the same as (B2) or $u_2u_1'u_2=u_1'u_2u_1'$, which is mapped on
\begin{align*}
& u_4b^{-1}vu_4=b^{-1}vu_4b^{-1}v \iff u_4b^{-1}a_4^{-1}(a_4vu_4)=b^{-1}a_4^{-1}(a_4vu_4)b^{-1}v\\
&\iff ba_4bu_4b^{-1}a_4^{-1}=b(a_4vu_4)b^{-1}v(a_4vu_4)^{-1}.
\end{align*}
Since $b,v\in\mathcal{S}$ and $a_4vu_4\in\mathcal{S}$ (see transformation (5) in the proof of Theorem \ref{mainA_punctured}), thus it suffices to show that the left hand side of the last relation is also in $\mathcal{S}$. This is true because 
\begin{align*}
&\underline{ba_4b}u_4b^{-1}a_4^{-1}\stackrel{(A4)}{=}
a_4\underline{ba_4u_4b^{-1}}a_4^{-1}\stackrel{(C8)}{=}
a_4a_4u_4a_4vu_4a_4^{-1}\stackrel{(C4a)}{=}a_4u_4va_4u_4.
\end{align*}
($\widetilde{\textrm{C1'}}$) is $a_1'u_i=u_ia_1'$ for $i>2$ and it is mapped on 
(C7a) $bu_j=u_jb$ for $j>4$.
($\widetilde{\textrm{C2'}}$) is either the same as (C2) or 
$a_1'u_2u_1'=u_2u_1'a_2$, which is mapped on
\[bu_4b^{-1}v=u_4b^{-1}va_4 \iff ba_4bu_4b^{-1}a_4^{-1}(a_4vu_4)=ba_4u_4b^{-1}va_4u_4.\]
Both sides of the last relation are in $\mathcal{S}$ because
we showed above that $ba_4bu_4b^{-1}a_4^{-1}\in\mathcal{S}$.
($\widetilde{\textrm{C3'}}$) is either the same as (C3) or 
$a_2u_1'u_2=u_1'u_2a_1'$, which is mapped on
\[a_4b^{-1}vu_4=b^{-1}vu_4b\stackrel{(A4)}{\iff} ba_4vu_4=a_4vu_4b.\]
The last relation is equivalent to (C9) from Lemma \ref{relC9}.
($\widetilde{\textrm{C5'}}$) is $u_2a_1'a_2u_1'=a_1'a_2$ and it is mapped on
\[u_4ba_4b^{-1}v=ba_4\stackrel{(A4)}{\iff}u_4a_4^{-1}ba_4v=ba_4\stackrel{(C4a)}{\iff}
a_4u_4ba_4vu_4=ba_4u_4
\]
The last relation follows from (C8, C9).
($\widetilde{\textrm{C7'}}$) is $b'u_5=u_5b'$ and it is mapped on
$b_2u_7=u_7b_2$, which follows from (A8, C1a, C7a).

It remains to show that $\theta$ respects the relations ($\widetilde{\textrm{C6'}}$) and ($\widetilde{\textrm{R8'}}$). This follows from the next lemma, whose proof will be given in the next subsection. Recall that a relation in $\M(N_{g-2,2})$ whose both sides are mapped by $\theta$ on elements of $\cS$ is called expressible in $\cS$.

\begin{lemma}\label{c6r8}
($\widetilde{\textrm{C6'}}$) and ($\widetilde{\textrm{R8'}}$) are expressible in $\mathcal{S}$.
\end{lemma}

\noindent{\bf Step 5: checking that $\psi'(\ker\rho)=1$.} 
Recall that  $d_1$, $d_2$ are Dehn twists about the boundary components of $N_{g-2,2}$  such that $\rho(f_\ast(d_1d_2))=1$, and so $f_\ast(d_1d_2)$ is a generator
of  $\ker\rho$. We have  $\psi'(f_\ast(d_1d_2))=\theta(d_1d_2)$ and to prove 
$\theta(d_1d_2)=1$ it suffices to show
$\theta(d_1d_2)\in\mathcal{S}$.
Let $z=\sigma_{g-2}^2\cdots \sigma_1^2\in\mathfrak{p}(\pi_1(N_{g-2,1},P))$. 
By Lemma \ref{push1}, in $\M^+(N_{g-2,1},P)$ we have the relation
\[z=d_2^{\varepsilon}(u_1\cdots u_{g-3})^{g-2},\]
where $\varepsilon\in\{-1,1\}$,
which gives in  $\M(N_{g-2,2})$
\[z(u_1\cdots u_{g-3})^{2-g}=d_2^\varepsilon d_1^k\] for some $k\in\mathbb{Z}$.
We have \[\theta(d_1d_2)^\varepsilon=\theta(z)(u_3\cdots u_{g-1})^{2-g}a_1^{\varepsilon-k}.\]
Since $(u_3\cdots u_{g-1})^{2-g}\in\mathcal{S}$ by Lemma \ref{Delta_in_stab},  it suffices to prove $\theta(z)\in\mathcal{S}$ and clearly it is enough to show that $\theta(\sigma_{g-2}^2)\in\mathcal{S}$. We have
\begin{align*}
&\theta(\sigma_{g-2}^2)=\theta((u_1\cdots u_{g-3})^{-1}\sigma^2(u_1\cdots u_{g-3}))=\\
&(u_3\cdots u_{g-1})^{-1}a_2a_1u_1u_2(u_3\cdots u_{g-1})=\\
&(u_3\cdots u_{g-1})^{-1}a_2u_2(u_1a_2u_2u_1^{-1})(u_3\cdots u_{g-1})=\\
&(u_2\cdots u_{g-2})a_{g-1}u_{g-1}(u_2\cdots u_{g-2})^{-1}(u_1\cdots u_{g-2})a_{g-1}u_{g-1}(u_1\cdots u_{g-2})^{-1}\end{align*}
where the last equality follows from (B1, B2, C1a, C3).

\medskip

\noindent{\bf Step 6: extending $\psi^+$.} 
We have a homomorphism $\psi^+\colon\Stab^+\to\G$ defined as $\psi^+(\rho(x))=\psi'(x)$ for $x\in\M(N_{\alpha_1})$.
Since $\rho\circ f_\ast\colon\M(N_{g-2,2})\to\Stab^+$ is an epimorphism,  $\Stab^+$ is generated by $u_i$, $a_i$ for $i\in\{3,\dots,g-1\}$, $b$, $c$ ,$v$ and $a_1$.  By the definition of $\psi'$ (Step 3) we have $\psi^+(c)=(a_1\cdots a_5)^2b(a_1\cdots a_5)^{-2}$, $\psi^+(v)=a_3a_2a_1u_1u_2u_3$ and $\psi^+(x)=x$ for the remaining generators of $\Stab^+$. Also $\psi^+(b_2)=b_2$, and by (A8)
$\psi^+(b_j)=b_j$ for $j\ge 3$.

By applying Lemma \ref{ext_pres}  to the sequence (\ref{stab_es})
we see that $\Stab$ is generated by $\Stab^+$ and two cokernel generators, for which we take $u_1$ (preserves orientation of $\alpha_1$ and swaps its sides) and $r=r_g$ (reverses orientation of $\alpha_1$ and preserves its sides). We let $\psi_{v_1}$ be equal to $\psi^+$ on $\Stab^+$ and $\psi_{v_1}(u_1)=u_1$,  
$\psi_{v_1}(r)=r=a_1\cdots a_{g-1}u_{g-1}\cdots u_1$.  Note that
$\varphi(\psi_{v_1}(x))=x$ for every generator $x$ of $\Stab$. It remains to check that $\psi_{v_1}$ respects the cokernel and conjugation relations.
The cokernel relations are (E2a) $r^2=1$, (E4a) $(ru_1)^2=1$ and
$u_1^2=(u_3\cdots u_{g-1})^{g-2}$ which holds in $\G$ by Lemma \ref{shortcut_AB}.

By Remark \ref{gensNg11}, $\Stab^+$ is generated by $u_{g-1}$, $b$, $c$, $v$ and $a_i$ for $i=1,3,\dots,g-1$. 
Set 
\[w=(u_4\cdots u_{g-1})(a_4\cdots a_{g-1})^{-1}=(u_4\cdots u_{g-2})a_{g-1}u_{g-1}(a_4\cdots a_{g-2})^{-1}.\]
The conjugation relations are (E3a) $ra_ir=a_i$, (E4) $ru_{g-1}r=u^{-1}_{g-1}$,
(C1a) $u_1a_iu_1^{-1}=a_i$ for $i=3,\dots,g-1$, (C4) $u_1a_1u_1^{-1}=a_1^{-1}$,
(B1) $u_1u_{g-1}u_1^{-1}=u_{g-1}$ and 
\begin{itemize}
\item[(1)] $rbr=w^{-1}b^{-1}w$
\item[(2)] $rvr=a_3w^{-1}a_3^{-1}va_3wa_3^{-1}$ 
\item[(3)] $rcr=c^{-1}(a_3a_4a_5)^{4}$
\item[(4)] $u_1^{-1}bu_1=u_3^{-1}rb^{-1}ru_3$ 
\item[(5)] $u_1^{-1}vu_1=u_3^{-1}a_3^{-1}rv^{-1}ra_3u_3$
\item[(6)] $u_1^{-1}cu_1=c$ 
\end{itemize}

It can be checked that $wr$ and $a_3wa_3^{-1}r$ preserve the curve $\beta$, preserve its orientation and reverse local orientation of its neighbourhood.
Additionally $a_3wa_3^{-1}r$ preserves $\mu_4$. Since
$b=T_\beta$ and $v=Y_{\mu_4,\beta}$, thus (1, 2) are satisfied in $\M(N_{g,0})$. 
Since $w\in\mathcal{S}$, they are also satisfied in $\G$. 
Similarly, $u_1u_3^{-1}r$ and $u_1u_3^{-1}a_3^{-1}r$ preserve $\beta$, reverse its orientation and local orientation of its neighbourhood.
Additionally $u_1u_3^{-1}a_3^{-1}r$ preserves $\mu_4$. It follows that (4, 5) are satisfied in $\M(N_{g,0})$ and  also in $\G$. 
By Lemma \ref{two-holed-torus},
 $rbr=b^{-1}(a_1a_2a_3)^4$. Conjugating the last relation by $(a_1\cdots a_6)^2$ we obtain (3). 
Recall from Step 4 that in $\G$ we have $c=s^2bs^{-2}$ and $s^{2}u_{g-1}s^{-2}=u_1$ where $s=(u_1\cdots u_{g-1})$.
 Conjugating the relation $bu_{g-1}=u_{g-1}b$ (C7a) by $s^2$ we obtain (6).
\end{proof}
\subsection{Proof o Lemma \ref{c6r8}.}
To finish the proof of Theorem \ref{pres_stab_alpha1}, we yet have prove that the relations
\begin{align*}
&(\widetilde{\textrm{C6'}})\quad (u_3b')^2=a_1'a_2a_1'(a_3u_3)u_2(a_3u_3)u_2^{-1}u_1'u_2(a_3u_3)^2a_1'd_1^{k_1}\\
&(\widetilde{\textrm{R8'}})\quad
\sigma^2b\sigma^{-2}=b'(a_2a_1'a_3a_2)^{-1}{a_1'}^{-1}\sigma^2a_1\sigma^{-2}(a_2a_1'a_3a_2){a_1'}^{-1}\sigma^2a_1\sigma^{-2}d_1^{k_2} 
\end{align*}
where $k_1, k_2$ are some integers, are expressible in $\mathcal{S}=\mathcal{S}_{g,0}(v_2)$. This is obvious if 
$g>6$, therefore in this subsection we assume $g=6$. We denote $\M(N_{6,0})$ by $\M$ and
$\G_{6,0}$ by $\G$.
\begin{lemma}\label{expC}
The following relations hold in $\G$. 
\[(1)\quad b_2a_1a_3a_5=cdb\qquad (2)\quad c^{-1}u_5^{-1}du_5c=u_5^{-1}du_5,\]
where $d=(a_4a_3a_5a_4)^{-1}b(a_4a_3a_5a_4)$.
\end{lemma}
\begin{proof}
It is easy to check that $(a_4a_3a_5a_4)$ maps the curve $\gamma_{\{1,2,5,6\}}$ on $\beta$ and so $d=T_{\gamma_{\{1,2,5,6\}}}$. Observe that $\beta_2$, $\alpha_1$, $\alpha_3$ and $\alpha_5$ bound a $4$-holed sphere in $N_{6,0}$, and so in $\M$ we have the well known lantern relation, which is (1). 
The same relation holds also in $\G$ by Lemma \ref{shortcut_AB}. 

Since the curves 
$u_5(\gamma_{\{3,4,5,6\}})$ and $\gamma_{\{1,2,5,6\}}$ are disjoint up to isotopy,  (2) holds in $\M$. It can be checked that $w=(u_4u_3u_5u_4)(a_4a_3a_5a_4)^{-1}$ preserves the curve $\beta$ and preserves local orientation of its neighbourhood. It follows that $b$ commutes with $w$ in $\M$ and also in $\G$, because by (B1, C3,C4a)
\[w=u_4u_3(u_5a_4u_4u_5^{-1})a_5u_5a_3^{-1}a_4^{-1}=u_4u_3(u_4^{-1}a_5u_5u_4)a_5u_5a_3^{-1}a_4^{-1}\in\mathcal{S}.\]
 It follows that in $\G$ we have
$d=(u_4u_3u_5u_4)^{-1}b(u_4u_3u_5u_4)$
and
\begin{align*}
c&=(a_1\cdots a_5)^2b(a_1\cdots a_5)^{-2}\stackrel{(E5)}{=}(u_5\cdots u_1)^{-2}b(u_5\cdots u_1)^2\\
&\stackrel{(C7,B1)}{=}
(u_4u_3u_5u_4u_2u_3u_1u_2u_1)^{-1}b(u_4u_3u_5u_4u_2u_3u_1u_2u_1)\\
&=(u_2u_3u_1u_2u_1)^{-1}d(u_2u_3u_1u_2u_1)=\Delta_4^{-1}u_1du_1^{-1}\Delta_4.
\end{align*}
Since $u_1$ commutes in $\G$ with $u_5$ and $c$ (see relation (6) in Step 6 above), 
(2) is equivalent in $\G$ to 
\[c^{-1}u_5^{-1}\Delta_4c\Delta_4^{-1}u_5c=u_5^{-1}\Delta_4c\Delta_4^{-1}u_5.\]
The last relation holds in $\G$ by Theorem \ref{pres_Stab_xi}, because $c$ and $u_5^{-1}\Delta_4$ are in the image of $\psi_{v_3}\colon\Stab[\xi]\to\G$.
\end{proof}

Clearly the right hand side of ($\widetilde{\textrm{C6'}}$) is expressible in $\mathcal{S}$ and so it suffices to show $(\theta(u_3)\theta(b'))^2=(u_5b_2)^2\in\mathcal{S}$. We have
\begin{align*}
&(u_5b_2)^2\stackrel{(1)}{=}(u_5cdba_1^{-1}a_3^{-1}a_5^{-1})^2
\stackrel{(A1, C1a)}{=}\\
&u_5cdb\underline{a_5^{-1}u_5a_5^{-1}}cdba_1^{-2}a_3^{-2}\stackrel{(C4a)}{=}
(u_5cdb)^2a_1^{-2}a_3^{-2}.
\end{align*}
Since $b$ commutes with $b_2, a_1, a_2, a_3$  (A3, A9b), it also commutes with $cd$ by (1). By (C7) $b$ commutes with $u_5$ and we have
\begin{align*}
&(u_5b_2)^2=(u_5cd)^2b^2a_1^{-2}a_3^{-2}=
(u_5c)^2(c^{-1}u_5^{-1}du_5cd)b^2a_1^{-2}a_3^{-2}
\stackrel{(2)}{=}(u_5c)^2u_5^{-2}(u_5d)^2b^2a_1^{-2}a_3^{-2}
\end{align*}
By (C6) and the transformation (3) from the proof of Theorem \ref{mainA_punctured},
$(u_3b)^2$ can be expressed in $\G$ in terms of
$a_1$, $u_1$, $a_2$, $u_2$, $a_3u_3$ as
\[(u_3b)^2=a_1a_2a_1(a_3u_3)u_2(a_3u_3)u_2^{-1}u_1u_2(a_3u_3)^2a_1.\]
Conjugating this relation by
$(a_1\cdots a_5)^2$ we obtain an expression of $(u_5c)^2$ in terms of
$a_3$, $u_3$, $a_4$, $u_4$, $a_5u_5$, and conjugating by
$(u_4u_3u_5u_4)^{-1}$ we obtain an expression of $(u_5d)^2$ in terms of
$a_1$, $u_1$, $(u_3u_4)^{-1}a_2(u_3u_4)$, $(u_3u_4)^{-1}u_2(u_3u_4)$, $a_5u_5$.
Hence $(u_5c)^2$ and $(u_5d)^2$ are in $\mathcal{S}$ and so is $(u_5b_2)^2$.

\medskip

The relation ($\widetilde{\textrm{R8'}}$) is mapped by $\theta$ on 
\[
ece^{-1}=b_2(a_4ba_5a_4)^{-1}{b}^{-1}ea_3e^{-1}(a_4ba_5a_4){b}^{-1}ea_3e^{-1}a_1^{k}, 
\]
where $e=\theta(\sigma^2)=a_2a_1u_1u_2$. Since $e$ commutes with $a_1$ by (B1,C3), $a_2a_1$ commutes with
$b_2$, $b$, $a_4$, $a_5$ by Lemma \ref{shortcut_AB} and
$b(a_4ba_5a_4)=(a_4ba_5a_4)a_5$ by (B2, B4), the relation is equivalent to 
\[u_1u_2c=b_2a_5^{-1}(a_4ba_5a_4)^{-1}u_1u_2a_3u_2^{-1}u_1^{-1}(a_4ba_5a_4)b^{-1}u_1u_2a_3a_1^{k}\]
We have to show that $w\in\mathcal{S}$ for $w$ defined as
\[w=c^{-1}u_2^{-1}u_1^{-1}b_2a_5^{-1}(a_4ba_5a_4)^{-1}u_1u_2a_3u_2^{-1}u_1^{-1}(a_4ba_5a_4)b^{-1}u_1u_2a_3a_1^{k}.\]
We define in $\G$ three equivalence relations $\sim_L$, $\sim_R$ and $\approx$ as follows. We set $w_1\sim_L w_2$ if there exists $u\in\mathcal{S}$ such that $w_2=uw_1$. Similarly, we set $w_1\sim_R w_2$ if there exists $u\in\mathcal{S}$ such that $w_2=w_1u$. Finally, we set $w_1\approx w_2$ if there exist $u, u'\in\mathcal{S}$ such that $w_2=uw_1u'$. Observe that the equivalence class of $w$ for the relation $\sim_L$ is the coset $\mathcal{S}w$, its equivalence class  for the relation $\sim_R$ is the coset $w\mathcal{S}$, and its equivalence class  for the relation $\approx$ is the double-coset $\mathcal{S}w\mathcal{S}$. Observe also that $\approx$ is the equivalence relation generated by the union of $\sim_L$ and $\sim_R$. Moreover, we have $w\in\mathcal{S}$ if and only if the equivalence class of $w$ for the relation $\approx$ is $\mathcal{S}$.

By (A1--A4) and (C1a) we have
\begin{align*}
&(a_4ba_5a_4)^{-1}u_1u_2a_3u_2^{-1}u_1^{-1}(a_4ba_5a_4)=\\
&a_4^{-1}b^{-1}\underline{a_5^{-1}a_4^{-1}u_1u_2}a_3\underline{u_2^{-1}u_1^{-1}a_4a_5}ba_4=\\
&a_4^{-1}b^{-1}u_1u_2\underline{a_5^{-1}a_4^{-1}a_3a_4a_5}u_2^{-1}u_1^{-1}ba_4=\\
&a_4^{-1}b^{-1}u_1u_2a_3a_4a_5a_4^{-1}a_3^{-1}u_2^{-1}u_1^{-1}ba_4\sim_R
a_4^{-1}b^{-1}u_1u_2a_3a_4a_5
\end{align*}
Thus
\begin{align*}
&w\sim_R c^{-1}u_2^{-1}u_1^{-1}b_2a_5^{-1}a_4^{-1}b^{-1}u_1u_2a_3a_4a_5\stackrel{(1)}{=}\\
&c^{-1}u_2^{-1}u_1^{-1}\underline{(cdba_5^{-1}a_3^{-1}a_1^{-1})a_5^{-1}a_4^{-1}b^{-1}}u_1u_2a_3a_4a_5=\\
&c^{-1}u_2^{-1}u_1^{-1}ca_5^{-2}a_3^{-1}(a_4a_5a_3a_4)^{-1}b(a_4a_5a_3a_4)ba_4^{-1}b^{-1}a_1^{-1}u_1u_2a_3a_4a_5.
\end{align*}
We have
\begin{align*}
&(a_4a_5a_3\underline{a_4)ba_4^{-1}b^{-1}}a_1^{-1}u_1u_2a_3a_4a_5=
a_4a_5a_3b^{-1}\underline{a_4a_1^{-1}u_1u_2a_3a_4a_5}=\\
&a_4a_5a_3b^{-1}a_1^{-1}u_1u_2a_3a_4a_5a_3\sim_R
a_4\underline{a_5a_3b^{-1}}a_1^{-1}u_1u_2a_3a_4a_5=\\
&a_4b^{-1}\underline{a_5a_3a_1^{-1}u_1u_2a_3a_4a_5}=
a_4b^{-1}a_3a_1^{-1}u_1u_2a_3a_4a_5a_4
\end{align*}
\begin{lemma}\label{inS1}
$a_5a_4a_3a_2^{-1}a_1^{-1}u_1u_2a_3a_4a_5\in\mathcal{S}$.
\end{lemma}
\begin{proof}
By (Da) and (C4a) we have 
\[u_5u_4u_3u_2u_1a_1a_2a_3a_4a_5=a_5u_5u_4u_3u_2u_1a_1a_2a_3a_4\in\mathcal{S}\]
and by (B1, B2, C3, C4a)
\[u_5u_4u_3(a_5a_4a_3)^{-1}=u_3^{-1}u_4^{-1}(a_5u_5)u_4u_3u_4^{-1}(a_5u_5)u_4(a_5u_5)\in\mathcal{S}\]
It follows that
\begin{align*}
&a_5a_4a_3a_2^{-1}a_1^{-1}u_1u_2a_3a_4a_5\approx 
(u_5u_4u_3)a_2^{-1}a_1^{-1}u_1u_2(u_5u_4u_3u_2u_1a_1a_2)^{-1}=\\
&(u_5u_4u_3)(u_2u_1a_1\underline{a_2u_2^{-1}}u_1^{-1}a_1a_2)^{-1}(u_5u_4u_3)^{-1}=\\
&(u_5u_4u_3)(\underline{u_2u_1a_1u_2^{-1}}a_2^{-1}\underline{u_1^{-1}a_1}a_2)^{-1}(u_5u_4u_3)^{-1}=\\
&(u_5u_4u_3)(u_1^{-1}u_2a_2\underline{u_1a_2^{-1}a_1^{-1}}u_1^{-1}a_2)^{-1}(u_5u_4u_3)^{-1}=\\
&(u_5u_4u_3)(u_1^{-1}u_2a_1^{-1}u_2^{-1}u_1^{-1}a_2)^{-1}(u_5u_4u_3)^{-1}=(u_5u_4u_3)u_1^2(u_5u_4u_3)^{-1}=u_1^2\approx 1
\qedhere\end{align*}
\end{proof}

\medskip

By Lemma \ref{inS1} $a_1^{-1}u_1u_2a_3a_4a_5\sim_R a_2a_3^{-1}a_4^{-1}a_5^{-1}$ and
\[w\sim_R c^{-1}u_2^{-1}u_1^{-1}ca_5^{-2}a_3^{-1}(a_4a_5a_3a_4)^{-1}ba_4b^{-1}a_3a_2a_3^{-1}a_4^{-1}a_5^{-1}.\]
We have
\begin{align*}
&\underline{ba_4b^{-1}a_3a_2a_3^{-1}}a_4^{-1}a_5^{-1}=
a_4^{-1}ba_4a_2^{-1}a_3a_2a_4^{-1}a_5^{-1}=
a_4^{-1}a_2^{-1}b\underline{a_4a_3a_4^{-1}}a_5^{-1}a_2=\\
&a_4^{-1}a_2^{-1}ba_3^{-1}a_4a_3a_5^{-1}a_2=
a_4^{-1}a_2^{-1}a_3^{-1}ba_4a_5^{-1}a_3a_2
\sim_R a_4^{-1}a_2^{-1}a_3^{-1}ba_4a_5^{-1}
\end{align*}
and thus
\[
w\sim_R c^{-1}u_2^{-1}u_1^{-1}ca_5^{-2}a_3^{-1}(a_4a_5a_3a_4)^{-1}a_4^{-1}a_2^{-1}a_3^{-1}ba_4a_5^{-1}
\]
Let $s=a_1\cdots a_5$. By (A1, A2, C1a, C5a) for $i>1$ we have
$a_is=sa_{i-1}$ and $u_is=su^{-1}_{i-1}$.
We also have $c=s^2bs^{-2}\stackrel{(E6)}{=}s^{-4}bs^4$ and
\begin{align*}
&c^{-1}u_2^{-1}u_1^{-1}c=
s^{-4}b^{-1}\underline{s^4u_2^{-1}u_1^{-1}s^{-4}}bs^4=
s^{-4}b^{-1}su_5u_4s^{-1}bs^4=\\
&s^{-3}a_5^{-1}a_4^{-1}b^{-1}a_4a_5u_5u_4a_5^{-1}a_4^{-1}ba_4a_5s^3=\\
&s^{-3}a_5^{-1}a_4^{-1}a_1a_2a_1b^{-1}a_4a_5u_5u_4a_5^{-1}a_4^{-1}ba_4a_5a_3a_4a_5a_2a_3a_4a_5a_1a_2a_3a_4a_5
\end{align*}
Write $w\approx ABC$ for $A=s^{-3}a_5^{-1}a_4^{-1}a_1a_2a_1$, $B=b^{-1}a_4a_5u_5u_4a_5^{-1}a_4^{-1}b$ and
\[C=(a_4a_5a_3a_4a_5a_2a_3a_4a_5a_1a_2a_3a_4a_5)(a_5^{-2}a_3^{-1}(a_4a_5a_3a_4)^{-1}a_4^{-1}a_2^{-1}a_3^{-1}ba_4a_5^{-1})\]
We have 
\begin{align*}
&A=a_2^{-1}a_1^{-1}s^{-3}a_1a_2a_1\sim_L 
s^{-3}a_1a_2a_1=
s^{-1}(a_3a_4a_5a_2a_3a_4a_5)^{-1}=\\
&(a_2a_3a_4a_1a_2a_3a_4)^{-1}s^{-1}\sim_L s^{-1}
\end{align*}
\begin{align*}
&C=a_4a_5a_3a_4\underline{a_5a_2a_3a_4a_5a_1a_2a_3a_4a_5^{-1}a_3^{-1}}a_4^{-1}a_3^{-1}a_5^{-1}a_4^{-2}a_2^{-1}a_3^{-1}ba_4a_5^{-1}=\\
&a_4a_5a_3a_4a_2a_3a_4\underline{a_5a_1a_2a_3a_4a_5^{-1}a_4^{-1}a_3^{-1}a_5^{-1}}a_4^{-2}a_2^{-1}a_3^{-1}ba_4a_5^{-1}=\\
&a_4a_5a_3a_4a_2a_3a_4a_1a_2a_3\underline{a_5a_4a_5^{-1}a_4^{-1}a_5^{-1}}a_3^{-1}a_4^{-2}a_2^{-1}a_3^{-1}ba_4a_5^{-1}=\\
&a_4a_5a_3a_4a_2a_3a_1a_2\underline{a_4a_3a_4^{-1}a_3^{-1}a_4^{-2}}a_2^{-1}a_3^{-1}ba_4a_5^{-1}=\\
&a_4a_5a_3\underline{a_4a_2a_3a_1}a_2a_3^{-1}a_4^{-1}a_2^{-1}a_3^{-1}ba_4a_5^{-1}=\\
&a_4a_5a_3a_2a_1\underline{a_4a_3a_2a_3^{-1}a_4^{-1}a_2^{-1}a_3^{-1}}ba_4a_5^{-1}=
a_4a_5a_3\underline{a_2a_1a_2^{-1}}a_3^{-1}\underline{a_4ba_4}a_5^{-1}=\\
&a_4a_5a_3a_1^{-1}a_2a_1a_3^{-1}ba_4a_5^{-1}b\sim_R a_1^{-1}a_4a_5\underline{a_3a_2a_3^{-1}}ba_4a_5^{-1}=\\
&a_1^{-1}a_4a_5a_2^{-1}a_3a_2ba_4a_5^{-1}\sim_R a_1^{-1}a_2^{-1}a_4a_5a_3ba_4a_5^{-1}=\\
&a_1^{-1}a_2^{-1}a_4a_3ba_5a_4a_5^{-1}\sim_R a_1^{-1}a_2^{-1}a_4a_3ba_4^{-1}a_5
\end{align*}

\begin{align*}
w\approx &s^{-1}(b^{-1}a_4a_5u_5u_4a_5^{-1}a_4^{-1}b)(a_1^{-1}a_2^{-1}a_4a_3ba_4^{-1}a_5)=\\
&s^{-1}a_1^{-1}a_2^{-1}b^{-1}a_4a_5u_5u_4a_5^{-1}a_4^{-1}\underline{ba_4b}a_3a_4^{-1}a_5=\\
&s^{-1}a_1^{-1}a_2^{-1}b^{-1}a_4a_5u_5u_4ba_5^{-1}\underline{a_4a_3a_4^{-1}}a_5=\\
&s^{-1}a_1^{-1}a_2^{-1}b^{-1}a_4a_5u_5u_4ba_3^{-1}\underline{a_5^{-1}a_4a_5}a_3=\\
&s^{-1}a_1^{-1}a_2^{-1}b^{-1}a_4a_5u_5u_4ba_3^{-1}a_4a_5a_4^{-1}a_3\sim_R\\ &s^{-1}a_1^{-1}a_2^{-1}(b^{-1}a_4a_5u_5u_4b)a_3^{-1}a_4a_5
\end{align*}
\begin{lemma}\label{inS2}In $\G$ we have
\[b^{-1}(a_4a_5u_5u_4)^{-1}b=a_4a_5u_4^{-1}vu_4va_5^{-1}a_4^{-1},\]
where $v=a_3a_2a_1u_1u_2u_3$.
\end{lemma}
\begin{proof}
Let $y_4=a_4u_4$, $x=u_5y_4u_5^{-1}$, $z=a_4vu_4$.
By (B1,C3) we have
\[a_4a_5u_5u_4=a_4u_4(u_4^{-1}a_5u_5u_4)=a_4u_4(u_5a_4u_4u_5^{-1})=y_4x\]
and by (C8, C9) $y_4^{-1}by_4=bz$. Conjugating the last relation by $u_5$ and by
$x^{-1}$ we obtain, using (C7)
\begin{align*}
&u_5y_4^{-1}by_4u_5^{-1}=bu_5zu_5^{-1}\iff x^{-1}bx=bu_5zu_5^{-1}\\
&x^{-1}y_4^{-1}by_4x=x^{-1}bzx=bu_5zu_5^{-1}x^{-1}zx
\end{align*}
The last relation is equivalent to
\begin{align*}
&b^{-1}(y_4x)^{-1}b=u_5zu_5^{-1}x^{-1}zy_4^{-1}=
u_5zy_4^{-1}u_5^{-1}zy_4^{-1}=
u_5a_4v\underline{a_4^{-1}u_5^{-1}a_4}va_4^{-1}\stackrel{(C5a)}{=}\\
&u_5a_4\underline{va_5}u_4\underline{a_5^{-1}v}a_4^{-1}\stackrel{(A1,C1a)}{=}
\underline{u_5a_4a_5}vu_4va_5^{-1}a_4^{-1}\stackrel{(C5a)}{=}a_4a_5u_4^{-1}vu_4va_5^{-1}a_4^{-1}
\qedhere\end{align*}
\end{proof}

By Lemma \ref{inS2}
\begin{align*}
&w^{-1}\approx \underline{a_5^{-1}a_4^{-1}a_3a_4a_5}u_4^{-1}vu_4v\underline{a_5^{-1}a_4^{-1}a_2a_1s}=\\
&a_3a_4a_5\underline{a_4^{-1}a_3^{-1}u_4^{-1}}vu_4va_2a_1sa_4^{-1}a_3^{-1}\stackrel{(C5a)}{=}
a_3a_4a_5u_3a_4^{-1}a_3^{-1}vu_4va_2a_1sa_4^{-1}a_3^{-1}\\
&\approx a_5a_4^{-1}a_2a_1u_1u_2u_3u_4a_3a_2a_1u_1u_2u_3a_2a_1s\sim_L a_5\underline{a_4^{-1}u_3u_4a_3}a_2a_1\underline{u_1u_2u_3a_2a_1}s\\
&=\underline{a_5u_3u_4a_2a_1}a_3a_2u_1\underline{u_2u_3s}=u_3a_2a_1a_5u_4a_3a_2u_1su_1^{-1}u_2^{-1}
\approx a_5u_4a_3a_2u_1s
\end{align*}
Since $s\sim_R(u_5u_4u_3u_2u_1)^{-1}$ (see the proof of Lemma \ref{inS1}) thus
\begin{align*}
&w\approx a_5u_4a_3a_2u_2^{-1}u_3^{-1}u_4^{-1}u_5^{-1}=
a_5u_5^{-1}(u_5u_4a_3u_3^{-1}u_4^{-1}u_5^{-1})
(u_5u_4u_3a_2u_2^{-1}u_3^{-1}u_4^{-1}u_5^{-1})=\\
&(a_5u_5)^{-1}u_3^{-1}u_4^{-1}(a_5u_5)^{-1}u_4u_3
u_2^{-1}u_3^{-1}u_4^{-1}(a_5u_5)^{-1}u_4u_3u_2\approx 1. 
\end{align*}
Thus ($\widetilde{\textrm{R8'}}$) is expressible in $\mathcal{S}$, which completes the proof of Lemma \ref{c6r8} and the proof of Theorem \ref{pres_stab_alpha1}.
\section{Edges.}\label{sec_edges}
In this section we assume that $g\ge 5$ is fixed and denote 
$\M(N_{g,0})$ as $\M$, $\G_{g,0}$ as $\G$, $\varphi_{g,0}$ as $\varphi$, and
$\Stab_{\M}\sigma$ as $\Stab\,\sigma$ for each simplex $\sigma$ of $\widetilde{X}$.
We are ready to define $\psi\colon\M\to\G$ on the generators of $\M$ given in Theorem \ref{Brown}.
In previous sections we defined homomorphisms
$\psi_{v_i}\colon\Stab\, s(v_i)\to\G_{g,0}$  and we let
$\psi$ be equal to $\psi_{v_i}$ on $\Stab\,s(v_i)$ for $i\in\{1,2,3\}$. For
$j\in\{1,\dots,7\}$ we define $\psi(h_{e_i})$ to be the element of $\G$ represented by the word in the generators of $\G$ given in the fourth column of Table \ref{tabE}, and $\psi(h_{\overline{e_i}})=\psi(h_{e_i})^{-1}$.
Observe that 
$\varphi\circ\psi$ is the identity on the generators of $\M$.
In this section we show that $\psi$ respects the relations associated to the edges of $X$. Namely, we show that for $e\in\cS_1(X)$ we have
\[(\ast)\quad \psi(h_e)^{-1}\psi_{i(e)}(x)\psi(h_e)=\psi_{t(e)}(h_e^{-1}xh_e)\] 
for $x\in\Stab\, s(e)$. Since $\Stab\,s(\overline{e})=h_e^{-1}\Stab\,s(e)h_e$ and
$h_{\overline{e}}=h^{-1}_e$, thus it suffices to check $(\ast)$ for
$e=e_i$, $i\in\{1,\dots,7\}$. 

To prove 
$(\ast)$ it suffices to show that its left hand side is equal in $\G$ to $\psi_{t(e)}(z)$ for some $z\in\Stab\,s(t(e))$, because then by applying $\varphi$ to both sides we get $z=h_e^{-1}xh_e$. 
\begin{lemma}\label{RelE1}
For $x\in\Stab[\alpha_1,\mu_g]$ we have $(\ast)\ \psi_{v_1}(x)=\psi_{v_2}(x)$.
\end{lemma}
\begin{proof}
By the proof of Theorem \ref{pres_stab_alpha1}, $\Stab[\alpha_1]$ is generated
by $\Stab^+[\alpha_1]$ and $\{u_1, r_g\}$. Note that 
$\{u_1, r_g\}\subset\Stab[\mu_g]$ and $\psi_{v_1}(x)=\psi_{v_2}(x)$ for
$x\in\{u_1, r_g\}$. It remains to show that the same is true for
$x\in H=\Stab^+[\alpha_1]\cap\Stab[\mu_g]$.
Let $N'$ be the surface obtained from $N=N_{g,0}$ by cutting along $\mu_g$
and gluing a disc with puncture $P$ along the resulting boundary component.
We have the exact sequence
(\ref{stab_bir_es}):
\[1\to\pi_1(N'_{\alpha_1},P)\stackrel{\mathfrak{c}}\to\Stab_{\M(N_{\alpha_1})}[\mu_g]\stackrel{\zeta}{\to}\M(N'_{\alpha_1})\to 1,\]
where $N'_{\alpha_1}$ and $N_{\alpha_1}$ are the surfaces obtained respectively from $N'$ and $N$ by cutting along $\alpha_1$. Set $G=\Stab_{\M(N_{\alpha_1})}[\mu_g]$ and note that
$\rho_{\alpha_1}(G)=H$ and $\rho_{\alpha_1}(\M(N'_{\alpha_1}))=\Stab^+_{\M(N')}[\alpha_1]$.
Observe that $\zeta$ maps $\ker\rho_{\alpha_1}\subset G$ isomorphically onto 
$\ker\rho_{\alpha_1}\subset\M(N'_{\alpha_1})$. It follows that $\zeta$ induces a map $\zeta'\colon H\to\Stab^+_{\M(N')}[\alpha_1]$, which fits in the following commutative diagram
\[
\begin{CD}
1 @>>> \pi_1(N'_{\alpha_1},P) @>\mathfrak{c}>> G @>\zeta>> \M(N'_{\alpha_1}) @>>> 1\\
@. @| @VV\rho_{\alpha_1}V @VV\rho_{\alpha_1}V \\
1 @>>> \pi_1(N'_{\alpha_1},P) @>\rho_{\alpha_1}\circ\mathfrak{c}>> H @>\zeta'>> \Stab^+_{\M(N')}[\alpha_1] @>>> 1, 
\end{CD}\]
whose both rows are exact. We can obtain generators of $H$ from the bottom sequence. Note that $N'$ is homeomorphic to $N_{g-1,0}$ and
by the proof of Theorem \ref{pres_stab_alpha1} (see Step 6), $\Stab^+_{\M(N')}[\alpha_1]$ is
generated by $u_i$, $a_i$ for $i=3,\dots,g-2$, $a_1$, $b$, $v$ and $c$ (if $g\ge 7$). The standard generators of $\pi_1(N'_{\alpha_1},P)$ are mapped by $\rho_{\alpha_1}\circ\mathfrak{c}$ on the crosscap slides
$Y_{\mu_g,\alpha_{g-1}}=a_{g-1}u_{g-1}$, $Y_{\mu_g,\gamma_{\{i,g\}}}=u_iY_{\mu_g,\gamma_{\{i+1,g\}}}u_i^{-1}$ for
$i=3,\dots,g-2$ and $Y_{\mu_g,\gamma_{\{1,2,g\}}}$.
It follows that $H$ is generated by
$Y_{\mu_g,\gamma_{\{1,2,g\}}}$, $a_{g-1}u_{g-1}$, $u_i$, $a_i$ for $i=3,\dots,g-2$, $a_1$, $b$, $v$ and $c$ (if $g\ge 7$).
By the definitions of $\psi_{v_1}$ and $\psi_{v_2}$ given in Theorems \ref{pres_stab_alpha1} and \ref{pres_Stab_mu_g}, it is easy to check that
$\psi_{v_1}(x)=\psi_{v_2}(x)$ for every generator $x$ of $H$, except for
$x=Y_{\mu_g,\gamma_{\{1,2,g\}}}$.
We have 
$Y_{\mu_g,\gamma_{\{1,2,g\}}}=(u_4\cdots u_{g-1})^{-1}Y_{\mu_4,\gamma_{\{1,2,4\}}}
(u_4\cdots u_{g-1})$
and\\
$Y_{\mu_4,\gamma_{\{1,2,4\}}}=Y_{\mu_4,\alpha_3}^{-1}Y_{\mu_4,\beta}=(a_3u_3)^{-1}v$.
Thus, by Theorem \ref{pres_stab_alpha1} we have
\begin{align*}
\psi_{v_1}(Y_{\mu_g,\gamma_{\{1,2,g\}}})&=(u_4\cdots u_{g-1})^{-1}(a_3u_3)^{-1}a_3a_2a_1u_1u_2u_3
(u_4\cdots u_{g-1})\\
&=(u_3\cdots u_{g-1})^{-1}a_2a_1u_1u_2(u_3\cdots u_{g-1})\\
&=(u_3\cdots u_{g-1})^{-1}(a_2u_2)u_2^{-1}(a_1u_1)u_2
(u_3\cdots u_{g-1})
\end{align*}
and it follows from (B1, B2, C3) that 
\begin{align*}
&\psi_{v_1}(Y_{\mu_g,\gamma_{\{1,2,g\}}})=
(u_2\cdots u_{g-2})a_{g-1}u_{g-1}(u_2\cdots u_{g-2})^{-1}\cdot\\
&(u_1\cdots u_{g-2})a_{g-1}u_{g-1}(u_1\cdots u_{g-2})^{-1}
\end{align*}
It follows that $\psi_{v_1}(Y_{\mu_g,\gamma_{\{1,2,g\}}})$ is in the image of $\psi_{v_2}$ and thus it is equal to
$\psi_{v_2}(Y_{\mu_g,\gamma_{\{1,2,g\}}})$ by the remark before Lemma \ref{RelE1}. 
\end{proof}

\begin{lemma}\label{RelE3}
For $x\in\Stab[\mu_{g-1},\mu_g]$ we have \[(\ast)\quad \psi(h_{e_3})^{-1}\psi_{v_2}(x)\psi(h_{e_3})=\psi_{v_2}(h^{-1}_{e_3}xh_{e_3}).\]
\end{lemma}
\begin{proof}
To obtain generators of $\Stab[\mu_{g-1},\mu_g]$ we use the exact sequence (\ref{stab_bir_es})
\[1\to\pi_1(N'\backslash\{P_1\},P_2)\to\Stab[\mu_{g-1},\mu_{g}]\to\Stab_{\M(N_{g-1,0})}[\mu_{g-1}]\to 1\]
where $N'$ is obtained from $N_{g,0}$ by cutting along $\mu_{g-2+i}$ and gluing a disc with puncture $P_i$ along the resulting boundary component for $i\in\{1,2\}$. By Lemma \ref{ext_pres} and the proof of Theorem \ref{pres_Stab_mu_g},  $\Stab[\mu_{g-1},\mu_g]$ is generated by
\[Z=\{u_i, a_i\,|\,i=1,\dots,g-3\}\cup\{b, a_{g-2}u_{g-2}, u_{g-2}a_{g-1}u_{g-1}u_{g-2}^{-1}\}.\]
We have $h_{e_3}=a_{g-1}^{-1}$ and
\[
\psi(h_{e_3})^{-1}\psi_{v_2}(u_2)\psi(h_{e_3})=a_{g-1}u_2a_{g-1}^{-1}=u_2=\psi_{v_2}(u_2)=\psi_{v_2}(h^{-1}_{e_3}u_2h_{e_3}).
\]
Analogously $\psi(h_{e_3})^{-1}\psi_{v_2}(a_2)\psi(h_{e_3})=a_2=\psi_{v_2}(h^{-1}_{e_3}u_2h_{e_3})$.
For $x\in Z\backslash\{a_2, u_2\}$ we have $x\in\Stab[\alpha_1]$ and 
$\psi_{v_2}(x)=\psi_{v_1}(x)$ by Lemma \ref{RelE1}. Since also $h_{e_3}\in\Stab[\alpha_1]$ and 
$\psi(h_{e_3})=\psi_{v_1}(h_{e_3})$, thus
\[
\psi(h_{e_3})^{-1}\psi_{v_2}(x)\psi(h_{e_3})=\psi_{v_1}(h^{-1}_{e_3}xh_{e_3})=\psi_{v_2}(h^{-1}_{e_3}xh_{e_3}).
\qedhere\]
\end{proof}
The next lemma follows from \cite[Proposition 2.10]{LabPar}.
\begin{lemma}\label{torus_gens}
Let $S=S_{1,r}$ be a torus with $r>1$ boundary components
$\delta_1,\dots,\delta_r$. Suppose that
$\alpha_1,\dots,\alpha_r$ and $\beta$ are simple closed curves on $S$ such that
(1) 
$\alpha_i$, $\alpha_{i+1}$, $\delta_i$ bound a pair of pants for $i=1,\dots,r$ and
$\alpha_{r+1}=\alpha_1$; (2)
$\beta$ intersects each of the curves $\alpha_i$ in one point. Then $\M(S)$ is generated by Dehn twists about
$\beta$, $\alpha_i$, $\delta_i$ for $i=1,\dots, r$.\hfill{$\Box$}
\end{lemma}

\begin{lemma}\label{RelE4}
If $g\in\{5,6\}$ then
$(\ast)\ \psi_{v_1}(x)=\psi_{v_3}(x)$ for $x\in\Stab[\alpha_1,\xi]$.
\end{lemma}
\begin{proof}
Denote by $S$ the surface obtained by cutting $N_{g,0}$ along $\alpha_1\cup\xi$. Note that $S$ is homeomorphic to $S_{1,g-2}$.
Recall the exact sequence (\ref{stab_es})
\[1\to\Stab^+[\alpha_1,\xi]\to\Stab[\alpha_1,\xi]\stackrel{\eta}{\to}\Z_2^{g-2}.\]
By Remark \ref{eta_onto}  $\eta$ is not onto and hence its image has rank at most $g-3$. 
It follows that this image is spanned by the images of  the following elements of $\Stab[\alpha_1,\xi]$:
$x_1=a_2a_1^2a_2$ (swaps sides and reverses orientation of $\alpha_1$,
preserves sides and orientation of $\xi$), 
 $x_2=a_2a_1a_3a_2\Delta_4$ for $g=5$ or  $x_2=a_2a_1a_3a_2(a_5u_5)^{-1}\Delta_4$ for $g=6$ (swaps sides and preserves  orientation of $\alpha_1$, reverses orientation of $\xi$ and swaps its sides if $g=6$), 
$x_3=r_6$ for $g=6$ (swaps sides and reverses orientation of $\alpha_1$,
swaps sides  and preserves orientation of $\xi$).
By Lemma \ref{ext_pres} $\Stab[\alpha_1,\xi]$ is generated by  $x_i$ for $i=1,2,3$  and $\Stab^+[\alpha_1,\xi]$.
It follows from Lemma \ref{torus_gens} that  $\Stab^+[\alpha_1,\xi]$ is generated by 
$a_i$ for $i=1,3,\dots,g-1$, $b$, 
$x_2bx_2^{-1}$
and if $g=6$ then also $b_2$. The generator $x_2bx_2^{-1}$ is redundant and it is trivial to check that $\psi_{v_1}$ and $\psi_{v_3}$ are equal on the remaining generators of $\Stab^+[\alpha_1,\xi]$. For $i=1,2,3$ we have $x_i\in\Stab[\mu_g]$, it is easy to check that
$\psi_{v_3}(x_i)=\psi_{v_2}(x_i)$ and by Lemma \ref{RelE1} we have
$\psi_{v_2}(x_i)=\psi_{v_1}(x_i)$.
\end{proof}
\begin{lemma}\label{gensE2}
$\Stab[\alpha_1,\alpha_3]$ is generated by
\begin{itemize}
\item $\Stab[\alpha_1,\alpha_3,\mu_g]\cup\Stab[\alpha_1,\alpha_3,\xi]$ if $g=5$ or $g=6$,
\item $\Stab[\alpha_1,\alpha_3,\mu_g]\cup\{a_{g-1}\}$ if $g=7$ or $g\ge 9$.
\item $\Stab[\alpha_1,\alpha_3,\mu_g]\cup\{a_7, b_3, T_{\gamma_{\{5,6,7,8\}}}, T_{\gamma_{\{1,2,5,6,7,8\}}}\}$  if $g=8$.
\end{itemize}
\end{lemma}
\begin{proof}
Let $\Stab=\Stab[\alpha_1,\alpha_3]$.
By Lemma \ref{ext_pres} applied to the sequence (\ref{stab_es}) $\Stab$ is generated by $\Stab^+$ and $4$ cokernel generators, for which we take  $u_1$, $u_3$, $r_g$ and $a_4a_3^2a_4$. Note that all of them are in $\Stab[\mu_g]$ if $g>5$ and if $g=5$ then the last one is in $\Stab[\xi]$. 
Let $N'$ be the surface obtained by cutting $N_{g,0}$ along
$\alpha_1\cup\alpha_3$. Note that $N'$ is homeomorphic to $N_{g-4,4}$. Denote by $\alpha_1'$ and $\alpha_1''$ (resp. $\alpha_3'$ and $\alpha_3''$) the boundary components of $N'$ resulting from cutting along $\alpha_1$ (resp. $\alpha_3$), where $\alpha_1'$, $\alpha_3'$ and $\beta$ bound a pair of pants in $N'$. 
To obtain generators of $\Stab^+$ we use the epimorphism $\rho\colon\M(N')\to\Stab^+$ induced by the gluing map. We consider cases according to the genus.
\begin{figure}
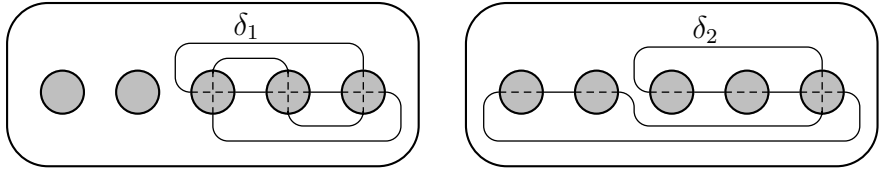

\begin{tabular}{cc}
\pspicture*(5.75,2.2)
\psframe[framearc=.5](0,0)(5.5,2.2)
\pscircle*[linecolor=lightgray](.75,1){.3}
\pscircle(.75,1){.3}
\pscircle*[linecolor=lightgray](1.75,1){.3}
\pscircle(1.75,1){.3}
\pscircle*[linecolor=lightgray](2.75,1){.3}
\pscircle(2.75,1){.3}
\psline[linewidth=.5pt,linestyle=dashed,dash=3pt 2pt](2.45,1)(3.05,1)
\psline[linewidth=.5pt,linestyle=dashed,dash=3pt 2pt](2.75,.7)(2.75,1.3)
\pscircle*[linecolor=lightgray](3.75,1){.3}
\pscircle(3.75,1){.3}
\psline[linewidth=.5pt,linestyle=dashed,dash=3pt 2pt](3.45,1)(4.05,1)
\psline[linewidth=.5pt,linestyle=dashed,dash=3pt 2pt](3.75,.7)(3.75,1.3)
\pscircle*[linecolor=lightgray](4.75,1){.3}
\pscircle(4.75,1){.3}
\psline[linewidth=.5pt,linestyle=dashed,dash=3pt 2pt](4.45,1)(5.05,1)
\psline[linewidth=.5pt,linestyle=dashed,dash=3pt 2pt](4.75,.7)(4.75,1.3)
\psline[linewidth=.5pt,linearc=.2](2.75,.7)(2.75,.35)(5.25,.35)(5.25,1)(5.05,1)
\psline[linewidth=.5pt,linearc=.2](3.75,.7)(3.75,.55)(4.75,.55)(4.75,.7)
\psline[linewidth=.5pt,linearc=.2](3.75,1.3)(3.75,1.45)(2.75,1.45)(2.75,1.3)
\psline[linewidth=.5pt,linearc=.2](4.75,1.3)(4.75,1.65)(2.25,1.65)(2.25,1)(2.45,1)
\psline[linewidth=.5pt,linearc=.2](3.05,1)(3.45,1)
\psline[linewidth=.5pt,linearc=.2](4.05,1)(4.45,1)
\rput[b](3.2,1.7){$\delta_1$}
%
%
%
%
%
%
\endpspicture&\pspicture*(5.75,2.2)
\psframe[framearc=.5](0,0)(5.5,2.2)
\pscircle*[linecolor=lightgray](.75,1){.3}
\pscircle(.75,1){.3}
\psline[linewidth=.5pt,linestyle=dashed,dash=3pt 2pt](.45,1)(1.05,1)
\pscircle*[linecolor=lightgray](1.75,1){.3}
\pscircle(1.75,1){.3}
\psline[linewidth=.5pt,linestyle=dashed,dash=3pt 2pt](1.45,1)(2.05,1)
\pscircle*[linecolor=lightgray](2.75,1){.3}
\pscircle(2.75,1){.3}
\psline[linewidth=.5pt,linestyle=dashed,dash=3pt 2pt](2.45,1)(3.05,1)
\pscircle*[linecolor=lightgray](3.75,1){.3}
\pscircle(3.75,1){.3}
\psline[linewidth=.5pt,linestyle=dashed,dash=3pt 2pt](3.45,1)(4.05,1)
\pscircle*[linecolor=lightgray](4.75,1){.3}
\pscircle(4.75,1){.3}
\psline[linewidth=.5pt,linestyle=dashed,dash=3pt 2pt](4.45,1)(5.05,1)
\psline[linewidth=.5pt,linestyle=dashed,dash=3pt 2pt](4.75,.7)(4.75,1.3)
\psline[linewidth=.5pt,linearc=.2](.45,1)(.25,1)(.25,.35)(5.25,.35)(5.25,1)(5.05,1)
\psline[linewidth=.5pt,linearc=.2](1.05,1)(1.45,1)
\psline[linewidth=.5pt,linearc=.2](2.05,1)(2.25,1)(2.25,.55)(4.75,.55)(4.75,.7)
\psline[linewidth=.5pt,linearc=.2](4.75,1.3)(4.75,1.6)(2.25,1.6)(2.25,1)(2.45,1)
\psline[linewidth=.5pt,linearc=.2](3.05,1)(3.45,1)
\psline[linewidth=.5pt,linearc=.2](4.05,1)(4.45,1)
\rput[b](3.2,1.65){$\delta_2$}
%
%
%
%
%
%
\endpspicture
\end{tabular}
\caption{\label{fig_N14}Curves from the proof of Lemma \ref{gensE2}.}
\end{figure}

{\bf Case $g=5$.}
In \cite[Theorem 7.6]{Szep1} a presentation 
of $\M(N_{1,4})$ is given, from which we deduce that $\Stab^+$ is generated by Dehn twists about curves disjoint from $\mu_g$  and
$T_{\delta_1}$, $T_{\delta_2}$, $u_3^{-1}T_{\delta_2}u_3$, where $\delta_1$, $\delta_2$ are shown on Figure \ref{fig_N14}. Since $\delta_1$ and $\delta_2$ are disjoint from $\xi$, the lemma is proved in this case.

{\bf Case $g>5$.} Let $N''$ be the surface obtained from $N'$ by gluing a disc with puncture $P_1$ along $\alpha_1'$, a disc with puncture $P_2$ along $\alpha_3'$, and a disc with puncture $P_3$ along $\alpha_3''$. For $i=1,2,3$ we set $\mathcal{P}_i=\{P_1,\dots, P_i\}$ and $H_i=\PM^+(N'',\mathcal{P}_i)$. For $i=2,3$ we set $K_i=\pi_1(N''\backslash\mathcal{P}_{i-1},P_i)$ and define $G_i$ to be the subgroup of $\PM(N'',\mathcal{P}_i)$ consisting of the isotopy classes of homeomorphisms preserving local orientation at each puncture in $\mathcal{P}_{i-1}$. Note that $H_i$ is an index-two subgroup of $G_i$ and we have the following short exact sequence
\begin{equation}\label{Bir_red}
1\to K_i\stackrel{\mathfrak{p}}{\to}G_i\stackrel{\mathfrak{f}}{\to}H_{i-1}\to 1
\end{equation} 
which is a restriction of the Birman sequence (\ref{Bir_es}).
We also have the exact sequence (\ref{Cup_es})
\[1\to\Z^3\to\M(N')\stackrel{\imath_\ast}{\to}H_3\to 1.\]
The kernel generators of $\M(N')$ are Dehn twists about the boundary components and they are mapped by $\rho$ on $a_1$ and $a_3$, which are in $\Stab[\mu_g]$. Also observe that if $\imath_\ast(x)\in \Stab_{H_3}[\mu_g]$ then
$\rho(x)\in\Stab[\mu_g]$.
\begin{figure}
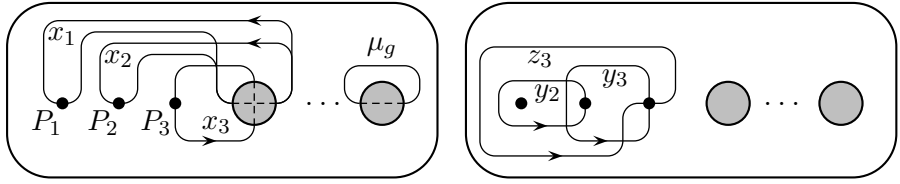

\begin{tabular}{cc}
\input{fig12}&\input{fig13}
\end{tabular}
\caption{\label{fig_loops}Loops from the proof of Lemma \ref{gensE2}.}
\end{figure}
Our next goal is to find generators for $H_3$. In fact we will obtain generators for the groups $H_i$, $G_i$ for $i=2,3$ by using the sequence (\ref{Bir_red}) and the method from Step 1 of the proof of Theorem \ref{pres_stab_alpha1}.
We set \[w=Y_{\mu_4,\alpha_4}Y_{\mu_3,\gamma_{\{3,5\}}}=Y_{\mu_4,\alpha_4}u_3Y_{\mu_4,\alpha_4}u_3^{-1}=u_4a_4u_3u_4a_4u_3^{-1}=a_4^{-1}a_3^{-1}u_3u_4.\]
Let $x_1, x_2, x_3$ and $y_2, y_3, z_3$ be the elements of
$\PM(N'',\mathcal{P}_i)$ obtained by pushing the punctures once along the loops represented in Figure \ref{fig_loops}, which we take so that for each
$I\subseteq\{5,\dots,g\}$ such that $5\in I$, the following equalities are satisfied, up to isotopy on $N''\backslash\mathcal{P}_3$.
\begin{align*}
& x_1(\gamma_I)=\gamma_{\{1,2\}\cup I},\quad x_2(\gamma_I)=\gamma_{\{3,4\}\cup I},\quad x_3(\gamma_I)=u_3^{-1}x_2(\gamma_I),\quad 
x_2x_1(\gamma_I)=\gamma_{\{1,2,3,4\}\cup I},\\
& x_3x_1(\gamma_I)=u_3^{-1}x_2x_1(\gamma_I),\quad 
x_3x_2(\gamma_I)=w(\gamma_I)\quad x_3x_2x_1(\gamma_I)=wx_1(\gamma_I).
\end{align*}
By Remark \ref{gensNg11} $H_1=\M^+(N'',P_1)$ is generated by
$a_i$ for $i=5,\dots,g-1$,
$u_j$ for $j=5,\dots,g-2$, $a_{g-1}u_{g-1}$, $x_1a_5x_1^{-1}$, $x_1a_5u_5x_1^{-1}$, and
$T_{\gamma_{\{5,6,7,8\}}}$ if $g\ge 8$.
We denote this set of generators by $Z_1$ and, by abuse of notation, we will treat it as a subset of $G_i$ for $i=1,2,3$. The image of $K_2$ in $G_2$ is generated by $y_2$ and $t_j$ for $j=5,\dots,g$ defined as $t_5=x_2$ and $t_{j+1}=u_j^{-1}t_ju_j$. From the sequence (\ref{Bir_red}) we obtain that $G_2$ is generated by $y_2$, $x_2$ and $Z_1$, and its index-two subgroup $H_2$ is generated by 
\[Z_2=\{y_2, x_2y_2x_2^{-1}, x_2^2\}\cup Z_1\cup x_2 Z_1x_2^{-1}.\]
Similarly, $G_3$ is generated by $y_3$, $z_3$, $x_3$ and $Z_2$, and  $H_3$ is generated by 
\[Z_3=\{y_3, z_3, x_3y_3x_3^{-1}, x_3z_3x_3^{-1},  x_3^2\}\cup Z_2\cup x_3 Z_2x_3^{-1}.\]
Set $Z_1'=Z_1\backslash\Stab_{H_3}[\mu_g]$ and observe that $H_3$ is generated
by \[\Stab_{H_3}[\mu_g]\cup Z_1'\cup x_2 Z'_1x_2^{-1}\cup x_3 Z'_1x_3^{-1}\cup x_3x_2 Z'_1x_2^{-1}x_3^{-1}.\]
{\bf Subcase $g=7$ or $g\ge 9$.}
 We have $Z'_1=\{a_{g-1}\}$ and  $a_{g-1}$ commutes with $x_2$, $x_3$. The lemma follows.

{\bf Subcase $g=6$.}
We have $Z'_1=\{a_5, x_1a_5x_1^{-1}\}$ and
\begin{align*}
& x_1a_5x_1^{-1}=T_{\gamma_{\{1,2,5,6\}}},\quad x_2a_5x_2^{-1}=T_{\gamma_{\{3,4,5,6\}}}=c,\quad x_3a_5x_3^{-1}=u_3^{-1}cu_3,\\
& (x_2x_1)a_5(x_2x_1)^{-1}=b_2,\quad (x_3x_1)a_5(x_3x_1)^{-1}=u_3^{-1}b_2u_3,\\
&(x_3x_2)a_5(x_3x_2)^{-1}=wa_5w^{-1}\quad (x_3x_2x_1)a_5(x_3x_2x_1)^{-1}=wT_{\gamma_{\{1,2,5,6\}}}w^{-1}.
\end{align*}
Since $a_5$, $c$, $b_2$ and $T_{\gamma_{\{1,2,5,6\}}}$ are in $\Stab[\alpha_1,\alpha_3,\xi]$ and
$u_3$, $w$ are in $\Stab[\alpha_1,\alpha_3,\mu_g]$, the lemma is proved for $g=6$.

{\bf Subcase $g=8$.} 
We have $Z'_1=\{a_7, T_{\gamma_{\{5,6,7,8\}}}\}$, $a_7$ commutes with $x_2$, $x_3$ and
\begin{align*}
& x_2T_{\gamma_{\{5,6,7,8\}}}x_2^{-1}=T_{\gamma_{\{3,4,5,6,7,8\}}},\quad x_3T_{\gamma_{\{5,6,7,8\}}}x_3^{-1}=u_3^{-1}T_{\gamma_{\{3,4,5,6,7,8\}}}u_3,\\
&(x_3x_2)T_{\gamma_{\{5,6,7,8\}}}(x_3x_2)^{-1}=wT_{\gamma_{\{5,6,7,8\}}}w^{-1}
\end{align*}
We have $u_3, w\in\Stab[\alpha_1,\alpha_3,\mu_g]$, and
to finish the proof it suffices to express $T_{\gamma_{\{3,\dots,8\}}}$ in terms of $b_3$ and the remaining generators. 
Observe that $\beta_3$, $\alpha_1$, $\alpha_3$ and $\gamma_{\{5,6,7,8\}}$ bound a four holed sphere and we have the lantern relation:
\[b_3a_1a_3T_{\gamma_{\{5,6,7,8\}}}=bT_{\gamma_{\{3,4,5,6,7,8\}}}T_{\gamma_{\{1,2,5,6,7,8\}}}\]
which does the job, because $a_1$, $a_3$ and $b$ are in $\Stab[\mu_8]$.
\end{proof}
\begin{lemma}\label{RelE2}
For $x\in\Stab[\alpha_1,\alpha_3]$ we have \[(\ast)\quad\psi(h_{e_2})^{-1}\psi_{v_1}(x)\psi(h_{e_2})=\psi_{v_1}(h^{-1}_{e_2}xh_{e_2}).\]
\end{lemma}
\begin{proof} Let $h=h_{e_2}=a_2a_3a_1a_2$.
We have $h\in\Stab[\mu_g]\cap\Stab[\xi]$ and
$\psi(h)=\psi_{v_2}(h)=\psi_{v_3}(h)$. Therefore for $x\in\Stab[\alpha_1,\alpha_3,\mu_g]$ we have by Lemma \ref{RelE1}
\[\psi(h)^{-1}\psi_{v_1}(x)\psi(h)=
\psi_{v_2}(h^{-1}xh)=
\psi_{v_1}(h^{-1}xh).\]
Analogously, if $g\in\{5,6\}$ then $(\ast)$ holds for $x\in\Stab[\alpha_1,\alpha_3,\xi]$ by Lemma \ref{RelE4}.
By Lemma \ref{gensE2} this finishes the proof for $g\in\{5,6\}$.
For $g\ge 7$ we have 
\[\psi(h)^{-1}\psi_{v_1}(a_{g-1})\psi(h)=(a_2a_1a_3a_2)^{-1}a_{g-1}(a_2a_1a_3a_2)=a_{g-1}=\psi_{v_1}(a_{g-1}).\]
It can be checked that for $g=8$ in $\M$ we have
\begin{align*}
&h^{-1}b_3h=b_3, \qquad h^{-1}T_{\gamma_{\{5,6,7,8\}}}h=T_{\gamma_{\{5,6,7,8\}}}\\ 
&h^{-1}T_{\gamma_{\{1,2,5,6,7,8\}}}h=T_{\gamma_{\{3,4,5,6,7,8\}}}=b^{-1}b_3a_1a_3T_{\gamma_{\{5,6,7,8\}}}T^{-1}_{\gamma_{\{1,2,5,6,7,8\}}}
\end{align*}
and for  $w=a_6a_7a_5a_6a_4a_5a_3a_4$,
\[T_{\gamma_{\{5,6,7,8\}}}=w^{-1}cw,\qquad
T_{\gamma_{\{1,2,5,6,7,8\}}}=w^{-1}b_2w,\]
It follows that for $x\in\{b_3, T_{\gamma_{\{5,6,7,8\}}}, T_{\gamma_{\{1,2,5,6,7,8\}}}\}$,
$\psi(h)$, $\psi_{v_1}(x)$ and $\psi_{v_1}(h^{-1}xh)$ are in the subgroup of $\G$ generated by $a_i$'s and $b_j$'s and so  $(\ast)$ is satisfied in $\G$ by  Lemma \ref{shortcut_AB}. 
\end{proof}

\begin{lemma}\label{RelE5}
If $g=5$ then for $x\in\Stab[\mu_5,\beta]$
\[(\ast)\ \psi(h_{e_5})^{-1}\psi_{v_2}(x)\psi(h_{e_5})=\psi_{v_1}(h^{-1}_{e_5}xh_{e_5})\] 
\end{lemma}
\begin{proof}
By similar argument as in the proof of Lemma \ref{RelE4}, $\Stab[\mu_5,\beta]$ is
generated by $\Stab^+[\mu_5,\beta]$ together with
$h_{e_5}u_1h_{e_5}^{-1}$ (preserves orientation of $\mu_5$, preserves orientation and swaps sides of $\beta$) and
$r_5\Delta_4$ (reverses orientation of $\mu_5$, reverses orientation and preserves sides of $\beta$). 
By Lemma  \ref{torus_gens} $\Stab^+[\mu_5,\beta]$ is generated by $b$, $a_1$, $a_2$, $a_3$ and $e=a_4u_4a_3u_4^{-1}a_4^{-1}$.
Note that $\{b, a_1, a_2, a_3, r_5\Delta_4\}\subset\Stab[\xi]$ and for $x=b, a_1, a_2, a_3$ we have $\psi_{v_2}(x)=\psi_{v_3}(x)$. Also $\psi_{v_2}(\Delta_4)=\psi_{v_3}(\Delta_4)$ and because $r_5\in\Stab[\alpha_1,\mu_5]\cap\Stab[\xi]$, thus $\psi_{v_2}(r_5)=\psi_{v_1}(r_5)=\psi_{v_3}(r_5)$ by Lemmas \ref{RelE1} and \ref{RelE4}. Also $h_{e_5}=a_4ba_3a_4a_2a_1a_3a_2\in\Stab[\xi]$ and $\psi(h_{e_5})=\psi_{v_3}(h_{e_5})$ and thus $(\ast)$ holds for $x\in\{b, a_1, a_2, a_3, r_5\Delta_4\}$  by Lemma \ref{RelE4}. 
Note that $h_{e_5}=a_4ba_3a_4h_{e_2}$ and $\psi(h_{e_5})=\psi_{v_1}(a_4ba_3a_4)\psi(h_{e_2})$. Also $\psi_{v_2}(e)=\psi_{v_1}(e)$ and 
\[\psi(h_{e_5})^{-1}\psi_{v_2}(e)\psi(h_{e_5})=\psi(h_{e_2})^{-1}\psi_{v_1}((a_4ba_3a_4)^{-1}e(a_4ba_3a_4))\psi(h_{e_2}).\]
It can be checked that $(a_4ba_3a_4)^{-1}e(a_4ba_3a_4)\in\Stab[\alpha_1,\alpha_3]$
and thus  $(\ast)$ holds for $x=e$ by Lemma \ref{RelE2}. Similarly, since $u_1\in\Stab[\alpha_1,\alpha_3]$ 
\[\psi(h_{e_5})\psi_{v_1}(u_1)\psi(h_{e_5})^{-1}=\psi_{v_1}(h_{e_5}u_1h_{e_5}^{-1})=\psi_{v_2}(h_{e_5}u_1h_{e_5}^{-1})\]
by Lemmas \ref{RelE1} and \ref{RelE2}.
\end{proof}

\begin{lemma}\label{RelE6}
If $g=6$ then for $x\in\Stab[\alpha_1,\gamma_{\{3,4,5,6\}}]$
\[(\ast)\ \psi(h_{e_6})^{-1}\psi_{v_1}(x)\psi(h_{e_6})=\psi_{v_1}(h^{-1}_{e_6}xh_{e_6})\]
\end{lemma}
\begin{proof}
Let $\gamma=\gamma_{\{3,4,5,6\}}$ and $d=u_3u_4u_5u_3u_4u_3$. 
By similar argument as in the proof of Lemma \ref{RelE4}, $\Stab[\alpha_1,\gamma]$
is
generated by $\Stab^+[\alpha_1,\gamma]$ together with
$u_1a_1$, $h_{e_6}u_1a_1h_{e_6}^{-1}$  and
$u_1r_6d$ 
(preserves sides  and reverses orientation of $\alpha_1$,
preserves sides and reverses orientation of $\gamma$).

By Lemma \ref{torus_gens} $\Stab^+[\alpha_1,\gamma]$ is generated by $c$, $a_i$ for $i\in\{1,3,4,5\}$, $r_6br_6$ and $d^{-1}bd$. Note that the last element can be expressed in terms of the other generators  of $\Stab[\alpha_1,\gamma]$.
For $x\in\{c, r_6br_6, u_1r_6d\}\cup\{ a_i\,|\,i=1,3,4,5\}$ we have $x\in\Stab[\xi]$, and since also 
$h_{e_6}\in\Stab[\xi]$ and $\psi(h_{e_6})=\psi_{v_3}(h_{e_6})$, thus $(\ast)$ holds by Lemma \ref{RelE4}. 

We have $u_1a_1=Y_{\mu_1,\alpha_1}$ and
\begin{align*}
&h_{e_6}u_1a_1h_{e_6}^{-1}=Y_{\mu_3,\gamma}=
Y_{\mu_3,\gamma_{\{3,6\}}}Y_{\mu_3,\gamma_{\{3,5\}}}Y_{\mu_3,\gamma_{\{3,4\}}}\\
&=(u_3u_4Y_{\mu_5,\alpha_5}u_4^{-1}u_3^{-1})(u_3Y_{\mu_4,\alpha_4}u_3^{-1})Y_{\mu_3,\alpha_3}
=u_3u_4u_5a_5a_4a_3.
\end{align*}
Hence $\psi_{v_1}(h_{e_6}u_1a_1h_{e_6}^{-1})=u_3u_4u_5a_5a_4a_3$. Recall that
$c=s^2bs^{-2}$ for $s=a_1\cdots a_5$. We have
\begin{align*}
&\psi(h_{e_6})\psi_{v_1}(u_1a_1)\psi(h_{e_6})^{-1}=
a_2ca_1a_2u_1a_1a_2^{-1}a_1^{-1}c^{-1}a_2^{-1}=
a_2cu_2^{-1}a_2c^{-1}a_2^{-1}\\
&=a_2s^2bs^{-2}u_2^{-1}a_2s^2b^{-1}s^{-2}a_2^{-1}=
sa_1sbs^{-1}u_1a_1sb^{-1}s^{-1}a_1^{-1}s^{-1}
\end{align*}
Thus for $x=h_{e_6}u_1a_1h_{e_6}^{-1}$, $(\ast)$ is equivalent to
\[a_1sbs^{-1}u_1a_1sb^{-1}s^{-1}a_1^{-1}=s^{-1}u_3u_4u_5a_5a_4a_3s.\]
The last relation holds in $\G$ because its both sides are in $\mathcal{S}(v_2)$. Indeed, we have 
 $sbs^{-1}=a_1a_2a_3a_4ba_4^{-1}a_3^{-1}a_2^{-1}a_1^{-1}$ and 
$s^{-1}u_3u_4u_5a_5a_4a_3s=u_2^{-1}u_3^{-1}u_4^{-1}a_4a_3a_2$.
It can be checked that
$h_{e_6}^2\in\Stab[\alpha_1]\cap\Stab[\xi]$ and by Lemma \ref{RelE4}
$\psi(h_{e_6})^2=\psi_{v_3}(h_{e_6}^2)=\psi_{v_1}(h_{e_6}^2)$. It follows that $(\ast)$ holds for $x=u_1a_1$ because
\[\psi(h_{e_6})^{-1}\psi_{v_1}(u_1a_1)\psi(h_{e_6})=\psi_{v_1}(h_{e_6}^{-2})\psi(h_{e_6})\psi_{v_1}(u_1a_1)\psi(h_{e_6})^{-1}\psi_{v_1}(h_{e_6}^2)\]
and the right hand side is equal to $\psi_{v_1}(h_{e_6}^{-1}u_1a_1h_{e_6})$ by earlier part of the proof.
\end{proof}

\begin{lemma}\label{RelE7}
If $g=6$ then for $x\in\Stab[\mu_6,\gamma_{\{1,2,3,4,5\}}]$
\[(\ast)\ \psi(h_{e_7})^{-1}\psi_{v_2}(x)\psi(h_{e_7})=\psi_{v_2}(h^{-1}_{e_7}xh_{e_7})\]
\end{lemma}
\begin{proof}
Let $\gamma=\gamma_{\{1,2,3,4,5\}}$. 
By similar argument as in the proof of Lemma \ref{RelE4}, $\Stab[\mu_6,\gamma]$
is
generated by $\Stab^+[\mu_6,\gamma]$ and an element reversing orientation of $\mu_6$ and $\gamma$, for which we can take   
\[w=a_2a_1a_3a_2a_4a_3\Delta_5(a_5u_5)^{-1}.\]
The surface obtained by cutting $N_{6,0}$ along $\mu_6\cup\gamma$ is homeomorphic to $S_{2,2}$ and it can be deduced from Theorem \ref{presS} that
$\Stab^+[\mu_6,\gamma]$ is generated by 
$b$, $a_1$, $a_2$, $a_3$, $a_4$,  $u_4^{-1}a_4^{-1}ba_4u_4$. 
 We have $h_{e_7}=b_2^{-1}$ and $\psi(h_{e_7})=b_2^{-1}$ commutes with 
$b=\psi_{v_2}(b)$ and $a_i=\psi_{v_2}(a_i)$. 
It can be checked that for $x\in\{w, u_4^{-1}a_4^{-1}ba_4u_4\}$ we have $x\in \Stab[\alpha_1]$. Since also $h_{e_7}\in\Stab[\alpha_1]$ and
$\psi(h_{e_7})=\psi_{v_1}(h_{e_7})$, thus $(\ast)$ holds for these $x$ by Lemma \ref{RelE1}.
\end{proof}
\section{Triangles.}\label{sec_triangles}
In this section we finish the proof of Theorem \ref{mainB} by showing that the map $\psi$ defined at the beginning of the previous section respects the relations corresponding to the triangles of $X$. For $f\in\cS_2(X)$ and 
$i=1,2,3$ let
$\nu_i=\nu_i(f)\in\cS_0(X)$, $\varepsilon_i=\varepsilon_i(f)\in\cS_1(X)$, 
$\widetilde{\varepsilon_i}=\widetilde{\varepsilon_i}(f)\in\cS_1(\widetilde{X})$ and 
$x_i=x_i(f)\in\Stab\,s(\nu_i)$ be as defined in Subsection \ref{orbits}.
We have to prove that
\[(\ast\ast)\quad\psi(h_{\varepsilon_1})\psi_{\nu_2}(x_2)\psi(h_{\varepsilon_2})\psi_{\nu_3}(x_3)\psi(h_{\varepsilon_3})^{-1}=\psi_{\nu_1}(x_1)\]
holds in $\G$. Note that we have not yet chosen the elements $x_i$.  Once any two of them are chosen, the third one is determined by the relation
$h_{\varepsilon_1}x_2h_{\varepsilon_2}x_3h^{-1}_{\varepsilon_3}=x_1$.
As explained in Subsection \ref{orbits}, for $j\in\{1,\dots,10\}$ and
for each permutation $\sigma\in\Sym_3$, $x_i(f_j^\sigma)$ are determined by $x_i(f_j)$. Moreover, it is easy to check that if $(\ast\ast)$ holds for $f_j$ then it also holds for $f_j^\sigma$. Therefore it suffices to prove the following.
\begin{lemma}
For $j\in\{1,\dots,10\}$ and $i\in\{1,2,3\}$ the elements $x_i(f_j)$ can be chosen is such a way that $(\ast\ast)$ is satisfied.
\end{lemma}
 \begin{proof}
Fix $f=f_j$ and let $[\gamma_1,\gamma_2,\gamma_3]=s(f)$.

{\bf Case 1.} Suppose that $\varepsilon_1=\varepsilon_2=\varepsilon_3$ and
$h_{\varepsilon_1}(\gamma_3)=\gamma_3$.
Then $s(\varepsilon_1)=[\gamma_1, \gamma_2]$ and $h_{\varepsilon_1}[\gamma_1,\gamma_3]=[\gamma_2,\gamma_3]$. Once 
$x_1$ is chosen such that $x_1[\gamma_1,\gamma_2]=[\gamma_1,\gamma_3]$, then we can take $x_2=x_1$ because $h_{\varepsilon_1}x_1[\gamma_1,\gamma_2]=[\gamma_2,\gamma_3]$, and $x_3$ is determined. The assumption of this case is satisfied for
$j\in\{1,4,7\}$. Indeed, for $j=1$ we have
$s(f)=[\alpha_1, \alpha_3, \alpha_5]$, $\varepsilon_1=\varepsilon_2=\varepsilon_3=e_2$, $h_{e_2}=a_2a_3a_1a_2$ and we take $x_1=a_4a_5a_3a_4$.
For $j\in\{4,7\}$ we have $\varepsilon_1=\varepsilon_2=\varepsilon_3=e_3$, $h_{e_3}=a_{g-1}^{-1}$ and $s(f_4)=[\mu_g, \mu_{g-1}, \mu_{g-2}]$,
$s(f_7)=[\mu_5, \mu_4, \gamma_{\{1,2,3\}}]$ ($g=5$). We take
$x_1(f_4)=a_{g-2}^{-1}$ and $x_1(f_7)=b^{-1}$. It is easy to check that in each case we have $x_3=x_1^{-1}$ and $(\ast\ast)$ is a consequence of the relations (A1, A2, A4).

{\bf Case 2.} Suppose that $\varepsilon_2=\varepsilon_3$, $h_{\varepsilon_3}=1$, 
$s(\varepsilon_3)=[\gamma_1,\gamma_3]$ and $h_{\varepsilon_1}(\gamma_3)=\gamma_3$.
Then $h_{\varepsilon_1}[\gamma_1,\gamma_3]=[\gamma_2,\gamma_3]$ and we can take $x_1=x_2=1$, $x_3=h_{\varepsilon_1}^{-1}$. In this case $(\ast\ast)$ is $\psi(h_{\varepsilon_1})=\psi_{\nu_3}(h_{\varepsilon_1})$.
The assumption of this case is satisfied for
$j\in\{2,5\}$. Indeed, for $j=2$ we have
$\varepsilon_1=e_2$, $\varepsilon_2=\varepsilon_3=e_1$, 
$s(f)=[\alpha_1,\alpha_3,\mu_g]$, $h_{e_2}=a_2a_3a_1a_2$. Since
$\psi(h_{e_2})=\psi_{v_2}(h_{e_2})$ thus $(\ast\ast)$ holds.
For $j=5$ we have
$\varepsilon_1=e_2$, $\varepsilon_2=\varepsilon_3=e_4$, 
$s(f)=[\alpha_1,\alpha_3,\xi]$ and $(\ast\ast)$ holds because
$\psi(h_{e_2})=\psi_{v_3}(h_{e_2})$.

{\bf Case 3.} Suppose that $\varepsilon_1=\varepsilon_3$, $h_{\varepsilon_1}=1$, 
$s(\varepsilon_2)=[\gamma_2,\gamma_3]$ and $h_{\varepsilon_2}(\gamma_1)=\gamma_1$.
Analogously as in Case 2 we can take
$x_2=x_3=1$, $x_1=h_{\varepsilon_2}$ and $(\ast\ast)$ becomes $\psi(h_{\varepsilon_2})=\psi_{\nu_1}(h_{\varepsilon_1})$.
The assumption of this case is satisfied for
$j\in\{3,8\}$. Indeed, for $j=3$ we have
$\varepsilon_1=\varepsilon_3=e_1$, $\varepsilon_2=e_3$,  
$s(f)=[\alpha_1,\mu_g,\mu_{g-1}]$, $h_{e_3}=a_{g-1}^{-1}$ and $(\ast\ast)$ holds
because
$\psi(h_{e_3})=\psi_{v_1}(h_{e_3})$.
For $j=8$ we have
$\varepsilon_1=\varepsilon_3=e_1$, $\varepsilon_2=e_7$,  
$s(f)=[\alpha_1,\mu_g,\gamma_{\{1,\dots,5\}}]$, $h_{e_7}=b_2^{-1}$ and $(\ast\ast)$ holds
because
$\psi(h_{e_7})=\psi_{v_1}(h_{e_7})$.

{\bf Case 4.} Suppose that $j\in\{6,9,10\}$. We have $h_{\varepsilon_i}\in\Stab[\xi]$ and $\psi(h_{\varepsilon_i})=\psi_{v_3}(h_{\varepsilon_i})$ for $i=1,2,3$. If we can choose $x_i\in\Stab[\xi]\cap\Stab\,s(\nu_i)$ such that 
$\psi_{\nu_i}(x_i)=\psi_{v_3}(x_i)$ then $(\ast\ast)$ will follow from the fact that $\psi_{v_3}$ is a homomorphism.

For $j=6$ we have 
$s(f)=[\mu_6,\mu_5,\beta]$ and $h_{\varepsilon_1}=a_5^{-1}$. It can be checked that for $x_1=a_4ba_3a_4a_2a_3a_1a_2$ we have $x_1[\mu_6,\alpha_1]=[\mu_6,\beta]$.
Since $a_5^{-1}[\mu_6,\beta]=[\mu_5,\beta]$ we can take $x_2=x_1$ and $x_3$ is determined. Clearly $x_1\in\Stab[\xi]\cap\Stab[\mu_6]$ and 
$\psi_{v_2}(x_1)=\psi_{v_3}(x_1)$. It follows that $x_3\in\Stab[\xi]\cap\Stab[\alpha_1]$ and $\psi_{v_1}(x_3)=\psi_{v_3}(x_3)$ by Lemma \ref{RelE4}.
 
For $j=9$ we have $s(f)=[\alpha_1,\mu_5,\beta]$. It can be checked that for $x_1=a_4ba_3a_4$ we have $x_1[\alpha_1,\alpha_3]=[\alpha_1,\beta]$. Since 
$s(\varepsilon_2)=[\mu_5,\beta]$ and $h_{\varepsilon_1}=1$, we can take $x_2=1$ and $x_3$ is determined. We have $x_3\in\Stab[\xi]$ and 
$\psi_{v_1}(x_3)=\psi_{v_3}(x_3)$ by Lemma \ref{RelE4}.

For $j=10$ we have $s(f)=[\alpha_1,\alpha_3,\gamma_{\{3,4,5,6\}}]$, $\varepsilon_1=\varepsilon_2=e_2$, $\varepsilon_3=e_6$ and $h_{e_2}=a_2a_3a_1a_2$, $h_{e_6}=a_2ca_1a_2$ . We take $x_1=1$ and $x_2=a_4a_5a_3a_4^2ba_3a_4$. It can be checked that
$h_{e_2}x_2[\alpha_1,\alpha_3]=[\alpha_3,\gamma_{\{3,4,5,6\}}]$. We have $\psi_{v_1}(x_2)=\psi_{v_3}(x_2)$,  $x_3\in\Stab[\xi]$ and $\psi_{v_1}(x_3)=\psi_{v_3}(x_3)$ by Lemma \ref{RelE4}.
\end{proof}
\appendix
\section{The presentation of $\M(S_{\rho,r})$.}
In this section we prove Theorem \ref{presS}. For $r=1$ the presentation given in Theorem \ref{presS} is essentially the same as the one given in \cite{Mat}, with additional generators and relations. In the case $r=2$ we follow the idea of \cite{LabPar}. 

We are going to use some basic facts about geometric representations of Artin groups (see \cite{LabPar} for details). Let $\Gamma_1$, $\Gamma_2$ be the Artin groups of types $A_5$ and $D_6$ respectively. 
For $i\ge 1$ and $j=1,2$ we have homomorphisms $\theta_{i,j}\colon\Gamma_j\to\M(S_{\rho,r})$
, such that $\theta_{i,1}$ maps the standard generators of $\Gamma_1$ on $b_{i-1}, a_{2i}, a_{2i+1}, a_{2i+2}, a_{2i+3}$ and $\theta_{i,2}$ maps the standard generators of $\Gamma_2$ on $b_{i-1}, a_{2i}, a_{2i+1}, a_{2i+2}, a_{2i+3}, b_i$. Let $\Delta(\Gamma_j)$ be the fundamental element of $\Gamma_j$ and $C_{i,j}=\theta_{i,j}(\Delta(\Gamma_j))$. Then we have
\begin{align*}
&C_{i,1}=b_{i-1}a_{2i}a_{2i+1}a_{2i+2}a_{2i+3}b_{i-1}a_{2i}a_{2i+1}a_{2i+2}b_{i-1}a_{2i}a_{2i+1}b_{i-1}a_{2i}b_{i-1}\\
&C^2_{i,1}=(b_{i-1}a_{2i}a_{2i+1}a_{2i+2}a_{2i+3})^6,\quad
C_{i,2}=(b_{i-1}a_{2i}a_{2i+1}a_{2i+2}a_{2i+3}b_i)^5
.
\end{align*}

\medskip

{\it Proof of Theorem \ref{presS}.} 
Let $S_1=S_{\rho,1}$ and $S_2=S_{\rho,2}$ for a fixed $\rho\ge 1$. We assume that $S_1$ is obtained from $S_2$ by gluing a disc along $\beta_\rho$ (see Figure \ref{fig_S}) and let $P$ be the center of the glued disc. We have the following short exact sequences, which are special cases of (\ref{Bir_es}) and (\ref{Cup_es}).
\begin{equation}\label{Bir_S}
1\to\pi_1(S_1,P)\to\M(S_1,P)\to\M(S_1)\to 1
\end{equation}
\begin{equation}\label{Cup_S}
1\to\lr{b_\rho}\to\M(S_2)\to\M(S_1,P)\to 1
\end{equation}
By \cite{Mat} $\M(S_1)$ admits a presentation with generators
$b_1$, $a_i$ for $i=1,\dots,2\rho$ and relations (A1--A6).
We add to this presentation the generators $b_j$ for
$j=0,2,\dots,\rho-1$ and relations (A7,A8). We need to show that (A8) are satisfied in $\M(S_1)$. Fix $i\ge 1$ and let $M_1$ and $M_2$ be  regular neighbourhoods of
$\beta_i\cup\beta_{i-1}\cup\alpha_{2i}\cup\dots\cup\alpha_{2i+3}$ and
$\beta_{i-1}\cup\alpha_{2i}\cup\dots\cup\alpha_{2i+3}$ respectively.
The boundary of $M_1$ consists of three connected components, one of which is  isotopic to $\beta_{i+1}$ and one bounds a disc. Let $\delta$ be the third component and note that the boundary of $M_2$ consists of two connected components isotopic to $\beta_{i+1}$ and $\delta$. By \cite[Proposition 2.12]{LabPar} we have
$C_{i,2}=T_\delta b_{i+1}^2$ and $C^2_{i,1}=T_\delta b_{i+1}$. The relation (A8) follows and Theorem \ref{presS} is proved for $r=1$.

We are going to show that $\M(S_1,P)$ admits a presentation with generators
$a_i$ for $i=1,\dots,2\rho+1$ and $b_j$ for $j=0,1,\dots,\rho-1$ and relations
(A1--A8) and if $\rho\ge 2$
\[\textrm{(A8a)}\quad(b_{\rho-2}a_{2\rho-2}a_{2\rho-1}a_{2\rho}a_{2\rho+1}b_{\rho-1})^5=(b_{\rho-2}a_{2\rho-2}a_{2\rho-1}a_{2\rho}a_{2\rho+1})^6.\]
To prove this we apply Lemma \ref{ext_pres} to the sequence (\ref{Bir_S}).
By gluing a punctured annulus along the boundary of $S_1$ we obtain an induced embedding $\M(S_1)\to\M(S_1,P)$, which is a splitting of (\ref{Bir_S}). Through this embedding we will identify the generators of $\M(S_1)$ with elements of $\M(S_1,P)$ which will be the cokernel generators, and the defining relations (A1--A8) of $\M(S_1)$ will be the cokernel relations in our presentation of $\M(S_1,P)$. For notational convenience we set
\[c_0=b_{\rho-1},\ c_i=a_{2\rho+2-i}\ \textrm{for\ }i=1,\dots,{2\rho+1},\quad d=b_{\rho-2}\ \textrm{if\ }\rho\ge 2.\]
The kernel is freely generated by
\[x_1=c_1c_0^{-1},\quad x_i=c_ix_{i-1}c_i^{-1}\ \textrm{for\ }i=2,\dots,2\rho.\]
These will be our kernel generators. We also add $c_1$ to the set of generators together with the relation
\[(0)\quad c_1=x_1c_0.\]
The following relations are consequences of (A1--A8).
\begin{align*}
&(H1)\ c_ic_j=c_jc_i\quad \textrm{for\ }1\le i<j-1\le 2\rho\\
&(H2)\ c_ic_{i+1}c_i=c_{i+1}c_ic_{i+1}\quad \textrm{for\ }1\le i\le 2\rho\\
&(H3)\ c_0c_i=c_ic_0\quad \textrm{for\ }i\neq 2\qquad
(H4)\ c_0c_2c_0=c_2c_0c_2\\
&(H5)\ dc_i=c_id\quad \textrm{for\ }i\neq 4\qquad
(H6)\ dc_4d=c_4dc_4
\end{align*}
Indeed, (H1,H2) are simply (A1,A2) rewritten in the symbols $c_i$, (H4, H6) involve only cokernel generators, and so they are consequences of the cokernel relations, because of the splitting, and so are (H3, H5) for $i\ne 1$. If $i=1$ then the last relations follow from (A1, A8).

Since $\M(S_1)$ is generated by $d$ and $c_i$ for
$i=0,2,\dots,2\rho$, we can use these cokernel generators to produce the conjugation relations. We write $c_i(x_j)$ instead of $c_ix_jc_i^{-1}$.
\begin{align*}
&(1)\ c_i(x_{i-1})=x_i,\quad (2)\ c_i(x_i)=x_ix_{i-1}^{-1}x_i,\quad
(3)\ c_i(x_j)=x_j,\quad\textrm{for\ }i\ge 2, j\notin\{i-1,i\},\\
&(4)\ c_0(x_1)=x_1,\quad (5)\ c_0(x_2)=x_1^{-1}x_2,\quad
(6)\ c_0(x_i)=c_ic_0(x_{i-1})\quad
\textrm{for\ }i>2;\\
&(7)\ d(x_i)=x_i\ \textrm{for\ }i=1,2,3,\quad (8)\ d(x_4)=x_1^{-1}x_2x_3^{-1}x_4,\\
&(9)\ d(x_j)=c_jd(x_{j-1})\quad
\textrm{for\ }j>4.
\end{align*}
We are going to show that (2--7,9) are consequences of (0,1) and (H1--H6). 
(6) follows from (1) and (H3):
\[c_0(x_i)=c_0c_i(x_{i-1})=c_ic_0(x_{i-1}).\]
Analogously (9) follows from (1) and (H5).
(7) follows from (0, 1, H5), and (3) follows from (0, 1, H1) if $j<i-1$,  for $j=i+1$ we have
\[
c_i(x_{i+1})=c_ic_{i+1}c_i(x_{i-1})=
c_{i+1}c_ic_{i+1}(x_{i-1})=x_{i+1}\]
and for $j=i+1+k$ by induction
\[
c_i(x_{i+1+k})=c_ic_{i+1+k}(x_{i+k})=
c_{i+1+k}c_i(x_{i+k})=x_{i+1+k}.\]
For $i=2$ (2)  is equivalent to
\begin{align*}
&c_2^2c_1c_0^{-1}c_2^{-2}=c_2c_1\underline{c_0^{-1}c_2^{-1}c_0c_1^{-1}c_2c_1}c_0^{-1}c_2^{-1}\iff
c_2c_1c_0^{-1}c_2^{-1}=c_1c_2\underline{c_0^{-1}c_1}c_2^{-1}c_0^{-1}\iff\\
&c_2c_1c_0^{-1}c_2^{-1}=\underline{c_1c_2c_1}c_0^{-1}c_2^{-1}c_0^{-1}\iff
c_0^{-1}c_2^{-1}=c_2c_0^{-1}c_2^{-1}c_0^{-1}\iff c_2c_0c_2=c_0c_2c_0
\end{align*}
and for $i>2$
\begin{align*}
&c_ix_ic_i^{-1}=c_ic_{i-1}x_{i-2}\underline{c_{i-1}^{-1}c_i^{-1}c_{i-1}}x_{i-2}^{-1}\underline{c_{i-1}^{-1}
c_ic_{i-1}}x_{i-2}c_{i-1}^{-1}c_i^{-1}\iff\\
&c_ic_{i-1}x_{i-2}c_{i-1}^{-1}c_i^{-1}=c_{i-1}x_{i-2}c_ic_{i-1}^{-1}\underline{c_i^{-1}x_{i-2}^{-1}
c_i}c_{i-1}c_i^{-1}x_{i-2}c_{i-1}^{-1}\iff\\
&\underline{c_{i-1}^{-1}c_ic_{i-1}}x_{i-2}\underline{c_{i-1}^{-1}c_i^{-1}c_{i-1}}=\underline{x_{i-2}c_i}c_{i-1}^{-1}x_{i-2}^{-1}
c_{i-1}\underline{c_i^{-1}x_{i-2}}\iff\\
&c_i\underline{c_{i-1}c_{i}^{-1}x_{i-2}c_ic_{i-1}^{-1}}c_i^{-1}=c_ix_{i-2}c_{i-1}^{-1}x_{i-2}^{-1}
c_{i-1}x_{i-2}c_i^{-1}\iff\\
&x_{i-1}=x_{i-2}c_{i-1}^{-1}x_{i-2}^{-1}
c_{i-1}x_{i-2}\iff\\
&c_{i-1}x_{i-1}c_{i-1}^{-1}=c_{i-1}x_{i-2}c_{i-1}^{-1}x_{i-2}^{-1}
c_{i-1}x_{i-2}c_{i-1}^{-1}=x_{i-1}x_{i-2}^{-1}x_{i-1}.
\end{align*}
and we are done by induction.
(4) is equivalent to $c_0c_1=c_1c_0$ and (5) to
\begin{align*}
&c_0c_2c_1\underline{c_0^{-1}c_2^{-1}c_0^{-1}}=c_0c_1^{-1}c_2c_1c_0^{-1}c_2^{-1}\iff
c_2c_1c_2^{-1}=c_1^{-1}c_2c_1.
\end{align*}
Finally we are going to show that (8) follows from (0,1), (H1--H6) and (A8a).
Let $C_j=C_{\rho-1,j}$ for $j=1,2$ so that we have
\begin{align*}
&C_1=dc_4c_3c_2c_1dc_4c_3c_2dc_4c_3dc_4d,\quad
C^2_1=(dc_4c_3c_2c_1)^6,\quad
C_2=(dc_4c_3c_2c_1c_0)^5
.
\end{align*}
We leave it as an exercise for the reader to check that by using (0,1) and (H1--H6) the relation (8) can be rewritten as
\[C_1=c_0c_2c_1c_3c_2c_0c_4c_3c_2c_1dc_4c_3c_2c_0,\]  
and by (H1--H6)  we have
\[C_2=C_1c_0c_2c_1c_3c_2c_0c_4c_3c_2c_1dc_4c_3c_2c_0.\]
Thus we have obtained the relation $C_2=C_1^2$, which is exactly (A8a).    

We can drop the generators $x_i$ and relations (0--9) to obtain a presentation of $\M(S_1,P)$ with generators 
$a_i$ for $i=1,\dots,2\rho+1$ and $b_j$ for $j=0,1,\dots,\rho-1$ and relations
(A1--A8,A8a). 

Now we will obtain a presentation of $\M(S_2)$ by applying Lemma \ref{ext_pres} to the sequence (\ref{Cup_S}). We take the generators of $\M(S_1,P)$ as cokernel generators and $b_\rho$ as kernel generator. The relations (A1--A8) are satisfied in $\M(S_2)$ and the cokernel relation corresponding to (A8a) is $C_2C_1^{-2}=b_\rho$ which gives (A8) for $i+1=\rho$. The conjugation relations are \[(\ast)\quad b_\rho y=y b_\rho\] for every cokernel generator $y$. 
It suffices to consider $y=a_i$ for $i=1,\dots,2\rho+1$ and $y=b_1$ if $\rho\ge 2$.
If $\rho=1$ then $(\ast)$ follows from (A3), so we suppose that $\rho\ge 2$. Since 
$b_\rho=C_2C_1^{-2}$ and $C_2$, $C_1^2$ are central in $\theta_{\rho-1,2}(\Gamma_2)$ and $\theta_{\rho-1,1}(\Gamma_1)$ respectively, $(\ast)$ is a consequence of (H1--H6) for $y=c_i=a_{2\rho+2-i}$, $i=1,2,3,4$ and $y=d=b_{\rho-2}$. In particular  $b_\rho$ commutes with $a_1=b_0$ if $\rho=2$ and with $b_1$ if $\rho=3$.
If $\rho\ge 3$ then it follows from (A1--A8) that $b_\rho$ commutes with $a_i$ for $i\le 2\rho-4$  and  $b_1$ if $\rho\ge 4$. Finally (A9a, A9b) imply that it also commutes with $a_{2\rho-3}$ if $\rho\ge 3$ and $b_1$ if $\rho=2$.
Since all conjugation relations are  consequences of (A1--A9), $\M(S_2)$ admits the presentation from Theorem \ref{presS}.
\hfill{$\Box$}
%

%
\end{document}